\documentclass{article}
\usepackage{amssymb}
\usepackage{amsthm}
\usepackage[margin=1.25in]{geometry}
\usepackage{color}
\usepackage{caption}
\usepackage{subcaption}

\usepackage{mathtools}
\numberwithin{equation}{section}
\mathtoolsset{showonlyrefs=true}

\usepackage{amsfonts}
\usepackage{graphicx}%
\providecommand{\U}[1]{\protect\rule{.1in}{.1in}}
\newtheorem{theorem}{Theorem}[section]

\newtheorem{corollary}[theorem]{Corollary}

\newtheorem{lemma}[theorem]{Lemma}

\newtheorem{proposition}[theorem]{Proposition}
\newtheorem{remark}[theorem]{Remark}

\newcommand{\e}{\varepsilon}
\newcommand{\dx}{\, dx}
\newcommand{\dt}{\, dt}
\newcommand{\ds}{\, ds}
\newcommand{\dy}{\, dy}
\newcommand{\I}{\mathcal{I}_\Omega}
\newcommand{\V}{\mathfrak{v}}
\renewcommand{\r}{\mathfrak{r}}

\title{Second-Order $\Gamma$-limit for the Cahn--Hilliard Functional}
\author{Giovanni Leoni\\
	Carnegie Mellon University \\
	Pittsburgh, PA, USA
	\and Ryan Murray\\
	Carnegie Mellon University \\
	Pittsburgh, PA, USA}

\begin{document}
\maketitle

\begin{abstract}
The goal of this paper is to solve a long standing open problem, namely, the asymptotic development of order $2$ by $\Gamma$-convergence of the mass-constrained Cahn--Hilliard functional.
\end{abstract}



\section{Introduction}

The goal of this paper is to study the asymptotic development by $\Gamma$-convergence of order $2$ of the Modica--Mortola or Cahn--Hilliard functional  (see \cite{GurtinMatano,Modica87,Sternberg88})
\begin{equation} 
F_\e(u) := \int_\Omega W(u) + \e^2 |\nabla u|^2 \dx,\qquad u \in H^1(\Omega), \label{F0Definition}
\end{equation}
subject to the mass constraint
\begin{equation} 
\int_\Omega u \dx = m. \label{massConstraintEquation}
\end{equation}
Here $\Omega \subset \mathbb{R}^n$ is an open, bounded set and $W$ is a double-well potential.

The notion of asymptotic development by $\Gamma$-convergence was introduced by Anzellotti and Baldo \cite{AnzellottiBaldo}. To be precise, given a metric space $X$ and a family of functions $\mathcal{F}_\e : X \to \overline{ \mathbb{R}}, \hspace{1mm}\e > 0$, we say that an \emph{asymptotic development} of \emph{order k}
\begin{equation}\label{asymptoticDevelopmentDefinition}
\mathcal{F}_\e = \mathcal{F}^{(0)} + \e\mathcal{F}^{(1)} + \dots + \e^k\mathcal{F}^{(k)} + o(\e^k)
\end{equation}
holds if there exist functions $\mathcal{F}^{(i)}:X\to\overline{\mathbb{R}}$, $i = 0,1,\dots,k$, such that the functions
\begin{equation} \label{higherOrderFunctionalDefinition}
\mathcal{F}^{(i)}_\e := \frac{\mathcal{F}^{(i-1)}_\e-\inf_X \mathcal{F}^{(i-1)}}{\e}
\end{equation}
are well-defined and
\begin{equation} \label{higherOrderGammaConvergenceDefinition}
\mathcal{F}^{(i)}_\e \xrightarrow{\Gamma} \mathcal{F}^{(i)},
\end{equation}
where $\mathcal{F}^{(0)}_\e := \mathcal{F}_\e$ and $\overline{\mathbb{R}}$ is the extended real line. Let
\begin{equation} \label{minimizingSetsDefinition}
\mathcal{U}_i := \{\text{minimizers of } \mathcal{F}^{(i)}\}.
\end{equation}
It can be shown that 
\begin{equation}\label{minimizingSetsNested1}
\mathcal{F}^{(i)} \equiv \infty \text{ in } X \backslash \mathcal{U}_{i-1},
\end{equation}
and that
\begin{equation} \label{minimizingSetsNested2}
\{\text{limits of minimizers of }\mathcal{F}_{\e_m}\} \subset \mathcal{U}_k \subset \dots \subset \mathcal{U}_0,
\end{equation}
with
\[
\inf \mathcal{F}_{\e_m} = \inf \mathcal{F}^{(0)} + \e_m \inf \mathcal{F}^{(1)} + \dots + \e_m^k \inf \mathcal{F}^{(k)} + o(\e_m^k)
\]
for every sequence $\e_m \to 0^+$, provided $\inf \mathcal{F}^{(i)} < \infty$ for all $i =0,\dots,k$.

Simple examples show that each of the inclusions in \eqref{minimizingSetsNested2} may be strict (see \cite{AnzellottiBaldo}). Thus asymptotic development by $\Gamma$-convergence provides a selection criteria for minimizers of $\mathcal{F}^{(0)}$. Some other works that describe asymptotic development via $\Gamma$-convergence include \cite{BraidesTruskinovsky}, \cite{FocardiReview}.

The first example of asymptotic development by $\Gamma$-convergence of order 2 for functionals of the type \eqref{F0Definition} was studied by Anzellotti and Baldo in \cite{AnzellottiBaldo}, who considered the case in which $n=1$, the wells of $W$ are not points but non-degenerate intervals and the mass constraint \eqref{massConstraintEquation} is replaced by a Dirichlet condition. Subsequently Anzellotti, Baldo and Orlandi \cite{Orlandi96} studied \eqref{F0Definition} in arbitrary dimension, in the case in which $W$ has only one well ($W(s) = s^2$) and again with Dirichlet boundary conditions in place of \eqref{massConstraintEquation}.

The problem of the asymptotic development of order $2$ for the Cahn--Hilliard functional \eqref{F0Definition} with $W$ a double-well potential  has remained an open problem, except when $n=1$. Indeed, in the one-dimensional case and for sufficiently smooth $W$, one can show that $\mathcal{F}^{(2)} = 0$. This can be deduced from the work of Carr, Gurtin and Slemrod \cite{CarrGurtinSlemrod}, Theorem 8.1, and from the recent paper \cite{Bellettini2013} of Bellettini, Nayam and Novaga, who gave a very precise higher-order asymptotic estimate for $\mathcal{F}_\e(u_\e)$, where $\{u_\e\} \subset H^1(\mathbb{T})$ is any sequence converging to $u \in BV(\mathbb{T};\{-1,1\})$ in $L^1(\mathbb{T})$, where $\mathbb{T}$ is the one-dimensional torus and $-1,1$ are the wells of $W$.

To our knowledge, the only result related to the second-order asymptotic development of \eqref{F0Definition} in the case $n\geq 2$ for \eqref{F0Definition}, \eqref{massConstraintEquation} has recently been obtained by the first author in collaboration with Dal Maso and Fonseca  in \cite{DalMasoFonsecaLeoni}. For a double-well potential satisfying
\begin{equation} \label{WEven}
W(s) = W(-s)
\end{equation}
 for all $s\in \mathbb{R}$ and
\begin{equation} \label{DMFLWAtWells}
W(s) = C|1-s|^{1+q}
\end{equation}
near $s=1$, for some $q \in (0,1)$, and \emph{under the additional assumption } that
\begin{equation} \label{DMFLDirichlet}
u=1 \text{ on } \partial \Omega,
\end{equation}
in addition to \eqref{massConstraintEquation}, it was shown that $\mathcal{F}^{(2)} = 0$. More generally, this was proved in the case in which $\e^2 \int_\Omega |\nabla u|^2 \dx$ is replaced by $\e^2 \int_\Omega \Phi^2(\nabla u) \dx$, with $\Phi: \mathbb{R}^n \to [0,\infty)$ an arbitrary norm. The Dirichlet condition \eqref{DMFLDirichlet} played a crucial role in the proof in \cite{DalMasoFonsecaLeoni} since it permitted the use of classical symmetrization techniques (see \cite{Kawohl}, \cite{KesavanBook}) in $H^1_0(\Omega)$ to reduce the problem to the radial case. Moreover, the behavior of $W$ near the wells (see \eqref{DMFLWAtWells}) did not allow for $C^2$ potentials $W$. The work of \cite{DalMasoFonsecaLeoni} left open several important questions, namely the characterization of $\mathcal{F}^{(2)}$ when
\begin{itemize}
\item the Dirichlet condition \eqref{DMFLDirichlet} is not imposed,
\item $W$ is of class $C^2$,
\item $W$ is not even.
\end{itemize}
In this paper we address all of these questions. In particular, we show that in general $\mathcal{F}^{(2)} \neq 0$ if $W$ is even and of class $C^2$, or if $W$ is not even.

Here we take $X:= L^1(\Omega)$ and define
\begin{equation}\label{formalFunctionalDefinition}
\mathcal{F}_\e(u) := \begin{cases}
F_\e(u) &\text{ if } u \in H^1(\Omega) \text{ and } (\ref{massConstraintEquation})\text{  holds},\\
\infty &\text{ otherwise in } L^1(\Omega).
\end{cases}
\end{equation}

The $\Gamma$-limit $\mathcal{F}^{(1)}$ of order $1$ (see \eqref{higherOrderFunctionalDefinition} and \eqref{higherOrderGammaConvergenceDefinition}) has been established by Carr, Gurtin and Slemrod \cite{CarrGurtinSlemrod} for $n=1$ and by Modica \cite{Modica87} and Sternberg \cite{Sternberg88} for $n \geq 2$ (see also \cite{Gurtin}, \cite{ModicaMortola}), and is known to be, under appropriate assumptions on $\Omega$ and $W$,
\begin{equation} \label{firstOrderFormalDefinition}
\mathcal{F}^{(1)}(u) := \begin{cases}
2c_W\operatorname*{P}(\{u=a\};\Omega) &\text{ if } u \in BV(\Omega;\{a,b\}) \text{ and \eqref{massConstraintEquation} holds},\\
\infty &\text{ otherwise in } L^1(\Omega),
\end{cases}
\end{equation}
where $\operatorname*{P}(\cdot;\Omega)$ is the perimeter in $\Omega$ (see \cite{AmbrosioFuscoPallara,EvansGariepy,Ziemer}), $a,b$ are the wells of $W$ and the constant $c_W$ is given by
\begin{equation} \label{c0Definition}
c_W := \int_a^b W^{1/2}(s)  \ds.
\end{equation}

In view of \eqref{minimizingSetsNested1}, in order to characterize the $\Gamma$-limit of order $2$, $\mathcal{F}^{(2)}$ (see \eqref{higherOrderFunctionalDefinition} \eqref{higherOrderGammaConvergenceDefinition}), it is important to understand the family $\mathcal{U}_1$ of minimizers of the functional $\mathcal{F}^{(1)}$ defined in \eqref{firstOrderFormalDefinition}. Observe that $u$ belongs to $\mathcal{U}_1$ if and only if $u \in BV(\Omega;\{a,b\})$ and the set $\{u=a\}$ is a solution of the classical \emph{partition problem}, namely, if it solves
\begin{equation}\label{partitionProblemDefinition}
\min \{\operatorname*{P}(E;\Omega): \: E\subset \Omega \text{ Borel}, \: \mathcal{L}^n(E) = \V_m\},
\end{equation}
where
\begin{equation}\label{partitionProblemMassConstraint}
\V_m := \frac{b\mathcal{L}^n(\Omega) - m}{b-a} .
\end{equation}

The properties of minimizers of \eqref{partitionProblemDefinition} have been studied by Gr\"uter \cite{GruterBoundaryRegularity} (see also \cite{GMT83,MaggiBook,SternbergZumbrunIsoPer}), who showed that when $\Omega$ is bounded and of class $C^2$, minimizers $E$ of \eqref{partitionProblemDefinition} exist, have constant generalized mean curvature $\kappa_E$, intersect the boundary of $\Omega$ orthogonally, and their singular set is empty if $n\leq 7$, and has dimension of at most $n-8$ if $n \geq 8$. Here and in what follows we use the convention that \emph{$\kappa_E$ is the average of the principal curvatures taken with respect to the outward unit normal to $\partial E$}.

A crucial hypothesis in our results  is that the \emph{isoperimetric function} or \emph{isoperimetric profile} (\cite{Ros}), given by
\begin{equation} \label{isoFunctionDefinition}
\I(\V) := \inf \{ \operatorname*{P}(E;\Omega) : E \subset \Omega \text{ Borel, } \mathcal{L}^n(E) = \V\}, \quad \V \in [0,\mathcal{L}^n(\Omega)] ,
\end{equation}
admits a Taylor expansion of order $2$ at the value $\V_m$ in \eqref{partitionProblemMassConstraint}. In particular the differentiability of $\I$ at $\V_m$ implies that (see \cite{MaggiBook})
\begin{equation} \label{isoPerDerivative}
\I'(\V_m) = (n-1)\kappa_E
\end{equation}
for every minimizer $E$ of \eqref{isoFunctionDefinition} at $\V = \V_m$. Hence, differentiability of $\I$ must fail whenever the mean curvature of minimizers of the partition problem \eqref{partitionProblemDefinition} is not uniquely determined. For example, if $\Omega$ is a square in $\mathbb{R}^2$, it can be shown that there exists a value of $\V_m$ for which there are two minimizers of \eqref{partitionProblemDefinition}, one being a line segment and the other being an arc of a circle. 

\begin{figure}
\centering
\begin{subfigure}{.5\textwidth}
  \centering
  \includegraphics[width = .8\linewidth]{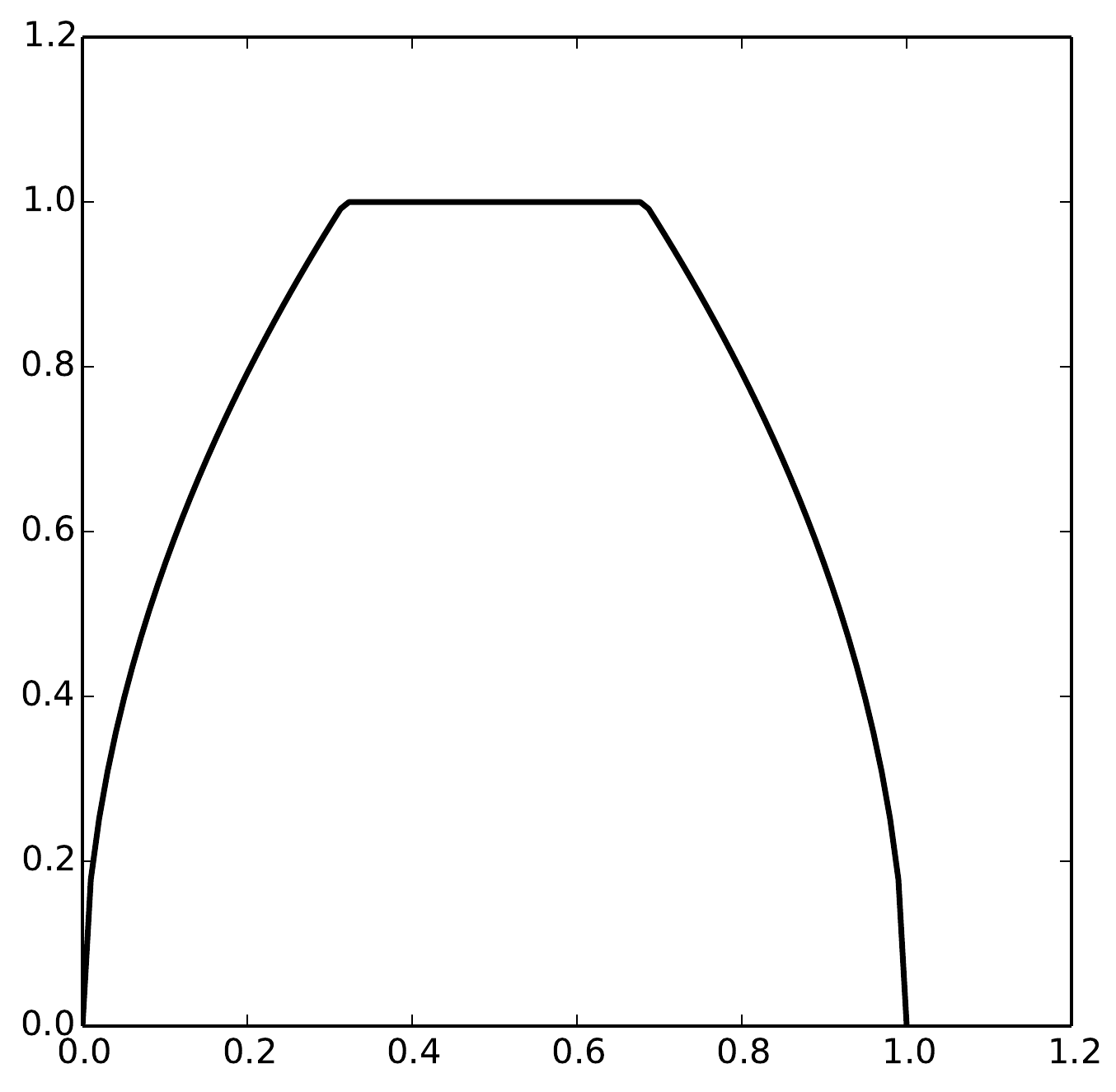}
  \label{fig:sub1}
\end{subfigure}%
\begin{subfigure}{.5\textwidth}
  \centering
  \includegraphics[width = .7\linewidth]{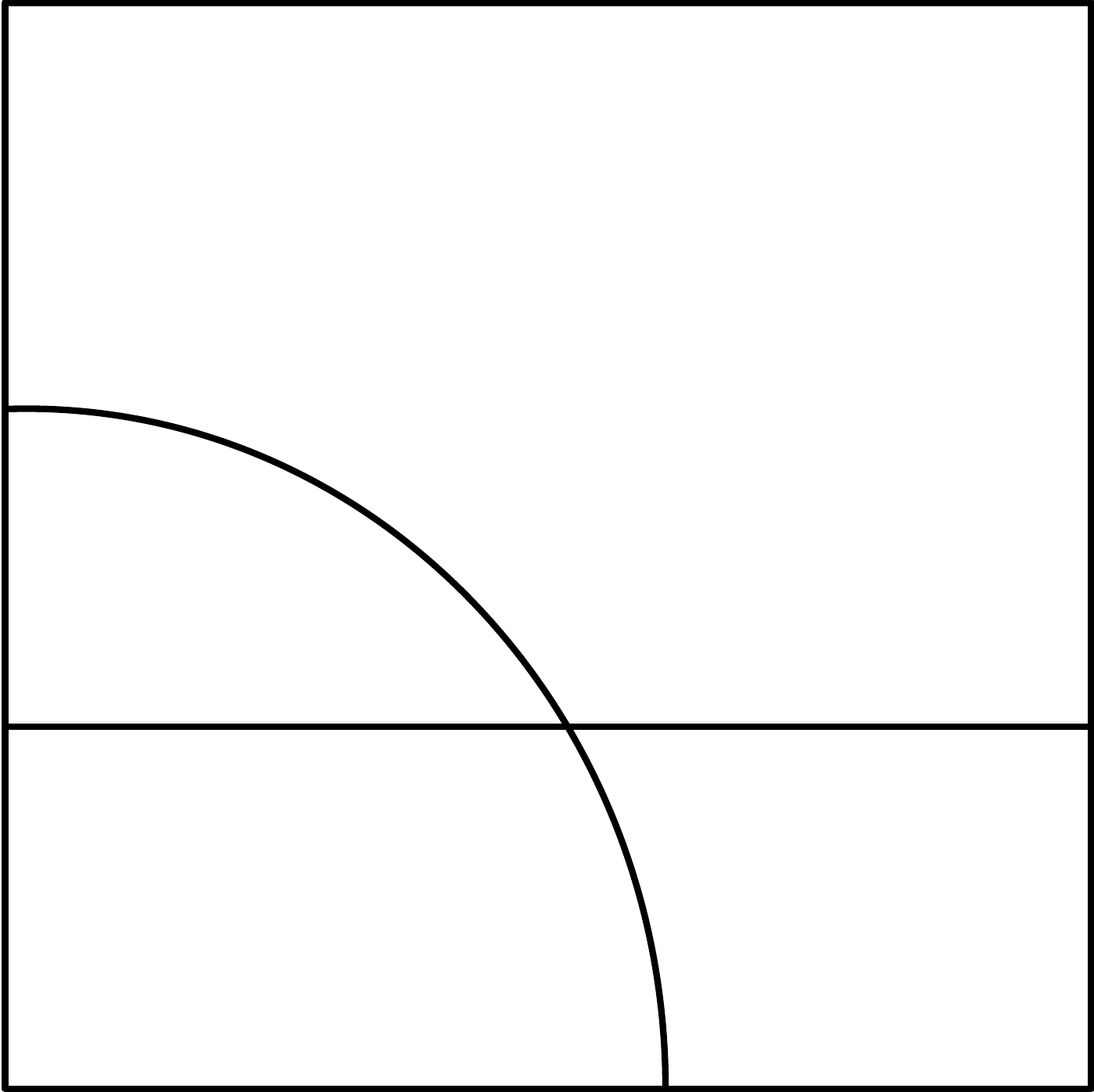}
  \label{fig:sub2}
\end{subfigure}
\caption{$\I$ for the domain $\Omega = Q_2$, the cube in $\mathbb{R}^2$. When $\I$ is not differentiable there are two competing sets minimizing the perimeter, as shown.}
\label{fig:test}
\end{figure}

We observe that, as $\I$ is semi-concave if $\Omega$ is sufficiently smooth \cite{BavardPansu} or convex \cite{SternbergZumbrunIsoPer}, a Taylor expansion of order $2$ holds for $\mathcal{L}^1$ a.e. $\V$, or equivalently for $\mathcal{L}^1$ a.e. mass $m$ in \eqref{massConstraintEquation}. Under this assumption on $\I$ and under other technical hypotheses on $\Omega,m$ and $W$ (see Section 2) we will show that if $W$ is quadratic near the wells $a, b$ then the following theorem holds.
\begin{theorem}\label{mainThm1}
Assume that $\Omega,m,W$ satisfy hypotheses \eqref{domainAssumptions}-\eqref{WGurtin_Assumption} with $q =1$. Then
\begin{equation}\label{mainThmEqn1}
\mathcal{F}^{(2)}(u) = \frac{2c_W^2(n-1)^2}{W''(a)(b-a)^2}\kappa_u^2 + 2(c_{\operatorname*{sym}}+c_W\tau_u)(n-1)\kappa_u\operatorname*{P}(\{u=a\};\Omega) 
\end{equation}
if $ u \in \mathcal{U}_1$ and $\mathcal{F}^{(2)}(u) = \infty$ otherwise in $L^1(\Omega)$.
\end{theorem}
Here $\kappa_u$ is the constant mean curvature of the set $\{u=a\}$,
\begin{equation}\label{c1Definition}
c_{\operatorname*{sym}} := \int_\mathbb{R} W(z(t))t  \dt,
\end{equation}
where $z$ is the solution to the Cauchy problem
\begin{equation} \label{profileCauchyProblem}
\begin{cases} z'(t) = \sqrt{W(z(t))} \quad &\text{ for } t \in \mathbb{R}, \\ z(0) = c,\quad &z(t) \in [a,b],\end{cases}
\end{equation}
with $c$ being the central zero of $W'$ (see \eqref{W_Number_Zeros}), and $\tau_u \in \mathbb{R}$ is a constant such that
\begin{equation}\label{deltaUDefinition}
\operatorname*{P}(\{u=a\};\Omega)\int_\mathbb{R} z(t - \tau_u) - \operatorname*{sgn}\nolimits_{a,b}(t)  \dt = \frac{2c_W(n-1)}{W''(a)(b-a)}\kappa_u,
\end{equation}
where
\begin{equation} \label{sgnabDefinition}
\operatorname*{sgn}\nolimits_{a,b}(t) := \begin{cases}
a &\text{ if } t \leq 0 ,\\
b &\text{ if } t > 0.
\end{cases}
\end{equation}

We note that when $W$ is quadratic near the wells, that is, when $q=1$ in \eqref{WPrimeLimits}, then the solution of the Cauchy problem \eqref{profileCauchyProblem} approaches $a$ and $b$ as $t \to -\infty$ and $\infty$ respectively, while when $W$ is subquadratic near the wells, that is, when $q < 1$ in \eqref{WPrimeLimits}, then the solution reaches $a$ and $b$ in finite time. \emph{This property plays a crucial role in our results, and helps explain why the two cases are different.}

We observe that, in view of \eqref{isoPerDerivative}, the quantities $\operatorname*{P}(\{u=a\};\Omega)$, $\kappa_u$ and $\tau_u$ are \emph{uniquely determined} by $\V_m$, $\I(\V_m)$ and $\I'(\V_m)$ for $u\in \mathcal{U}_1$.

Without assuming the differentiability of the isoperimetric function $\I$ at $\V_m$ one can only conclude that $(n-1)\kappa_u\in [ (\I)^\prime_{-}(\V_m),(\I)^\prime_{+}(\V_m)]$, where $(\I)^\prime_{-}$, $(\I)^\prime_{+}$ are the left and right derivatives of $\I$, which must exist as $\I$ is semi-concave \cite{BavardPansu}. We conjecture that \eqref{mainThmEqn1} continues to hold even in this case, but we have not been able to prove it. One potential avenue of investigation involves studying $L^1$ isolated families of perimeter minimizers where the mean curvature is unique. While this could potentially remove the issue of differentiability it does not remove the technical necessity of a higher-order Taylor expansion of $\I$ at $\V_m$.

If such a conjecture holds then \eqref{mainThmEqn1} would provide an additional selection criterion among minimizers of $\mathcal{U}_1$. In particular, when $W$ is symmetric about $\frac{a+b}{2}$ then surfaces with larger magnitude mean curvature are energetically favored (see Corollary \ref{cor:symmetricMainResult} below).

\medskip

We can offer a heuristic explanation for the terms in \eqref{mainThmEqn1}. Critical points $u_\e$ of \eqref{F0Definition} subject to \eqref{massConstraintEquation} satisfy the Neumann problem
\begin{equation}
\begin{cases}
2\e \Delta u_\e = \frac{1}{\e}  W'(u_\e) + \Lambda_\e &\text{ in } \Omega, \\
\frac{\partial u_\e}{\partial \nu} = 0 &\text{ on } \partial \Omega,
\end{cases}
\end{equation}
where $\nu$ is the outward unit normal to $\partial \Omega$ and $\Lambda_\e$ is a Lagrange multiplier that accounts for the constraint \eqref{massConstraintEquation}. In \cite{LuckhausModica}, Luckhaus and Modica proved that if $0<a<b$ and $\{u_\e\}$ is a sequence of non-negative minimizers of \eqref{F0Definition}, \eqref{massConstraintEquation}, uniformly bounded in $L^\infty(\Omega)$ and converging in $L^1(\Omega)$ to a minimizer of $\mathcal{F}^{(1)}$, then
\begin{equation} \label{def:lambdaNot}
\Lambda_\e \to \Lambda_u := \frac{2c_W(n-1) }{b-a}\kappa_u.
\end{equation}

Thus the first term in equation \eqref{mainThmEqn1} can be written as $\frac{\Lambda_u^2}{2W''(a)}$. Our proofs suggest (see \eqref{eqn:RecoverySequenceDefinition}) that minimizers $u_\e$ of the energy $E_\e$ will in fact be of the form
\begin{equation} \label{Eqn:MinimizerHeuristic}
u_\e(x) \approx z\left(\frac{d(x,\{u = a\}) -\e \tau_u}{\e}\right) - \frac{\Lambda_u \e}{W''(a)}.
\end{equation}
It turns out that the first term in equation \eqref{mainThmEqn1} is linked to a small vertical shift in the bulk values of minimizers, namely the second term in \eqref{Eqn:MinimizerHeuristic}. The $\tau_u$ term in \eqref{mainThmEqn1} is caused by the shift inside $z$ in the first term of \eqref{Eqn:MinimizerHeuristic}, which essentially pushes the transition layer ``outward'' along curved surfaces. We note that the horizontal shift caused by $\tau_u$ and the vertical shift in the bulk must be in some sense balanced so that the mass constraint is satisfied.

The term involving $c_{\operatorname*{sym}}$ may be thought of as a penalty for directional asymmetry. If the profiles are symmetric this term disappears entirely. This term is of order $\e$ for any $q$ that we consider.

On the other hand, if $W$ has subquadratic growth near the wells then the following theorem holds:
\begin{theorem}\label{mainThm2}
Assume that $\Omega,m,W$ satisfy hypotheses \eqref{domainAssumptions}-\eqref{WGurtin_Assumption} with $q \in (0,1)$. Then
\begin{equation}\label{mainThmEqn2}
\mathcal{F}^{(2)}(u) = \begin{cases}
2(c_{\operatorname*{sym}}+c_W\tau_u)(n-1)\kappa_u\operatorname*{P}(\{u=a\};\Omega) &\text{ if } u \in \mathcal{U}_1, \\
\infty &\text{ otherwise in } L^1(\Omega).
\end{cases}
\end{equation}
\end{theorem}
Here now $\tau_u$ is a constant such that
\begin{equation}\label{deltaUDefinitionCase2}
\int_\mathbb{R} z(t - \tau_u) - \operatorname*{sgn}\nolimits_{a,b}(t)  \dt = 0.
\end{equation}
Note that \eqref{mainThmEqn2} and \eqref{deltaUDefinitionCase2} correspond to the case $W''(a) = \infty$ in \eqref{mainThmEqn1} and \eqref{deltaUDefinition} respectively.

To prove \eqref{mainThmEqn1} and \eqref{mainThmEqn2} we follow the approach of \cite{DalMasoFonsecaLeoni}, namely we use rearrangement to reduce the problem to a one-dimensional one. However, since we are not imposing boundary conditions \eqref{DMFLDirichlet} we cannot use standard symmetrization techniques in $H_0^1(\Omega)$ (see, e.g.,  \cite{DalMasoFonsecaLeoni,Kawohl,LeoniBook}). Thus we implement a different type of rearrangement technique \cite{Cianchi1996,CianchiPick2009}, which makes direct use of the \emph{isoperimetric function } \eqref{isoFunctionDefinition}.

In particular, if $W$ is symmetric about $(b+a)/2$, then the function $z$ in \eqref{profileCauchyProblem} is symmetric, and so the constants $c_{\operatorname*{sym}}$ and $\tau_u$ simplify to give the following:
\begin{corollary} \label{cor:symmetricMainResult}
Suppose that, additionally, $W$ is symmetric about $(b+a)/2$. Then for $u\in \mathcal{U}_1$ we have that
\[
\mathcal{F}^{(2)}(u) = \begin{cases}
-\frac{2c_W^2(n-1)^2}{W''(a)(b-a)^2}\kappa_u^2 
 & \text{ if } q = 1,\\
0 &\text{ if } q<1.
\end{cases}
\]
\end{corollary}
Thus in the case $q \in (0,1)$, and with isotropic energy, we recover the result of \cite{DalMasoFonsecaLeoni} \emph{without the additional Dirichlet boundary condition} \eqref{DMFLDirichlet}. 

We conjecture that in the case $q \in (0,1)$ with symmetric potential, to obtain a nonzero asymptotic development of order two, one should replace the functionals $\mathcal{F}_\e^{(2)}$ in \eqref{higherOrderFunctionalDefinition} with  the family of functionals
\[
\mathcal{F}_{\e,q}^{(2)} := \frac{\mathcal{F}_\e^{(1)} - \min \mathcal{F}^{(1)}}{\e^{1/q}}.
\]
We have not been able to characterize the $\Gamma$-limit of $\mathcal{F}_{\e,q}^{(2)}$.

\begin{remark}
A straightforward calculation shows that in the case of the Cahn--Hilliard potential $W(s) = \frac12(1-s^2)^2$ the second-order $\Gamma$-limit takes the form
\[
\mathcal{F}^{(2)}(u) = -\frac{(n-1)^2}{9}\kappa_u^2.
\]
\end{remark}


Besides their intrinsic interest, Theorems \ref{mainThm1} and \ref{mainThm2} have important applications in the study of the speed of motion of the associated gradient flow in dimension $n\ge2$. Indeed the asymptotic development of \eqref{F0Definition} in one dimension has been utilized by many authors to establish the slow motion of solutions of the gradient flows associated with \eqref{F0Definition} in different function spaces. We recall that the gradient flow associated with $F_\e$ in $L^2(\Omega)$ without the mass constraint \eqref{massConstraintEquation} is the \emph{Allen--Cahn equation}
\begin{equation}\label{AllenCahn}
\partial_t u = \e^2 \Delta u - W'(u),
\end{equation}
while mass-constrained gradient flows in $L^2(\Omega)$ and $H^{-1}(\Omega)$ of $F_\e$ are, respectively, the \emph{non-local Allen--Cahn equation}
\begin{equation}\label{NLAllenCahn}
\partial_t u = \e^2 \Delta u - W'(u) + \frac{1}{\mathcal{L}^n(\Omega)} \int_{\Omega} W'(u) \dx
\end{equation}
and the \emph{Cahn--Hilliard equation}
\begin{equation}\label{CahnHilliard}
\partial_t u = - \Delta ( \e^2 \Delta u - W'(u)  ),
\end{equation}
each taken with either Neumann or periodic boundary conditions. For background on these equations and their applications see, e.\,g.\,\cite{GarckeReview}. The phenomenon of slow motion of solutions to \eqref{AllenCahn} was analyzed via variational methods first by Bronsard and Kohn \cite{BronsardKohn} for $n=1$. They demonstrate that if $u_\e(0)$ converges in $L^1(J)$ to  $v$, with $J\subset \mathbb{R}$ an open bounded interval and $v$ a local minimizer of $F^{(1)}$, and if $F_\e^{(1)}(u_\e(0)) \leq F^{(1)}(v) + C\e^k$ for some $k$, then for any $l>0$ we have the following \emph{slow-motion inequality}:
\begin{equation}\label{SlowMotionInequality}
\lim_{\e \to 0^+} \sup_{0\leq t \leq l \e^{-k}} \int_J |u_\e(x,t) - v(x)| \dx = 0.
\end{equation}
A crucial estimate in their analysis is the following higher-order asymptotic estimate: that if $\|v_\e- v\|_{L^1} < \delta$ then
\begin{equation} \label{eqn:slowMotionEnergy}
F_\e(v_\e) \geq F^{(1)}(v) - C\e^k
\end{equation}
for appropriately chosen $C>0$.

 Later similar results were established for the non-local Allen--Cahn equation \eqref{NLAllenCahn}, as well as the Cahn--Hilliard equation \eqref{CahnHilliard} (see \cite{BronsardKohn2,BronsardStoth,Grant}). The strength of these results is that they prove this slow motion using transparent variational methods, for initial data that are generic in the sense that they only need have small initial energy. Even though the above-mentioned papers do not explicitly use the setting of $\Gamma$-convergence, they all rely on asymptotic energy inequalities of the form \eqref{eqn:slowMotionEnergy}, which is precisely the $\liminf$ part of the asymptotic development by $\Gamma$-convergence of order 2.
 
 More recently, a tight, higher-order asymptotic expansion of the family $F_\e$ was given by Bellettini, Nayam and Novaga \cite{Bellettini2013}. In that paper they use their result to prove a type of slow motion bound. Their results match the well-known results of Carr and Pego \cite{CarrPego}, which state that phase boundaries of the Allen--Cahn equation should move at speed $e^{-C\e^{-1}}$.
 
In the case $n=1$ these slow dynamics are generally understood to be related to the existence of \emph{ slow manifolds}, and many works focus on the existence of data that approximately moves along a slow manifold. Some critical first works in this direction include \cite{CarrPego,FuscoHale,Pego89}, while a more recent perspective can be found in \cite{OttoReznikoff} and the references therein. It can also be shown \cite{Chen1992} that the time it takes to approach the slow manifold from arbitrary initial data is generally very short.

The slow motion of phase boundaries in higher dimension has been studied by many authors (see, e.g.,  \cite{AlikakosBronsardFusco, AlikakosFusco, AlikakosFusco1998,AlikakosBatesChen}). These works generally focus on the existence of solutions that move very slowly, often along slow manifolds. Generally these results require some ansatz on the initial data, such as radial data or data parametrized by the distance from a manifold. The requirement of such an ansatz in higher dimensions is, in our opinion, due to the lack of higher-order asymptotics of the functional \eqref{F0Definition} in dimension greater than one.

An immediate consequence of \eqref{higherOrderFunctionalDefinition} and Theorems \ref{mainThm1} and \ref{mainThm2} is that when $v$ is a global minimizer of $\mathcal{F}^{(1)}$ we then have, for any sequence $v_\e$ converging to $v$,
\[
\mathcal{F}_\e^{(1)}(v_\e) \geq \mathcal{F}^{(1)}(v) - C \e,
\]
for some $C>0$. Using exactly the techniques from \cite{BronsardKohn} it is possible to establish generic, slow motion results similar to \eqref{SlowMotionInequality}, for the non-local Allen--Cahn and Cahn--Hilliard equations in dimension greater than one, for data that are close to \emph{global} perimeter minimizers \cite{MurrayRinaldi}. We are currently investigating extensions of this type of result in the more interesting case of local perimeter minimizers.

One other setting where a type of higher-order regularity has been studied for $\Gamma$-limits is in the setting of limits of gradient flows \cite{SerfatySandier2004}. Although the types of estimates we derive here are not precisely the type that they use to study convergence of gradient flows, they are certainly related.

This paper is organized as follows. In Section \ref{notationSection} we state our technical assumptions. In Section 3 we develop our new rearrangement result. In Section 4 we analyze a weighted, one-dimensional functional problem. In Section 5 we use the results in Sections 3 and 4 to prove Theorems \ref{mainThm1} and \ref{mainThm2}.

\section{Preliminaries and Main Assumptions} \label{notationSection}

In this paper we consider the Cahn--Hilliard functional \eqref{F0Definition}, where we assume that $\Omega \subset \mathbb{R}^n, n \leq 7,$ is an open, connected, bounded set with
\begin{equation}\label{domainAssumptions}
\mathcal{L}^n(\Omega) =1 \quad \text{ and } \quad \partial \Omega \text{ is of class }  C^{2,\alpha}, \quad \alpha \in (0,1] . 
\end{equation}
We observe that the restriction to $n \leq 7$ is necessary to guarantee regularity of minimizers of the problem \eqref{partitionProblemDefinition} \cite{GMT83,GruterBoundaryRegularity,MaggiBook,SternbergZumbrunIsoPer}, while the assumption that $\mathcal{L}^n(\Omega) = 1$ is for simplicity (the general case follows by a scaling argument). We assume that the mass $m$ in \eqref{massConstraintEquation} satisfies
\begin{equation} \label{originalMassRange}
a<m<b,
\end{equation}
where $a,b$ are the wells of $W$, and that the isoperimetric function $\I$ defined in \eqref{isoFunctionDefinition} satisfies the Taylor expansion
\begin{equation} \label{isoFunctionSmooth}
\I(\V) = \I(\V_m) + \I'(\V_m)(\V-\V_m) + O(|\V-\V_m|^{1+\beta})
\end{equation}
for all $\V$ close to $\V_m= \frac{b-m}{b-a}$ (see \eqref{partitionProblemMassConstraint}) and for some $\beta \in (0,1]$. As remarked in the introduction, for domains of class $C^2$, $\I$ is semi-concave (see \cite{BavardPansu}) and so \eqref{isoFunctionSmooth} holds with $\beta = 1$ at $\mathcal{L}^1$ a.e. $\V_m$ in $[0,1]$ (see \cite{ConvexBook}), or equivalently for $\mathcal{L}^1$ a.e. $m \in (a,b)$.

We also make the following assumptions on the potential $W: \mathbb{R} \to [0,\infty)$:
\begin{align}
\label{W_Smooth}&W \text{ is of class $C^2(\mathbb{R}\backslash \{a,b\})$ and has precisely two zeros at } a<b, \\
\label{WPrime_At_Wells}&\lim_{s \to a} \frac{W''(s)}{|s-a|^{q-1}} = \lim_{s \to b}\frac{W''(s)}{|s-b|^{q-1}} := \ell > 0, \quad q \in (0,1] ,\\
\label{W_Number_Zeros} &\text{ $W'$ has exactly 3 zeros  at $a<c<b$,} \quad W''(c)<0, \\
\label{WGurtin_Assumption}& \liminf_{|s| \to \infty} |W'(s)| > 0.
\end{align}

Most of these assumptions are standard (see \cite{GurtinMatano}). We note that in the case where $q=1$ we have that $\ell$ is simply $W''(a)$. In particular we note that $q=1$ when $W(s)=\frac12(s^2-1)^2$, which is the classical Cahn--Hilliard potential (see, e.g.,  \cite{CahnHilliard}). While it is possible to deal with different limits at $a$ and $b$ in \eqref{WPrime_At_Wells}, we do not handle those cases in our analysis for clarity of presentation.

\begin{remark}
In view of \eqref{W_Smooth}-\eqref{WGurtin_Assumption}, we have that there exists an $\hat L>0$ and $\hat T>0$ so that 
\begin{equation}
W(s) \geq \hat L|s| \label{WLinearGrowth}
\end{equation}
for all $|s| > \hat T$.
\end{remark}

\begin{remark}
In view of \eqref{W_Smooth} and \eqref{WPrime_At_Wells} if follows from de l'H\^{o}pital's rule that 
\begin{align} \label{W_Limits}
\lim_{s\to a} \frac{W(s)}{|s-a|^{1+q}} = \lim_{s\to b} \frac{W(s)}{|s-b|^{1+q}} = \frac{\ell}{q(1+q)},\\
\lim_{s\to a} \frac{W'(s)}{(s-a)|s-a|^{q-1}} = \lim_{s\to b} \frac{W'(s)}{(s-b)|s-b|^{q-1}} = \frac{\ell}{q}. \label{WPrimeLimits}
\end{align}
\end{remark}

In turn, by \eqref{W_Smooth}, there exist $c_1, c_2 >0$ such that $c_1^2(b-s)^{1+q} \leq W(s) \leq c_2^2(b-s)^{1+q}$ for all $s \in [\frac{a+b}{2},b]$. It follows that the solution $z$ of the Cauchy problem \eqref{profileCauchyProblem} satisfies
\begin{align}
\left[	 (b-z(t_0))^{\frac{1-q}2} - \frac{(1-q)c_2}{2}(t-t_0)\right]_+^{\frac2{1-q}} &\leq b-z(t) \nonumber \\
&\leq \left[	 (b-z(t_0))^{\frac{1-q}2} - \frac{(1-q)c_1}{2} (t-t_0)\right]_+^{\frac2{1-q}} \label{CPDecay1}
\end{align}
for all $t \geq t_0 \geq 0$ if $0<q < 1$ and
\begin{equation}\label{CPDecay2}
(b-z(t_0))e^{-c_2 (t-t_0)} \leq b-z(t) \leq (b-z(t_0))e^{-c_1(t-t_0)}
\end{equation}
for all $t \geq t_0 \geq 0$ for $q =1$, where $[\cdot]_+$ denotes the positive part. In particular, in the case $0<q<1$,  since $z(0) = c$, there exists a constant
\[
\left( \frac{b-a}{2}\right)^{\frac{1-q}2} \frac{2}{c_2(1-q)} \leq t_b \leq \left( \frac{b-a}{2}\right)^{\frac{1-q}2} \frac{2}{c_1(1-q)}
\]
such that
\begin{equation} \label{CPFiniteWidth}
z(t) \equiv b \quad \text{for all } t \geq t_b.
\end{equation}
Similar estimates hold near $a$, so that $z(t) \equiv a$  for all $t \leq t_a < 0$ when $0<q < 1$.

In what follows, given a non-empty set $E\subset \mathbb{R}^m$, we denote by $E^\circ$, $\bar E$ and $E^c$ the interior, closure and complement of $E$ respectively. We let $d(x,E)$ be the distance from $x$ to $E$ and we define $d_E$ to be the \emph{signed distance function} from the set $E$, namely
\begin{equation} \label{def:SignedDistance}
d_E(x) := \begin{cases} \phantom{-}d(x, \partial E) &\text{ if } x \in E^c, \\ -d(x,\partial E) &\text { if } x \in E. \end{cases}
\end{equation}
Also, $\mathcal{L}^m$ and $\mathcal{H}^m$ are the $m$-dimensional Lebesgue and Hausdorff measures, respectively.

In the remainder of the paper the constant $C$ \emph{varies from line to line} and is \emph{independent of $\e$}, without further mention.

\section{A P\'olya--Szeg\H o Type Inequality} \label{PolyaSzegoSection}

The classical P\'olya--Szeg\H o inequality states that in $\mathbb{R}^n$ the decreasing spherical rearrangement $u^*$ of a positive function $u \in W^{1,p}(\mathbb{R}^n)$ will not increase the $L^p$ norm of the gradient \cite{Kawohl,LeoniBook}. This permits complicated problems in arbitrary dimensions to be reduced to radial, one-dimensional problems. For Dirichlet problems it is often possible to obtain similar inequalities for functions on a bounded domain $\Omega$. Specifically, if $u \in H_0^1(\Omega)$ is positive then we can use the P\'olya--Szeg\H o inequality in the whole space (after extending the function $u$) to show that $ \|\nabla u^*\|_{L^p(\Omega^*)}\leq \|\nabla u\|_{L^p(\Omega)}$. A classical example where this technique is used is in the proof of Talenti's inequality \cite{Talenti1976}. In this section we study a type of rearrangement \cite{Cianchi1996}, which does not require extending functions to all of $\mathbb{R}^n$, and is thus better suited to analyzing certain Neumann problems. Although many of the techniques are identical to those used in proving the standard P\'olya--Szeg\H o inequality, we include all the proofs for the convenience of the reader.

In this section only we assume that $\Omega$ is bounded, connected, has measure $\mathcal{L}^n(\Omega) =1$ and has Lipschitz boundary. Then the isoperimetric function (see \eqref{isoFunctionDefinition}) satisfies
\begin{equation} \label{muBoundsNearZero}
\I(\V) \geq  C_1 \min\{\V,1-\V\}^{\frac{n-1}{n}} \quad \text{ for all } \V \in [0,1].
\end{equation}
Indeed, this bound follows from Corollary 3 in Section 5.2.1 of \cite{MazyaBook} (see also \cite{AlbericoCianchi} and \cite{CianchiMazya}). By considering sets and their complements it is clear that $\I(\V) = \I(1 - \V)$. We now prove an elementary proposition.

\begin{proposition}\label{prop:IStar}
Suppose that $\I$ satisfies \eqref{isoFunctionSmooth} and \eqref{muBoundsNearZero}. Then there exists a function $\I^* \in C^{1,\beta}_{\operatorname*{loc}}(0,1)$ such that
\begin{align}
\I^*(\V) &= \I^*(1-\V) \quad \text{ for all }  \V\in (0,1), \label{eqn:IStarSymmetric}\\
\I &\geq \I^*>0\quad  \text{ in } (0,1), \label{eqn:IStarPositive}\\
\I(\V_m) &= \I^*(\V_m),\quad \I'(\V_m) = (\I^*)'(\V_m),\label{eqn:IStarTouches}\\
\I^*(\V) &= C_0 \V^{\frac{n-1}{n}} \quad \text{ for all }  \V \in (0,\delta) \label{eqn:IStarTail}
\end{align}
for some $C_0>0$ and $0<\delta<1$.
\end{proposition}
\begin{proof}
Assume first that $\V_m \in (0,1/2)$. By \eqref{isoFunctionSmooth} there exists $C_0>0$ so that 
\begin{equation} \label{eqn:ITaylor}
|\I(\V) - \I(\V_m) - \I'(\V_m)(\V-\V_m)| \leq C_0|\V-\V_m|^{1+\beta}
\end{equation}
for all $\V \in (\V_m-\delta,\V_m+\delta)$, for all $\delta>0$ sufficiently small. 
Define 
\begin{equation} \label{eqn:IStarDefinition1}
\widehat \I (\V) : = \I(\V_m) + \I'(\V_m)(\V-\V_m) -2C_0|\V-\V_m|^{1+\beta}
\end{equation}
for $\V \in (\V_m-\delta,\V_m+\delta) \cap (0,1/2)$. Then by \eqref{eqn:ITaylor}, $\I \geq \widehat \I$ in $(\V_m-\delta,\V_m+\delta) \cap (0,1/2)$ with strict inequality for $\V \neq \V_m$. Moreover
\[
\widehat \I(\V) \geq \I(\V_m)/4 > 0
\]
for $\V \in (\V_m-\delta,\V_m+\delta) \cap (0,1/2)$, for all $\delta$ sufficiently small. Since $\I(\V) > \widehat \I(\V)$ for all $V \in (\V_m-\delta,\V_m+\delta) \cap (0,1/2) \backslash \{\V_m\}$, and since $\I(\V) \geq C_1 \V^{n-1}$ for $\V \in (0,1/2)$ and $\I$ is continuous and positive in $[0,1]$, and is symmetric, we can extend $\widehat \I$ to a function $\I^*$ that satisfies \eqref{eqn:IStarSymmetric}-\eqref{eqn:IStarTail}.

If $\V_m=1/2$ then $\I'(\V_m) = 0$, by the symmetry of $\I$ with respect to $1/2$. In turn the function $\widehat \I$ in \eqref{eqn:IStarDefinition1} is also symmetric with respect to $1/2$, and so we can define $\widehat \I$  as in \eqref{eqn:IStarDefinition1} for all $r \in (1/2-\delta,1/2+\delta)$ and continue as before. The case $\V_m \in (1/2,1)$ follows by symmetry.
\end{proof}

In subsequent sections it is more convenient to work with $\I^*$ instead of $\I$. The results of this section, however, hold when using $\I$ instead of $\I^*$.


Our goal now is to construct a rearranged domain $\Omega^*$ such that the perimeter of the set $\{x_n >t\}$ inside $\Omega^*$ matches the modified isoperimetric function $\I^*$ of $\Omega$ when evaluated at the measure of $\Omega^* \cap \{x_n>t\}$, in other words so that $P(\Omega^* \cap \{x_n>t\};\Omega^*) = \I^*(\mathcal{L}^n(\Omega^* \cap \{x_n>t\}))$ (see Lemma \ref{rearrangedVolumeLemma} below). 

To this end we define a function $V_\Omega$ as the solution to the following Cauchy problem:
\begin{equation} \label{DomainRearrangeODE}
\frac{d}{dt} V_\Omega(t) = \I^*(V_\Omega(t)),\hspace{4mm} V_\Omega(0) = 1/2 .
\end{equation}

We can extend $\I^*$ to be zero outside of $[0,1]$. Since $\I^*$ is bounded and continuous (see Proposition \ref{prop:IStar}), the Cauchy problem \eqref{DomainRearrangeODE} admits a global solution $V_\Omega: \mathbb{R} \to [0,1]$. It follows from inequality \eqref{muBoundsNearZero} that there is a $T > 0$ so that $0 < V_\Omega(t)$ for $-T<t<0$ and $V_\Omega(-T) = 0$. Moreover by equation \eqref{eqn:IStarSymmetric} we have that $V_\Omega(T) = 1$ and $V_\Omega(t) < 1$ for all $0<t<T$. Define
\begin{equation}
I := (-T,T). \label{DefI}
\end{equation}
Observe that $V_\Omega$ is uniquely defined because $\I^*$ is locally Lipschitz on any compact subset of $(0,1)$, and $\I^* =0$ outside of $(0,1)$.

In what follows for $y\in\mathbb{R}^n$ we use the notation $y=(y',y_n)\in \mathbb{R}^{n-1}\times \mathbb{R}$. Next we define a set $\Omega^* \subset \mathbb{R}^n$, which will be a type of rearrangement of $\Omega$. 
This set is defined by:
\begin{equation}\label{rearrangedDomainDefinition}
\Omega^* := \left\{y : \hspace{2mm}y_n \in I, \hspace{2mm}y' \in B_{n-1}(0,r(y_n)) \right\} ,
\end{equation}
where for $t \in I$,
\[
r(t):= \left(\frac{\I^*(V_\Omega(t))}{\alpha_{n-1}}\right)^{1/(n-1)} \text{ and } \alpha_{n-1} := \mathcal{L}^{n-1}(B_{n-1}(0,1)).
\]
Note that the definition of $r(t)$ implies that
\begin{equation} \label{rearrangedSlicesRelation}
\mathcal{L}^{n-1}(B_{n-1}(0,r(t))) = \I^*(V_\Omega(t))
\end{equation}
for all $t \in \overline{I}$.

\begin{figure}
\centering
\begin{subfigure}{.5\textwidth}
  \centering
  \includegraphics[width = .8\linewidth]{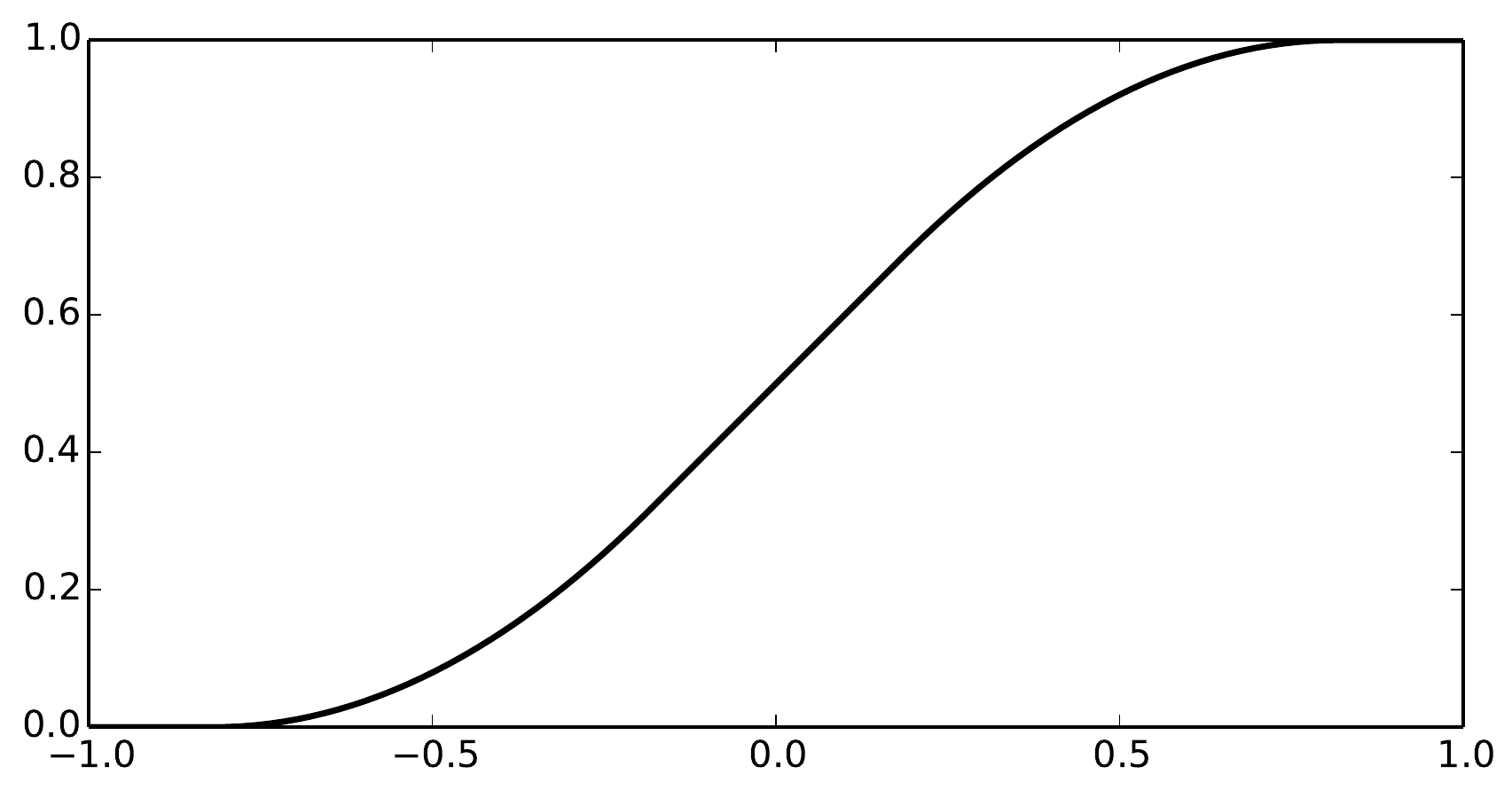}
  \label{fig:vOmega}
\end{subfigure}%
\begin{subfigure}{.5\textwidth}
  \centering
  \includegraphics[width = .8\linewidth]{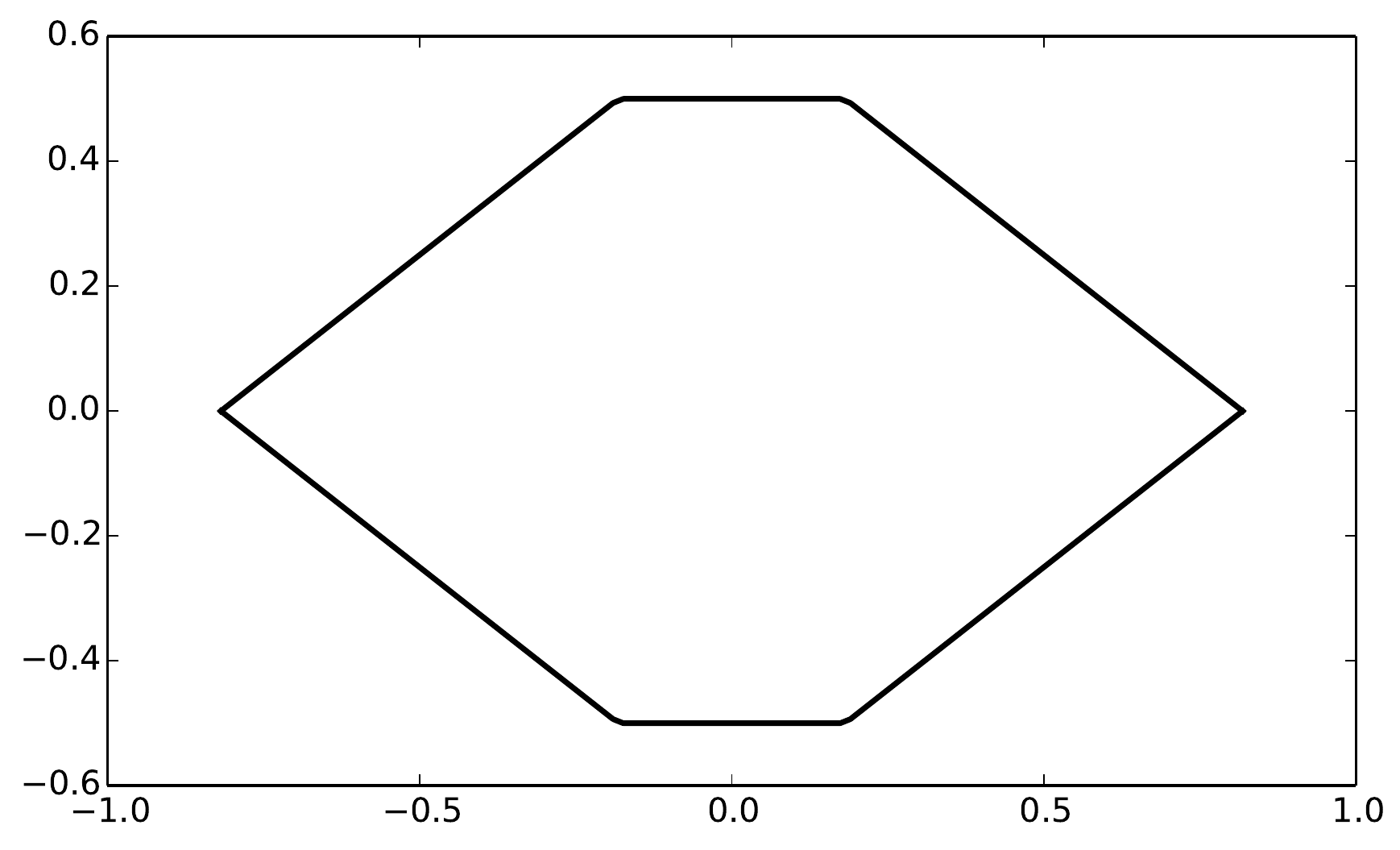}
  \label{fig:omegaStar}
\end{subfigure}
\caption{$V_\Omega$ and $\Omega^*$ for the domain $\Omega = Q_2$, the cube in $\mathbb{R}^2$. In $\mathbb{R}^2$, the quantity $r(t)$ is half the height of $\Omega^*$. In this case $T = 1/2 + 1/\pi \approx .82$.}
\label{fig:rearrangedDomain}
\end{figure}

The following lemma motivates our choice of the Cauchy problem \eqref{DomainRearrangeODE}.
\begin{lemma}\label{rearrangedVolumeLemma}
For any $t \in \overline{I}$ the following equalities hold:
\begin{align}
V_\Omega(t) &= \mathcal{L}^n(\Omega^* \cap \{y_n < t\}) , \label{rearrangedVolumeRelation}\\
\I^*(V_\Omega(t)) &= \operatorname*{P}(\{y_n < t\}; \Omega^*). \label{rearrangedPerimeterRelation}
\end{align}
\end{lemma}

\begin{proof}
We prove equation \eqref{rearrangedVolumeRelation} by using Fubini's theorem, equation \eqref{rearrangedSlicesRelation}, the Cauchy problem \eqref{DomainRearrangeODE}, the fundamental theorem of calculus, and the fact that $V_\Omega(-T) = 0$, in that order:
\begin{align*}
\mathcal{L}^n(\Omega^* \cap \{y_n < t\}) &= \int_{-T}^t \mathcal{H}^{n-1}(\Omega^* \cap \{y_n=s\}) \ds \\
&= \int_{-T}^t \I^*(V_\Omega(s))  \ds \\
&= V_\Omega(t) - V_\Omega(-T) = V_\Omega(t) .
\end{align*}

Equality \eqref{rearrangedPerimeterRelation} follows immediately from equation \eqref{rearrangedSlicesRelation} and definition \eqref{rearrangedDomainDefinition}.
\end{proof}

Now given any measurable function $u:\Omega \to \mathbb{R}$, we define the distribution function $\varrho_u(s) := \mathcal{L}^n(\{u > s\})$ and the following function:
\begin{equation} \label{gSubUDefinition}
g_u(t) := \sup \{s \in \mathbb{R} : \varrho_u(s) > V_\Omega(t)\} .
\end{equation}
We then define a function $u^*: \Omega^* \to \mathbb{R}$ as follows:
\begin{equation}\label{uStarDefinition}
u^*(y', y_n) := g_u(y_n) .
\end{equation}

\begin{figure}
\centering
\begin{subfigure}{.33\textwidth}
  \centering
  \includegraphics[width = \linewidth]{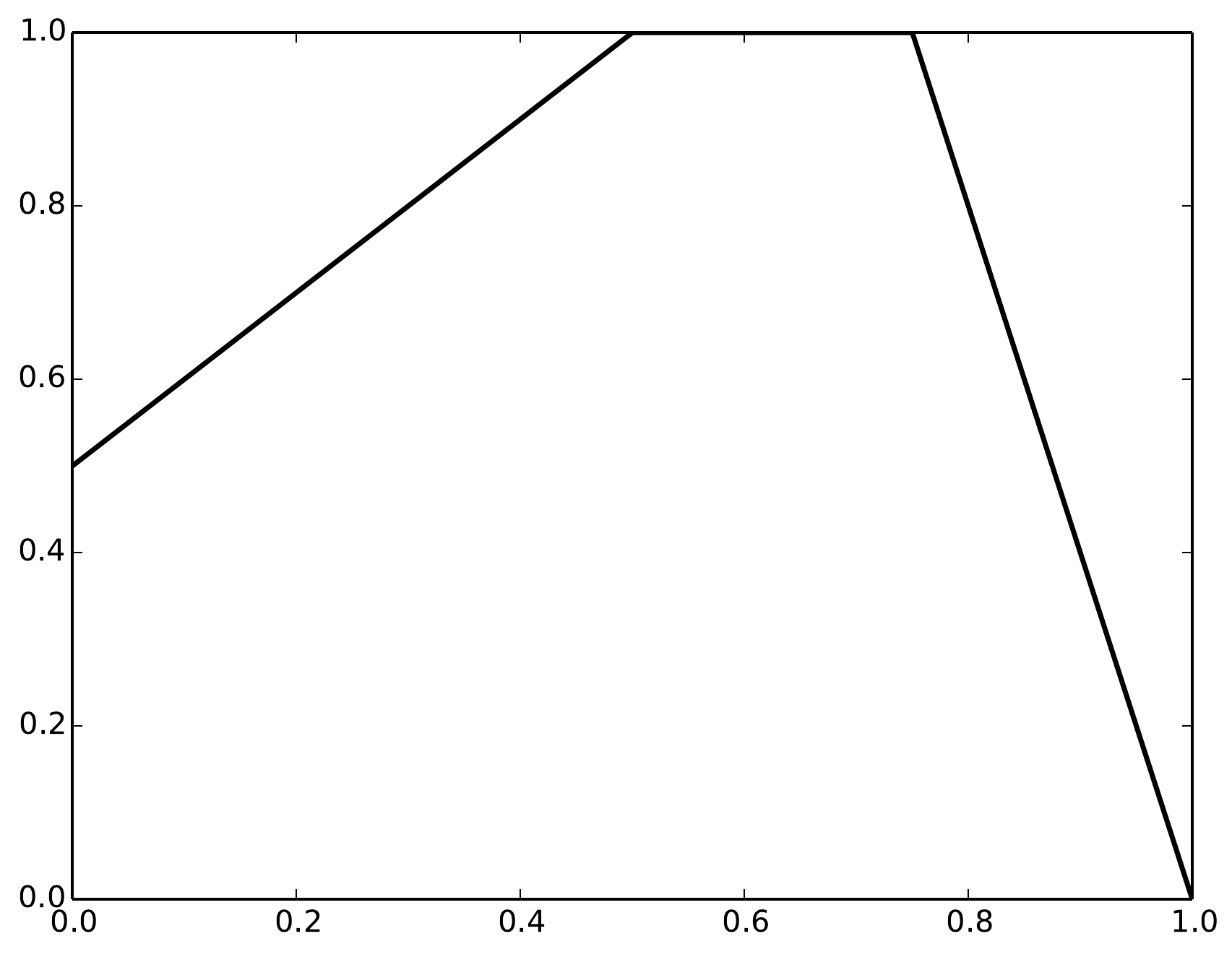}
  \label{fig:rearrange1}
\end{subfigure}%
\begin{subfigure}{.33\textwidth}
  \centering
  \includegraphics[width = \linewidth]{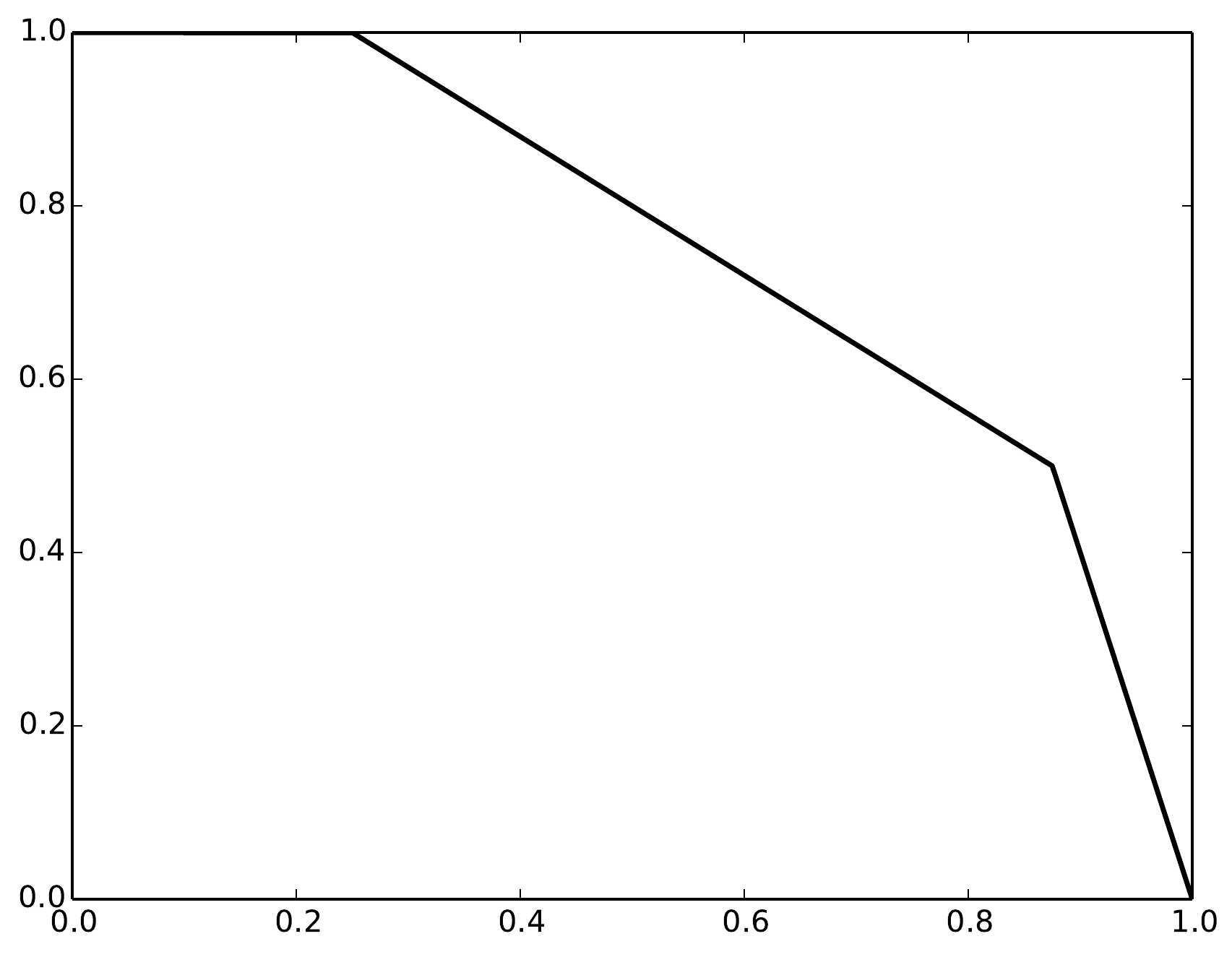}
  \label{fig:rearrange2}
\end{subfigure}%
\begin{subfigure}{.33\textwidth}
  \centering
  \includegraphics[width = \linewidth]{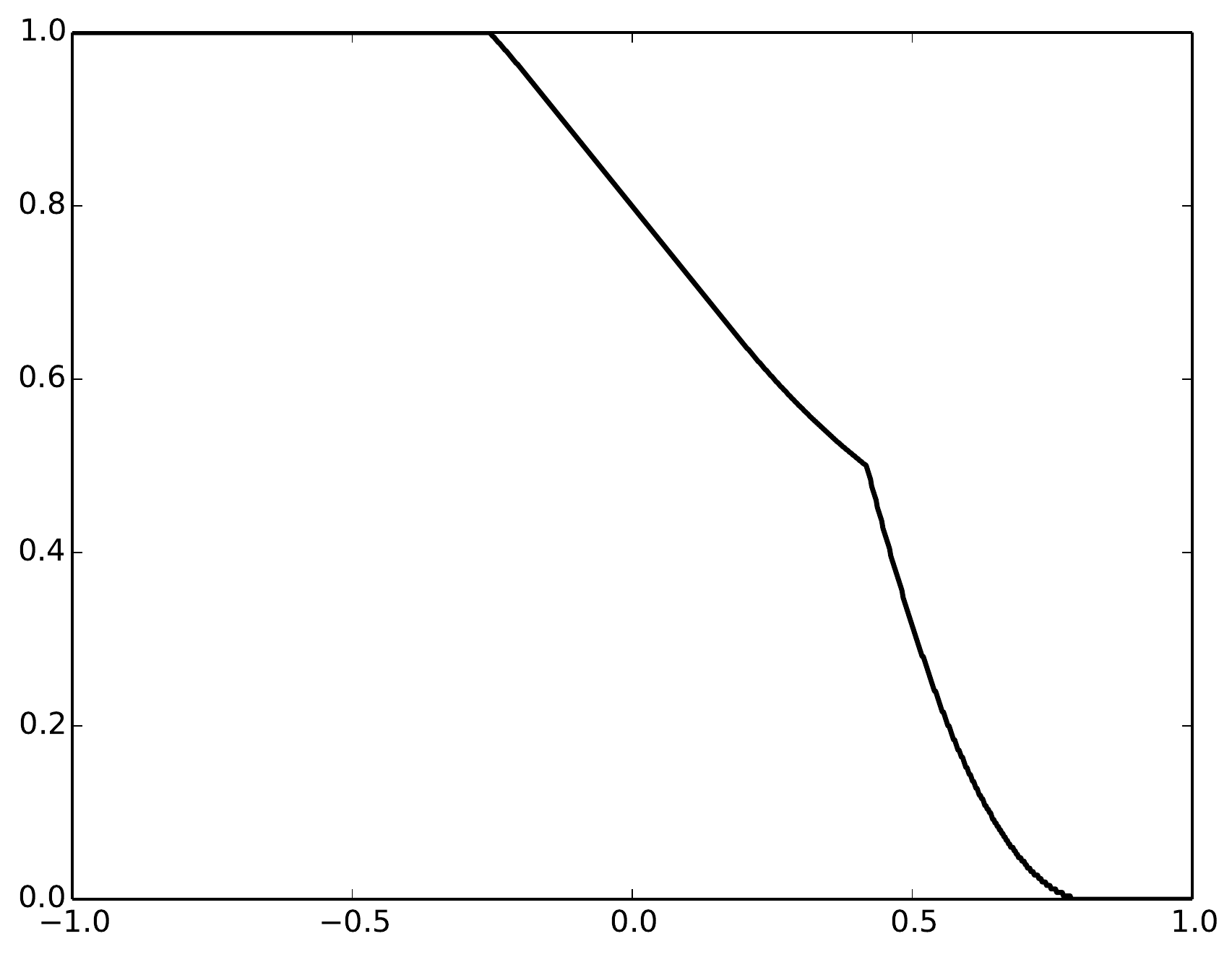}
  \label{fig:rearrange3}
\end{subfigure}
\caption{A piecewise constant function, its decreasing rearrangement, and the corresponding $g_u$}
\label{fig:rearrangeMain}
\end{figure}

The first important property of our rearranged function can be summarized by the following lemma:

\begin{lemma}\label{Equimeasurable}
Let $u:\Omega \to \mathbb{R}$ be a measurable function. Then the functions $u^*$ and $u$ are equimeasurable, meaning that $\varrho_u= \varrho_{u^*}$. This implies that for any Borel function $\psi:\mathbb{R} \to \mathbb{R}$,
\[
\int_\Omega \psi (u) \dx= \int_{\Omega^*} \psi (u^*) \dy = \int_I \psi(g_u) \I^*(V_\Omega)  \dt,
\]
assuming that the previous integrals are well-defined. In particular the $L^p$ norms of $u$ and $u^*$ are preserved, as well as the integral of $W\circ u$.
\end{lemma}

\begin{proof}
First we note that, by standard arguments, $\varrho_u$ is decreasing and right continuous and that $g_u$ is decreasing and left continuous (see, e.g.,  \cite{LeoniBook}, p. 478).

Let $h(t) := \sup \{s :  g_u(s) > t\}$. Since $g_u$ is decreasing we have that
\begin{align*}
\varrho_{u^*}(t) &= \mathcal{L}^n(\{y \in \Omega^* : g_u(y_n) > t\})\\
&=\mathcal{L}^n(\{y \in \Omega^* : y_n < h(t)\}) = V_\Omega(h(t)), 
\end{align*}
where in the last equality we have used Lemma \ref{rearrangedVolumeLemma}.

We then claim that $V_\Omega(h(t)) = \varrho_u(t)$. To see this observe that since $\I^*>0$ in $(0,1)$, by \eqref{DomainRearrangeODE} we have that $V_\Omega$ is strictly increasing and of class $C^1$ in $I$. Hence:
\begin{align}
V_\Omega(h(t)) &= V_\Omega(\sup\{s: g_u(s) > t\}) = \sup\{ V_\Omega(s) : g_u(s) > t\} \nonumber \\
&= \sup \{V_\Omega(s) : \sup\{\varkappa: \varrho_u(\varkappa) > V_\Omega(s)\} > t\} \nonumber \\
&= \sup\{\rho : \sup \{\varkappa: \varrho_u(\varkappa)>\rho\} > t\} . \label{VOmegaSupEquality}
\end{align}

For every $\rho$ such that $\sup\{\varkappa: \varrho_u(\varkappa) > \rho\} > t$, there exists $\varkappa > t$ such that $\varrho_u(\varkappa) > \rho$. But since $\varrho_u$ is decreasing we have that $\varrho_u(t) \geq \varrho_u(\varkappa) > \rho$, which then shows that
\[
V_\Omega(h(t)) = \sup\{\rho : \sup \{\varkappa: \varrho_u(\varkappa)>\rho\} > t\} \leq \varrho_u(t).
\]

Now if $V_\Omega(h(t)) < \varrho_u(t)$, then $V_\Omega(h(t)) < \varrho_u(t) - \epsilon$ for some $\epsilon>0$. By equation \eqref{VOmegaSupEquality} this implies that
\[
\sup\{ \varkappa: \varrho_u(\varkappa) > \varrho_u(t) - \e\} \leq t .
\]
By the right continuity of $\varrho_u$ for some $\delta > 0$ we have that $\varrho_u(t + \delta) > \varrho_u(t) -\epsilon$, which violates the previous inequality. This then implies that $\varrho_u(t) = V_\Omega(h(t)) = \varrho_{u^*}(t)$ for all $t$, which is the desired conclusion.

To see the integral equality stated, we note that (see, e.g.,  Theorem B.61 in \cite{LeoniBook}):
\[
\int_\Omega \psi(u(x)) \dx = \int_{\mathbb{R}} \psi(s) \, d\varrho_u(s) = \int_{\mathbb{R}} \psi(s) \, d\varrho_{u^*}(s) = \int_{\Omega^*} \psi(u^*(y))\dy .
\]

This concludes the proof.

\end{proof}

Next we prove that the operation of rearrangement is a contraction in $L^1$. It should be possible to prove a more general version of this proposition, but this suffices for our purposes.

\begin{proposition} \label{rearrangementContraction}
Suppose that $u_1,u_2 \in L^1(\Omega)$. Then
\begin{equation}\label{eqn:ContractionInequality}
\|u_1^*-u_2^*\|_{L^1(\Omega^*)} \leq \|u_1 - u_2\|_{L^1(\Omega)}.
\end{equation}
\end{proposition}

\begin{proof}
By the definition of $u^*$ (see \eqref{gSubUDefinition} and \eqref{uStarDefinition}) we have  that if $u_1 \leq u_2$ then $u_1^* \leq u_2^*$. In light of Lemma \ref{Equimeasurable}, we may apply Proposition 1 in \cite{CrandallTartar}, which states that a mapping from $L^1$ to $L^1$ that preserves integrals and is order preserving must be a contraction in $L^1$, to obtain \eqref{eqn:ContractionInequality}.
\end{proof}

Next we prove a lemma stating that truncation and rearrangement commute. This will later allow us to establish that the rearrangement of a Sobolev function is still a Sobolev function.

\begin{lemma}\label{TrucationRearrangementLemma} 
Let $u: \Omega \to \mathbb{R}$ be measurable. Given $s_1<s_2$, let $\operatorname*{Tr}\nolimits_{s_1,s_2}(s) := (s \vee s_1) \wedge s_2, s \in \mathbb{R}$. Then the following equality holds:
\[
\operatorname*{Tr}\nolimits_{s_1,s_2}(u^*) = (\operatorname*{Tr}\nolimits_{s_1,s_2}(u))^* .
\]
\end{lemma}

\begin{proof}
Set $v := \operatorname*{Tr}\nolimits_{s_1,s_2}(u)$. By definition \eqref{uStarDefinition} it suffices to show that $\operatorname*{Tr}\nolimits_{s_1,s_2} g_u = g_v$. Let $t_1 := \sup\{t : g_u(t) \geq s_2\}$ and $t_2 := \inf \{t : g_u(t) \leq s_1\}$.

\textbf{Step 1:} By Lemma \ref{Equimeasurable} and the definition of truncation we can deduce the following:
\begin{align*}
\mathcal{L}^n(\{\operatorname*{Tr}\nolimits_{s_1,s_2} (u^*) = s_2\}) &= \mathcal{L}^n(\{u^* \geq s_2\}) = \mathcal{L}^n(\{u \geq s_2\}) \\
&= \mathcal{L}^n(\{v = s_2\}) = \mathcal{L}^n(\{v^* = s_2\}) .
\end{align*}

As $g_u$ is decreasing we find that for $t < t_1$, $\operatorname*{Tr}\nolimits_{s_1,s_2} g_u(t) = s_2$. Since $g_v$ is decreasing and is bounded above by $s_2$ the previous chain of equalities implies that for $t<t_1$ we have that $g_v(t) = s_2$, which then implies that $g_v(t) = \operatorname*{Tr}\nolimits_{s_1,s_2} g_u(t)$ for such $t$.

Using an identical argument we find that for $t > t_2$ we have that $\operatorname*{Tr}\nolimits_{s_1,s_2} g_u(t) = g_v(t)$.

\textbf{Step 2:} Next we consider $t \in (t_1,t_2)$. First, for $s \in (s_1,s_2)$ we have that $\varrho_u(s) = \varrho_v(s)$. Next, we note that for $t \in (t_1,t_2)$ we have that $g_u(t) \in (s_1,s_2)$ and that $g_v(t) \in (s_1,s_2)$. Thus we can write the following for $t \in (t_1,t_2)$:
\begin{align*}
\operatorname*{Tr}\nolimits_{s_1,s_2} (g_u(t)) &= g_u(t) = \sup\{s \in \mathbb{R} : \varrho_u(s) > V_\Omega(t)\} \\
&= \sup\{s \in (s_1,s_2) : \varrho_u(s) > V_\Omega(t)\} \\
&= \sup\{s \in (s_1,s_2) : \varrho_v(s) > V_\Omega(t)\} \\
&= \sup\{s \in \mathbb{R} : \varrho_v(s) > V_\Omega(t)\} = g_v(t).
\end{align*}

\textbf{Step 3:} As $\operatorname*{Tr}\nolimits_{s_1,s_2} (g_u)$ and $g_v$ are both left continuous, we then have that $\operatorname*{Tr}\nolimits_{s_1,s_2} g_u = g_v$ everywhere, as desired.
\end{proof}

Next we state a simple identity related to level sets of functions. This is well-known (see \cite{FuscoCianchi}), but we include the proof for completeness.

\begin{lemma} \label{Lem:CoareaTrick}
For $u \in W^{1,1}(\Omega)$ there exists a representative of $u$ such that  the following equality holds for all $s_1<s_2$:
\begin{equation}\label{coareaTrick}
\int_{s_1}^{s_2} \int_{\{u=s, \nabla u \neq 0\}} |\nabla u(x)|^{-1}\, d\mathcal{H}^{n-1}  \ds = \mathcal{L}^n(\{x \in \Omega : u(x) \in (s_1,s_2), \nabla u(x) \neq 0\} ) .
\end{equation}
\end{lemma}

\begin{proof}
Let $H_\e := (\e + |\nabla u|)^{-1}$. By the coarea formula (see \cite{EvansGariepy}) we find that:
\begin{align*}
\int_{\{s_1 < u< s_2 , \hspace{1mm} \nabla u \neq 0\}} H_\e |\nabla u| \dx &=\int_{\{s_1 < u < s_2\}} H_\e |\nabla u|\dx \\
&= \int_{s_1}^{s_2} \int_{u^{-1}(s)} H_\e \, d\mathcal{H}^{n-1}  \ds .
\end{align*}
By noting that $H_\e \to |\nabla u|^{-1}$ monotonically in the set $\{\nabla u \neq 0\}$, we find that \eqref{coareaTrick} holds.
\end{proof}

Next we state and prove a simple lemma, which is essentially an isoperimetric inequality.

\begin{lemma} \label{RearrangedIsoperimetricInequality}
Given $u \in BV(\Omega)$, for any $t$ the following must hold:
\[
\operatorname*{P}(\{u^* > t\};\Omega^*) \leq \operatorname*{P}(\{u>t\};\Omega).
\]
\end{lemma}

\begin{proof}
As $g_u$ is a decreasing function (see \eqref{gSubUDefinition}), we note that the set $\{u^*>t\}$ is actually a set of the form $\{y_n < s\}$ or $\{y_n \leq s\}$. Since hyperplanes have $\mathcal{L}^n$ measure zero, by Lemma \ref{rearrangedVolumeLemma} we have that
\[
V_\Omega(s) =  \mathcal{L}^n(\Omega^* \cap \{y_n < s\}) = \varrho_{u^*}(t).
\]
By then recalling that $u$ and $u^*$ are equimeasurable (see Lemma \ref{Equimeasurable}) and by Lemma \ref{rearrangedVolumeLemma} we have the following:
\begin{align*}
\operatorname*{P}(\{u^*>t\};\Omega^*) &= \I^*(\varrho_{u^*}(t)) = \I^*(\varrho_u(t)) \leq \I(\varrho_u(t)) \\
&\leq \operatorname*{P}(\{u>t\};\Omega) ,
\end{align*}
where we have used the fact that $\I^*\leq \I$ and \eqref{isoFunctionDefinition}. This concludes the proof.
\end{proof}

Next we prove two lemmas that are preliminary to establishing our P\'olya--Szeg\H o type result.

\begin{lemma}\label{rearrangedTVLemma}
Given $u \in BV(\Omega)$, we have that $u^* \in BV(\Omega^*)$ and that the following inequality holds:
\[
\int_I\I^*(V_\Omega(s)) d |Dg_u|(s) = |Du^*|(\Omega^*) \leq |Du|(\Omega) .
\]
\end{lemma}
\begin{proof}
By Lemma \ref{Equimeasurable} we have that $u^* \in L^1(\Omega^*)$. By \eqref{gSubUDefinition} and by the fact that $g_u$ is decreasing, it follows that $g_u \in BV_{\operatorname*{loc}}(I)$ (see, e.g.,  Theorem 7.2 in \cite{LeoniBook}).

Moreover by our definition of $u^*$ (see \eqref{rearrangedDomainDefinition}, \eqref{rearrangedSlicesRelation}, \eqref{gSubUDefinition}, and Lemma \ref{Equimeasurable}) we can write the following:
\begin{align*}
&|Du^*|(\Omega^*) =  \sup\left\{\int_{\Omega^*} \phi(y',y_n) \,dy' \,d(Dg_u)(y_n): \phi \in C_0(\Omega^*), \|\phi\|_{C_0} \leq 1 \right\} \\
&= \sup\left\{\int_I \left(\int_{B_{n-1}(0,r(y_n))}\phi(y', y_n) \, dy' \right)\, d(Dg_u)(y_n): \phi \in C_0(\Omega^*), \|\phi\|_{C_0} \leq 1 \right\} \\
&= \sup\left\{\int_I \I^*(V_\Omega(y_n)) \psi(y_n) \, d(Dg_u)(y_n): \psi \in C_0(-T,T), \|\psi\|_{C_0} \leq 1\right\} \\
&= \int_I \I^*(V_\Omega(y_n)) \, d|Dg_u|(y_n) .
\end{align*}

Next we utilize the coarea formula and Lemma \ref{RearrangedIsoperimetricInequality} as follows:
\[
|Du^*|(\Omega^*) = \int_{\mathbb{R}} \operatorname*{P}(\{u^* > t\};\Omega^*)  \dt \leq \int_{\mathbb{R}} \operatorname*{P}(\{u>t\};\Omega) \dt = |Du|(\Omega).
\]

This proves the desired lemma.
\end{proof}

\begin{lemma}\label{sobolevPreservationLemma}
Given $u \in W^{1,1}(\Omega)$, it follows that $u^*\in W^{1,1}(\Omega^*)$.
\end{lemma}

\begin{proof}
By \eqref{uStarDefinition} and Lemma \ref{rearrangedTVLemma} it suffices to show that $g_u$ is absolutely continuous on any sub-interval $[t_0,t_1]$ compactly contained in $I$, that is, that for any $\epsilon>0$ there exists $\delta_0>0$ such that for any finite collection of non-overlapping subintervals $(a_k, b_k)$ of $[t_0,t_1]$ satisfying $\sum_{k=1}^N (b_k - a_k) \leq \delta_0$ we have $\sum_{k=1}^N |g_u(b_k)-g_u(a_k)| \leq \epsilon$. Fix $\epsilon >0$, and let $\delta$ be small enough such that for any measurable $E \subset \Omega$ with $\mathcal{L}^n(E) \leq \delta$ the following holds (see \eqref{eqn:IStarPositive}):
\begin{equation} \label{uniformlyIntegrableBound}
\int_E |\nabla u| \dx \leq \epsilon \min_{t \in [t_0,t_1]}\I^*(V_\Omega(t)) .
\end{equation}
Now consider any finite collection of non-overlapping subintervals $(a_k, b_k)$ of $[t_0,t_1]$, satisfying
\begin{equation}\label{ACIntervalBoundSize}
\sum_{k=1}^N (b_k - a_k) \leq \frac{\delta}{\max_{t \in [t_0,t_1]}\I^*(V_\Omega(t))} .
\end{equation}

The following estimate holds by \eqref{DomainRearrangeODE}, \eqref{rearrangedVolumeRelation}, \eqref{gSubUDefinition}, \eqref{uStarDefinition}, Lemma \ref{Equimeasurable} and \eqref{ACIntervalBoundSize}:
\begin{align}
&\mathcal{L}^n\left(\bigcup_{k=1}^N \{x \in \Omega : g_u(b_k) < u(x) < g_u(a_k)\}\right)  \nonumber \\
&= \sum_{k=1}^N \mathcal{L}^n( \{y \in \Omega^* : g_u(b_k) < u^*(y) < g_u(a_k)\}) \nonumber \\
& \leq \sum_{k=1}^N (V_\Omega(b_k) - V_\Omega(a_k)) \leq \max_{t \in [t_0,t_1]} \I^*(V_\Omega(t))  \sum_{k=1}^N (b_k - a_k)  \leq \delta . \label{ACEstimate}
\end{align}

Next, set $s_1 := g_u(b_k)$ and $s_2 := g_u(a_k)$ and let $v := \operatorname*{Tr}\nolimits_{s_1,s_2} u$. As the pointwise variation and the total variation of the decreasing and left continuous function $g_u$ on an interval coincide (see Theorem 7.2 in \cite{LeoniBook}), by applying Lemma \ref{TrucationRearrangementLemma}, Lemma \ref{rearrangedTVLemma} above we obtain
\begin{align*}
&\min_{t \in [t_0,t_1]}\I^*(V_\Omega(t))  |g_u(a_k) - g_u(b_k)| \\
&\leq \int_{a_k}^{b_k} \I^*(V_\Omega(t)) \, d |D g_u|(t) = \int_I \I^*(V_\Omega(t)) \, d |D (\operatorname*{Tr}\nolimits_{s_1,s_2} g_u)|(t) \\
&= |D v^*|(\Omega^*) \leq |Dv|(\Omega) = \int_{\{g_u(b_k)<u  < g_u(a_k)\}} |\nabla u|  \dx .
\end{align*}

We then find the following:
\begin{align*}
&\min_{t \in [t_0,t_1]}\I^*(V_\Omega(t))  \sum |g_u(a_k) - g_u(b_k)|  \leq \int_{\bigcup_k \{g_u(b_k) <u  < g_u(a_k)\}} |\nabla u| \dx \\
& \leq \min_{t \in [t_0,t_1]}(\I^*(V_\Omega(t))) \epsilon,
\end{align*}
where we have used \eqref{uniformlyIntegrableBound} and \eqref{ACEstimate}. This implies that $g_u$ is absolutely continuous on $[t_0,t_1]$, as claimed.

\end{proof}

Now we prove the main result of this section.

\begin{theorem} \label{1DRearrangement}
If $u \in W^{1,p}(\Omega)$ for $1 \leq p \leq \infty$ then $u^* \in W^{1,p}(\Omega^*)$ and furthermore:
\[
\int_I |g_u'|^p \I^*(V_\Omega)  \ds = \int_{\Omega^*} |\nabla u^*|^p \dy \leq \int_\Omega |\nabla u|^p \dx.
\]
\end{theorem}

\begin{proof}

Lemmas \ref{rearrangedTVLemma} and \ref{sobolevPreservationLemma} immediately give this inequality if $p = 1$. For $p>1$ we can still apply the previous lemmas to show that $u^* \in W^{1,1}(\Omega^*)$, because $\Omega$ has finite measure.

Next we note that the following equality holds (by using the coarea formula)
\begin{equation} \label{distributionFunctionRepresentation}
\varrho_u(t) = \mathcal{L}^n(\{u > t\} \cap \{\nabla u = 0\}) + \int_t^\infty \int_{\{u = s,\,\nabla u \neq 0\}} |\nabla u|^{-1} \, d\mathcal{H}^{n-1}  \ds =: f_1^u(t) + f_2^u(t) .
\end{equation}
Clearly $f_2^u$ is absolutely continuous, and $f_1^u$ is decreasing. Thus by the Lebesgue differentiation theorem (see Theorems 1.21 and 3.30 in \cite{LeoniBook}), $\varrho_u$ is differentiable for a.e.  $t$, with
\begin{equation} \label{distributionFunctionDerivativeInequality}
\varrho_u'(t) \leq -\int_{\{u = t, \nabla u \neq 0\}} |\nabla u|^{-1} \, d\mathcal{H}^{n-1} .
\end{equation}

Next we claim that (following \cite{FuscoCianchi}) for a.e. $t$,
\begin{equation} \label{h1Derivative}
 \frac{d}{dt} f_1^{u^*}(t)= \frac{d}{dt} \mathcal{L}^n(\{u^* > t,\,\nabla u^* = 0\}) = 0 .
\end{equation}
To establish this claim, we first note that for any open interval $J$ we have the following
\[
\mathcal{L}^1(g_u(J)) \leq \int_{J} |g_u'| \ds .
\]
By approximating measurable sets with disjoint open intervals we can then establish that
\[
\mathcal{L}^1(g_u( \{g'_u = 0\})) \leq \int_{\{g_u' = 0\}} |g_u'| \ds=0 .
\]
Following \cite{FuscoCianchiBVRearrangement} we then find that
\begin{align*}
\mathcal{L}^1( u^* (\{\nabla  u^* = 0\}))= \mathcal{L}^1(g_u( \{g'_u = 0\}))= 0 .
\end{align*}
Thus there exists a Borel set $F_0$ in $\mathbb{R}$ so that $\mathcal{L}^1(F_0) = 0$ and so that $u^*(\{\nabla u^* = 0\}) \subset F_0$. 

We then claim that for any Borel set $B$ in $\mathbb{R}$ we have that
\[
|Df_1^{u^*}|(B) = \mathcal{L}^n( (u^*)^{-1}(B)\cap \{\nabla u^* = 0\}) .
\]
To see this, we first note that $f_1^{u^*}$ is right continuous and decreasing. We then have that
\begin{align*}
&|Df_1^{u^*}|((t_1,t_2)) = f_1^{u^*}(t_1) - \lim_{t\to t_2^- }f_1^{u^*}(t_2) \\
&= \mathcal{L}^n(\{u^* > t_1,\, \nabla u^* = 0\}) - \lim_{t\to t_2^- }\mathcal{L}^n(\{u^* > t,\, \nabla u^* = 0\})\\
&= \mathcal{L}^n(\{u^* > t_1,\, \nabla u^* = 0\}) - \mathcal{L}^n(\{u^* \ge  t_2,\, \nabla u^* = 0\})\\
&= \mathcal{L}^n((u^*)^{-1}((t_1,t_2))\cap\{\nabla u^* =0\} ) .
\end{align*}
As both $|Df_1^{u^*}|$ and $\mathcal{L}^n(  (u^*)^{-1}(\cdot)\cap \{\nabla u^* = 0\})$ are Borel measures, and as they are equal on open intervals, they must be equal on all Borel sets. This and the fact that $u^*(\{\nabla u^* = 0\}) \subset F_0$ immediately give that
\[
|Df_1^{u^*}|(\mathbb{R} \backslash F_0) = \mathcal{L}^n((u^*)^{-1}(\mathbb{R} \backslash F_0)\cap\{\nabla u^* = 0\}  ) = \mathcal{L}^n(\emptyset) = 0,
\]
which proves \eqref{h1Derivative}. Utilizing \eqref{distributionFunctionRepresentation} this then immediately implies that for a.e. $t$,
\begin{equation}\label{derivativeFSubU}
\varrho_{u^*}'(t) =  -\int_{\{u^* = t, \nabla u^* \neq 0\}} |\nabla u^*|^{-1} \, d\mathcal{H}^{n-1} .
\end{equation}

By the coarea formula we can write the following:
\begin{align*}
\int_{\Omega^*} |\nabla u^*|^p \dy &= \int_{\Omega^*\cap \{\nabla u^* \neq 0\}} |\nabla u^*|^p \dy \\
&= \int_{\mathbb{R}} \int_{\{u^* = t,\, \nabla u^* \neq 0\}} |\nabla u^*|^{p-1} \,d\mathcal{H}^{n-1}  \dt .
\end{align*}
By \eqref{uStarDefinition} we know that $\nabla u^*(y) = (0,g_u'(y_n)) \in \mathbb{R}^{n-1}\times \mathbb{R}$. Since $g_u$ is decreasing we have that the set $\{u^* = t\}$ is a set of the form $\{y: \, y_n \in J,\, y' \in B_{n-1}(0,r(y_n))\}$, for some (possibly degenerate) interval $J$, with endpoints $t_1\le t_2$. If $t_1 = t_2$ then clearly $\nabla u^*$ is constant on the set $\{u^* = t\}$. If $t_1 \neq t_2$ then $g_u'$ is zero on the set $(t_1,t_2)$, and is either zero at $t_1,t_2$ or is undefined. Since $\nabla u^*$ is constant on level sets of $u^*$ (where it's defined) by Lemma \ref{Lem:CoareaTrick}, with $u^*$ in place of $u$, we can then write
\[
\int_{\Omega^*} |\nabla u^*|^p \dy =\int_{\mathbb{R}} \frac{\left(\mathcal{H}^{n-1}(\{u^* = t,\,\nabla u^* \neq 0\})\right)^p}{(\int_{\{u^* = t,\, \nabla u^* \neq 0\}} |\nabla u^*|^{-1} \,d\mathcal{H}^{n-1})^{p-1}}  \dt  .
\]
By \eqref{derivativeFSubU} we have that
\[
\int_{\Omega^*} |\nabla u^*|^p \dy =  \int_{\mathbb{R}} \frac{\operatorname*{P}(\{u^* > t\};\Omega^*)^p}{(-\varrho_{u^*}'(t))^{p-1}}  \dt .
\]
Next we utilize Lemma \ref{Equimeasurable} and Lemma \ref{RearrangedIsoperimetricInequality} to find that
\[
\int_{\Omega^*} |\nabla u^*|^p \dy \leq \int_{\mathbb{R}} \frac{\operatorname*{P}(\{u > t\};\Omega)^p}{(-\varrho_{u}'(t))^{p-1}}  \dt .
\]
Next \eqref{distributionFunctionDerivativeInequality} gives
\begin{align*}
\int_{\Omega^*} |\nabla u^*|^p \dy &\leq \int_{\mathbb{R}} \frac{\operatorname*{P}(\{u > t\};\Omega)^p}{(\int_{\{u = t,\, \nabla u \neq 0\}} |\nabla u|^{-1} \,d\mathcal{H}^{n-1})^{p-1}}  \dt \\
&= \int_{\mathbb{R}} \frac{\operatorname*{P}(\{u > t\};\Omega)^p}{(\mathcal{H}^{n-1}(\{u = t,\, \nabla u \neq 0\}))^{(p-1)}}\left(\int_{\{u = t,\, \nabla u \neq 0\}} |\nabla u|^{-1} \,d\mu_t\right)^{-(p-1)}\dt \\
&\leq  \int_{\mathbb{R}} \frac{\operatorname*{P}(\{u > t\};\Omega)^p}{(\mathcal{H}^{n-1}(\{u = t,\, \nabla u \neq 0\}))^{(p-1)}}\int_{\{u = t,\, \nabla u \neq 0\}} |\nabla u|^{(p-1)} \,d\mu_t\dt \\
&= \int_{\mathbb{R}} \int_{\{u = t\}} |\nabla u|^{p-1} \,d\mathcal{H}^{n-1}  \dt,
\end{align*}
where we use Jensen's inequality with $f(s) := s^{-(p-1)}$ and the probability measure
\[
\mu_t := \frac{d\mathcal{H}^{n-1}}{\mathcal{H}^{n-1}(\{u = t,\, \nabla u \neq 0\})}.
\] 
The result then follows after applying the coarea formula.

\end{proof}

\begin{remark} \label{increasingRearrangementRemark}
In this section we have considered a rearrangement of the function $u$, via the decreasing function $g_u: I \to \mathbb{R}$. However, all of the arguments would hold for an increasing rearrangement. Indeed, by utilizing \eqref{eqn:IStarSymmetric} and by Lemma \ref{Equimeasurable} for any Borel function $\psi:\mathbb{R} \to [0,\infty)$, the function $f_u(t) := g_u(-t)$ satisfies the following
\begin{align*}
\int_I \psi(f_u(t)) \I^*(V_\Omega(t))\dt &= \int_I \psi(g_u(t)) \I^*(V_\Omega(t)) \dt = \int_\Omega \psi(u) \dx,\\
\int_I |f_u'(t)|^p \I^*(V_\Omega(t))\dt &=\int_I |g_u'(t)|^p \I^*(V_\Omega(t)) \dt \leq \int_\Omega |\nabla u|^p \dx.
\end{align*}
We chose to work with the decreasing rearrangement in this section because that is the standard convention chosen in the literature involving rearrangement. However, we will work with the increasing rearrangement $f_u$ of $u$ in subsequent sections because we prefer to think of phase transitions as increasing functions.
\end{remark}

The following corollary is the motivation for our development of the rearrangement in this section and is a simple application of Lemma \ref{Equimeasurable} and Theorem \ref{1DRearrangement}.

\begin{corollary} \label{Cor:EnergyRearrangement}
Let $u \in H^1(\Omega)$. Then the following inequality holds:
\begin{equation}\label{eqn:weightedVersion}
\int_\Omega W(u) + \e^2 |\nabla u|^2 \dx \geq \int_I (W(f_u) + \e^2 (f_u')^2)\I^*(V_\Omega) \dt.
\end{equation}
Moreover 
\[
\int_\Omega u \dx = \int_I f_u \I^*(V_\Omega) \dt.
\]
\end{corollary}

\section{A 1D Functional Problem}

In light of Corollary \ref{Cor:EnergyRearrangement}, one possible avenue for studying the $\Gamma$-$\liminf$ of $\mathcal{F}_\e^{(1)}$ is to study the weighted, one-dimensional functional in \eqref{eqn:weightedVersion}. This was precisely the approach in \cite{DalMasoFonsecaLeoni}, where the radial case was studied. Because we do not have a specific form for $\I^*(V_\Omega)$, it is necessary for us to consider a much more general class of weights. The general weighted case, to our knowledge, has only been studied in \cite{kurata2001one}. They studied monotonicity properties of minimizers along curved strips in $\mathbb{R}^2$. Our work focuses on an entirely different question and applies to a wider class of weights.

We recall \eqref{DefI}, namely
\[
I := (-T,T).
\]
In this section only $T$ may be any positive number. We will assume that the weight $\eta$ satisfies the following:
\begin{align}
\eta &\in C^{1,\beta}(I), \quad \eta > 0 \text{ in } I,\label{etaSmooth}\\
d_1 (t+T)^{n_1-1} &\leq \eta(t) \leq d_2(t+T)^{n_1-1} \text{ for } t \in (-T,-T+t^*],\label{etaTail1}\\
d_3 (T-t)^{n_2-1} &\leq \eta(t) \leq d_4(T-t)^{n_2-1} \text{ for } t \in [T-t^*,T),\label{etaTail2}\\
|\eta'(t)| &\leq \frac{d_5 \eta(t)}{\min\{T-t,t+T\}}  \quad \text{ for } t\in I, \label{etaPrimeRatio}
\end{align}
for some constants $\beta \in (0,1]$, $d_1, \dots,  d_5>0$, $n_1,n_2 \in \mathbb{N}$ and $t^*>0$. These assumptions are naturally satisfied if $\eta(t) = \I^*(V_\Omega(t))$.

\begin{remark}
Two important weights are covered in our analysis. The unweighted case $\eta \equiv 1$  can be recovered by taking $n_1 = n_2 = 1$ and $d_i = 1$ for $i= 1,\dots,4$, while the radial weight $\eta (t)= (T+t)^{n-1}$ can be obtained by taking $n_1 = n$, $n_2 = 1$, $d_1 = d_2 = 1$ and appropriate $d_3$ and $d_4$. Both of these cases have been previously studied by various authors (see, e.g.,  \cite{Bellettini2013, NiethammerRadial}).
\end{remark}

By way of notation, we will write $L_\eta^p$ to be the space $L^p(I;\mathbb{R},\eta)$, where $p\geq 1$. We will also write $BV_\eta$ to be the space $BV(I;\mathbb{R}, \eta)$ with weight $\eta$, meaning that
\[
\|v\|_{BV_\eta} := \int_I |v(t)| \eta(t)  \dt + \int_I \eta(t) \, d|Dv|(t).
\]
For $v \in BV_\eta$ we will also write the weighted total variation of the derivative in the following manner
\[
|Dv|_\eta (E) = \int_E \eta(t) \, d|Dv|(t).
\]
We will write $H_\eta^1$ to be the analogous weighted version of $H^1$. We conduct our analysis in the weighted spaces because it is the natural setting for this variational problem.

In this section we study the functional
\begin{equation} \label{GDefinition}
G_\e(v) := \int_I (W(v) + \e^2 (v')^2)\eta \dt,\quad v \in H_\eta^1,
\end{equation}
subject to the constraint that
\begin{equation}\label{1DMassConstraint}
\int_I v \eta \dt =m \in \left(a\int_I \eta \dt,b\int_I \eta \dt\right).
\end{equation}
We extend $G_\e$ to $L_\eta^1$ by setting $G_\e(v): = \infty$ if $v \in L_\eta^1 \backslash H_\eta^1$ or if \eqref{1DMassConstraint} fails.

\subsection{Zero and First-Order $\Gamma$-limit of $G_\e$}

We begin by establishing the zeroth-order $\Gamma$-limit of the functional $G_\epsilon$.

\begin{theorem}\label{1DZeroGammaLimit}
Assume that $W$ satisfies hypotheses \eqref{W_Smooth}-\eqref{WGurtin_Assumption} and that $\eta$ satisfies hypotheses \eqref{etaSmooth}-\eqref{etaPrimeRatio}. Then the family $\{G_\e\}$ $\Gamma$-converges to $G^{(0)}$ in $L_\eta^1$, where
\begin{equation*}
G^{(0)}(v) := \begin{cases}
\int_I W(v) \eta \dt& \text{ if } v \in L_\eta^1 \text{ and } \int_I v \eta  \dt = m,\\
\infty& \text{ otherwise in  } L_\eta^1.
\end{cases}
\end{equation*}
\end{theorem}

\begin{proof}
For the $\liminf$ inequality assume that $v_\e \to v$ in $L_\eta^1$. By utilizing Fatou's lemma along with \eqref{W_Smooth} we have that
\[
\liminf_{\e \to 0^+} G_\e(v_\e) \geq \liminf_{\e \to 0^+} \int_I W(v_\e) \eta \dt \geq \int_I W(v) \eta  \dt.
\]
For the $\limsup$ inequality, we begin by assuming that $v$ is bounded and satisfies \eqref{1DMassConstraint} (the case where $v$ does not satisfy \eqref{1DMassConstraint} is trivial). Let $\phi_\delta$ be the standard mollifier, let $\tilde v$ be $v$ extended to all of $\mathbb{R}$ by zero and consider $\tilde v_\e := \phi_{\delta_\e}*\tilde v$, where we select $\delta_\e$ so that $\|v-\tilde v_\e\|_{L_\eta^1} = o(1)$ and so that
\[
\int_I (\tilde v_\e')^2 \eta  \dt \leq C\e^{-1}.
\]

We then select $d_\e \in \mathbb{R}$ so that $v_\e :=\tilde  v_\e + d_\e$ satisfies \eqref{1DMassConstraint}. It is evident that $d_\e = o(1)$. Finally, by the Lebesgue dominated convergence theorem we have that
\[
\lim_{\e \to 0^+} \int_I W(v_\e) \eta \dt = \int_I W(v)\eta  \dt,
\]
which gives the desired result for $v$ bounded. Now if $v \in L_\eta^1$ and $\int_I v \eta  \dt = m$ we can construct a sequence $\{v_k\}$ of truncations of $v$, so that $W(v_k) \leq W(v_{k+1})$ (see \eqref{W_Number_Zeros}) and so that $\int_I v_k \eta \dt = m$. Since the $\Gamma$-$\limsup$ is lower semicontinuous (see Proposition 6.8 in \cite{DalMasoBook}), by applying the Lebesgue monotone convergence theorem we have that
\begin{equation}\label{Eqn:GLimSupTrick}
\Gamma \text{-}\limsup G_\e(v) \leq \liminf_{k\to \infty} \Gamma \text{-}\limsup G_\e(v_k) \leq \liminf_{k \to \infty} \int_I W(v_k) \eta \dt = \int_I W(v) \eta \dt,
\end{equation}
which concludes the proof.

\end{proof}

Clearly we have that $\inf G^{(0)} = 0$, and thus
\begin{equation}\label{Def:G1}
G_\e^{(1)}(v) = \e^{-1} G_\e(v) = \int_I \left(\frac{W(v)}{\e} + \e |v'|^2 \right) \eta \dt
\end{equation}
for all $v \in H_\eta^1$ satisfying \eqref{1DMassConstraint}, and $G_\e^{(1)}(v) = \infty$ otherwise in $L_\eta^1$. We now state a compactness result, which utilizes arguments from \cite{FonsecaTartar1989}.

\begin{proposition} \label{1DCompactnessLemma}
Let $v_\e \in H^1_\eta$ be such that $\sup_\e G_\e^{(1)}(v_\e) < \infty$. Then up to a subsequence $v_\e \to v\in \mathcal{C}$ in $L^1_\eta$, where 
\begin{equation}
\mathcal{C} := \{v \in BV_\eta(I;\{a,b\}) : v \text{ satisfies } \eqref{1DMassConstraint}\} .\label{Def:C}
\end{equation}
\end{proposition}

\begin{proof}
We first show that $\{v_\e\}$ is uniformly bounded in $L_\eta^1$ and equi-integrable. This is since, by applying \eqref{WLinearGrowth},
\[
\int_{|v_\e|>\hat T} |v_\e| \eta \dt \leq \hat L^{-1} \int_I W(v_\e) \eta \dt \leq C \e G_\e^{(1)}(v_\e) \leq C\e,
\]
which, in turn, implies that
\[
\int_E |v_\e| \eta \dt \leq \hat T \int_E \eta \dt + C\e.
\]
As $\int_I \eta \dt < \infty$ and using the fact that any finite collections of $L_\eta^1$ functions in $L_\eta^1$ is equi-integrable, we obtain that the sequence $\{v_\e\}$ is bounded in $L_\eta^1$ and equi-integrable.

Next, define 
\begin{equation}\label{Eqn:W1Def}
W_1(s) := \min\{W(s), K\},\quad \Phi_1(t) := \int_a^t W_1^{1/2}(s) \ds,
\end{equation}
where $K :=\max_{s \in [a,b]} W(s)$. Using Young's inequality, and the fact that $0\leq W_1 \leq W$ we have that
\[
2\int_I W_1^{1/2}(v_\e) |v_\e'| \eta \dt \leq G_\e^{(1)} (v_\e) \leq C.
\]
Utilizing the chain rule, we find that
\[
\int_I |(\Phi_1 \circ v_\e)'|\eta \dt \leq C.
\]
Furthermore, as $\Phi_1$ is Lipshitz and $\Phi_1(a) = 0$, we have that $\Phi_1 \circ v_\e$ is uniformly bounded in $L_\eta^1$. This then implies, by BV compactness, that, up to a subsequence, not relabeled,
\begin{equation} \label{phiConvergence}
\Phi_1 \circ v_\e \to \tilde v\quad\text{in } L^1_\eta
\end{equation}
for some function $\tilde v\in BV_\eta$. It is easy to show, using \eqref{W_Number_Zeros}, that $\Phi_1$ has a continuous inverse. This implies that, up to a subsequence, $v_\e$ must converge pointwise to $v := \Phi_1^{-1}(\tilde v)$. Thus, up to a subsequence, the $v_\e$ converge in $L_\eta^1$ to $v$. Using Fatou's lemma and the fact that $\sup_\e G_\e^{(1)}(v_\e)<\infty$, it must be $W(v(t))=0$ for a.e. $t \in I$, or, in other words, that $v \in L_\eta^1(I;\{a,b\})$ by \eqref{W_Smooth}. As $\tilde v \in BV_\eta$, this implies that $v \in BV_\eta(I;\{a,b\})$. The $L^1_\eta$ convergence of the $v_\e$ then implies that $v$ satisfies \eqref{1DMassConstraint}. This concludes the proof.

\end{proof}

We now state the first main theorem of this section, which characterizes the first-order $\Gamma$-limit of $G_\e$.

\begin{theorem}\label{1DFirstGammaLimit}
Assume that $W$ satisfies \eqref{W_Smooth}-\eqref{WGurtin_Assumption} and that $\eta$ satisfies \eqref{etaSmooth}-\eqref{etaPrimeRatio}. Then the family $\{G_\e^{(1)}\}$ $\Gamma$-converges to the functional
\begin{equation}\label{G0Definition}
G^{(1)}(v) = \begin{cases}  \frac{2c_W}{b-a} |Dv|_\eta(I) &\text{ if } v \in \mathcal{C}, \\ 
\infty &\text{ otherwise in } L_\eta^1,
\end{cases}
\end{equation}
where $c_W$ is the constant given in \eqref{c0Definition} and $\mathcal{C}$ defined in \eqref{Def:C}.
\end{theorem}

We note that here $$|Dv|_\eta=(b-a) \sum \eta(t_i),$$ where $t_i$ are the locations of jumps of the function $v$. We also note that Proposition \ref{1DCompactnessLemma} and Theorem \ref{1DFirstGammaLimit} are completely analogous to classical results (e.\,g.\,\cite{Modica87,Sternberg88}) in the unweighted, higher-dimensional case.

\begin{proof}
We first characterize the $\Gamma$-$\limsup$. Specifically, given a $v \in \mathcal{C}$, we construct a family of functions $ v_\e$ that converge in $L^1_\eta$ to $v$ satisfying
\begin{equation} \label{firstOrderLimSup}
\limsup_{\e \to 0^+} G_\e^{(1)}( v_\e) \leq G^{(1)}(v) .
\end{equation}
To begin with, we assume that $v$ is of the form
\[
v(t) = \begin{cases} a &\text{ if } t \in [t_{2k},t_{2k+1}), \\
b &\text{ otherwise,}\end{cases}
\]
where $-T=t_0<t_1<\cdots<t_{2N}=T$. Define 
\begin{equation} \label{f200}
f(t) := \begin{cases} t-t_1 &\text{ if } t \in [t_0,t_1), \\ -\min\{t-t_{2k},t_{2k+1}-t\}  &\text{ if } t \in [t_{2k},t_{2k+1}),\text{ and } k\ge 1, \\ \min\{t-t_{2k+1},t_{2k+2}-t\} &\text { if }  t \in (t_{2k+1},t_{2k+2}], \text{ and } k < N-1, \\ t-t_{2N-1} &\text{ if } t \in [t_{2N-1}, t_{2N}). \end{cases}
\end{equation}
Observe that $f$ is the signed distance function (see \eqref{def:SignedDistance}) of the set $E:=\{t\in I:\,v(t)=a\}$, where we naturally are considering $\partial E$ relative to $I$, not $\mathbb{R}$. We note that $v(t) = \operatorname*{sgn}\nolimits_{a,b}(f(t))$, where $\operatorname*{sgn}\nolimits_{a,b}$ is the function given in \eqref{sgnabDefinition}. Thus our goal is to construct smooth approximations of the function $\operatorname*{sgn}\nolimits_{a,b}$ that make the energy $G_\e^{(1)}$ small.

We will follow the construction of \cite{Modica87}. Although the argument is almost identical, we include it for completeness. Consider the function
\begin{equation}\label{CandidateInverseProfile}
\Psi_\e(s) := \int_{a}^s \left(\frac{\e^2}{\e + W(r)} \right)^{1/2}  \,dr,
\end{equation}
and define the constant
\[
\xi_\e := \Psi_\e(b).
\]
We note that since $W \geq 0$, equation \eqref{CandidateInverseProfile} gives
\begin{equation} \label{xiBound}
0\leq \xi_\e \leq (b-a)\e^{1/2} .
\end{equation}

Note that $\Psi_\e$ is strictly increasing and differentiable. Now define $\phi_\e : [0,\xi_\e] \to [a,b]$ to be the inverse of $\Psi_\e$ on the interval $[a,b]$. By the fundamental theorem of calculus and the inverse function theorem $\phi_\e$ will satisfy the equation
\[
\e \phi_\e'(t) = (\e + W(\phi_\e(t)))^{1/2} .
\]
Next, extend $\phi_\e$ to be equal to $a$ for $t<0$ and $b$ for $t>\xi_\e$. Note that for all $t \in \mathbb{R}$ we have that $\phi_\e(t) \leq \operatorname*{sgn}\nolimits_{a,b}(t)$ and that $\phi_\e(t+\xi_\e) \geq \operatorname*{sgn}\nolimits_{a,b}(t)$. Thus as $v \in \mathcal{C}$ we can find a $\tau_\e \in (0,\xi_\e)$ that gives
\[
\int_I \phi_\e(f(t) +  \tau_\e) \eta(t)  \dt = m .
\]

Define $v_\e (t):= \phi_\e(f(t) + \tau_\e)$. As $\{v_\e\}$ converges to $v$ pointwise and $|v_\e|<C$ we have that $v_\e \to v$ in $L^1_\eta$. We then examine the energy associated with $v_\e$, when $\e$ is sufficiently small that transition layers do not overlap or leave $\overline{I}$:
\begin{align*}
G_\e^{(1)}(v_\e) &= \sum_{k=1}^{2N-1}\int_0^{\xi_\e} \left( \e (\phi_\e'(t))^2 + \e^{-1}W(\phi_\e(t)) \right)\eta(t_k+(t-\tau_\e)(-1)^{k+1})  \dt \\
&\leq  \sum_{k=1}^{2N-1}\int_0^{\xi_\e} 2(\e + W(\phi_\e(t)))^{1/2} \phi_\e'(t) \eta(t_k+(t-\tau_\e)(-1)^{k+1})  \dt \\
&\leq \sum_{k=1}^{2N-1}\sup \{\eta(t_k+(s-\tau_\e)(-1)^{k+1}):\, s \in (0,\xi_\e) \}\int_0^{\xi_\e} 2(\e + W(\phi_\e(t)))^{1/2} \phi_\e'(t) \dt \\
&= \sum_{k=1}^{2N-1}\sup \{\eta(t_k+(s-\tau_\e)(-1)^{k+1}):\, s \in (0,\xi_\e) \}\int_{a}^b  2(\e + W(s))^{1/2}  \ds .
\end{align*}
Thus taking the limit as $\e\to 0^+$ we find that
\[
\limsup_{\e \to 0^+} G_\e^{(1)}(v_\e) \leq 2c_W \sum_{k=1}^{2N-1} \eta(t_k) = G^{(1)}(v).
\]
The cases where $v$ has a finite number of jump points, but starting or ending at different values than we assumed are analogous. Reasoning as in \eqref{Eqn:GLimSupTrick}, by noting that functions with a finite number of jumps are dense in $\mathcal{C}$, and as the $\Gamma$-$\limsup$ is lower semicontinuous, we then have \eqref{firstOrderLimSup}.

Next we will establish our $\Gamma$-$\liminf$. Assume that $v_\e \to v$ in $L_\eta^1$. By Proposition \ref{1DCompactnessLemma} if $v \notin \mathcal{C}$ then $\liminf_{\e \to 0^+} G_\e^{(1)} = \infty$, and there is nothing to prove. We claim that for any sequence $\{v_\e\}$ that converges in $L_\eta^1$ to some $v \in \mathcal{C}$ the following inequality holds:
\begin{equation} \label{firstOrderLimInf}
\liminf_{\e \to 0^+} G_\e^{(1)}(v_\e) \geq G^{(1)}(v).
\end{equation}
To establish this inequality we use Young's inequality, the chain rule and lower semicontinuity of $\|\cdot\|_{BV_\eta}$  (see, e.g., \cite{Spector2011}) and the definition \eqref{Eqn:W1Def} as follows:
\begin{align*}
\liminf_{\e \to 0^+} G_\e^{(1)}(v_\e) &\geq \liminf_{\e \to 0^+}\int_I (\e^{-1} W_1(v_\e) + \e (v_\e')^2)\eta  \dt\\
&\geq \liminf_{\e \to 0^+} 2 \int_I |(\Phi_1 \circ v_\e)'| \eta  \dt \geq 2\int_I \eta \,d |D\Phi_1(v)| \\
&= 2\int_I \eta \,d |D\Phi(v)|  = \frac{2c_W}{b-a}\int_I \eta \,d |Dv| = G^{(1)}(v_0).
\end{align*}
Here we have used the fact that $\Phi_1 \circ v_\e$ converges to $\Phi_1 \circ v$ in $L_\eta^1$ (because $\Phi_1$ is Lipschitz), and the fact that $\Phi_1 \circ v = \Phi \circ v$, where $\Phi := \int_a^t W^{1/2}(s) \ds$. This proves the claim.
\end{proof}

Properties of $\Gamma$-limits \cite{DalMasoBook} along with Proposition \ref{1DCompactnessLemma} then establishes the following corollary.

\begin{corollary}
Under the hypotheses of Theorem \ref{1DFirstGammaLimit} if $v_\e$ are minimizers of $G_\e^{(1)}$ then, up to a subsequence, they converge in $L^1_\eta$ to $v$ which is a minimizer of $G^{(1)}$. Furthermore the $v_\e$ will satisfy the following
\begin{equation} \label{1DLimitEnergy}
\lim_{\e \to 0^+}G_\e^{(1)}(v_\e) = G^{(1)}(v).
\end{equation}
\end{corollary}

To conclude this subsection we prove two theorems that will be important later in our analysis. We select $t_0$ so that
\begin{equation}\label{Def:V0}
v_0(t):= \operatorname*{sgn}\nolimits_{a,b}(t-t_0)
\end{equation}
satisfies \eqref{1DMassConstraint}. By \eqref{etaSmooth} it is clear that $t_0$ is uniquely determined. We note that in general, $v_0$ is \emph{not a global minimizer} of $G^{(1)}$. However, we prove here that $v_0$ is an isolated local minimizer of $G^{(1)}$ in $L_\eta^1$.

\begin{theorem}
	\label{theorem local minimizer G0}Assume that $W$ satisfies \eqref{W_Smooth}-\eqref{WGurtin_Assumption} and that $\eta$ satisfies \eqref{etaSmooth}-\eqref{etaPrimeRatio}. Then there exists $\delta>0$ such that $v_{0}$ is
	an isolated $\delta$-local minimizer of $G^{(1)}$ in $L_{\eta}^{1}$, that is,
	there is no $v_{1}\in \mathcal{C}$ (see \eqref{Def:C}), with $0<\left\Vert v_{1}%
	-v_{0}\right\Vert _{L_{\eta}^{1}}\leq\delta$ such that
	\[
	G^{(1)}(v_{1})\leq G^{(1)}(v_{0}).
	\]
	
\end{theorem}

\begin{proof}
	Assume by contradiction that such $v_{1}$ exists. By continuity of $\eta$, for
	every $\epsilon>0$ there is $\r_{\epsilon}>0$ such that
	\begin{equation}
	|\eta(t)-\eta(t_{0})|\leq\epsilon\label{continuity eta}%
	\end{equation}
	for all $t\in\lbrack t_{0}-\r_{\epsilon},t_{0}+\r_{\epsilon}]$. Let $M_0 := \max |\eta'|+1$ and fix
	\begin{equation}
		0<\r_0<\min\left\{ \frac12 t^*,T-t_{0},T+t_{0},\frac{d_1 n_1\eta(t_{0})}{2 d_2 M_{0}},\frac{d_3 n_2\eta(t_{0})}{2 d_4 M_{0}}\right\},   \label{choice h0}%
	\end{equation}
	where $t^*,n_1,n_2$ and the constants $d_i, i=1 \dots 4$ are given in \eqref{etaTail1} and \eqref{etaTail2}. Then define
	\begin{equation}
	I_{0}:=[-T+\r_0,T-\r_0], \label{I0}%
	\end{equation}
	and fix 
	\begin{equation}\label{Def:eps1}
	0<\epsilon_{1}<\min\{\min_{I_{0}}\eta,\eta(t_{0})/2\}
	\end{equation}
	 in	\eqref{continuity eta} and let $\r_{\epsilon_1}$ be the corresponding $\r_{\epsilon}$.

	\noindent\textbf{Step 1: }We claim that $v_{1}$ has a jump at some $t_{1}\in
	B(t_0,\r_{\epsilon_1})$. If not, then either $v_{1}\equiv a$ in
	$B(t_0,\r_{\epsilon_1})$ or $v_{1}\equiv b$ in $B(t_0,\r_{\epsilon_1})$. Assume that $v_{1}\equiv a$ in $B(t_0,\r_{\epsilon_1})$.
	Then by \eqref{continuity eta},
	\[
	\delta\geq\int_{B(t_0,\r_{\epsilon_1})}|v_{1}-v_{0}|\eta~dt\geq(b-a)\frac{\eta(t_{0})}{2}%
	\r_{\epsilon_1},
	\]
	where we used the fact that $0<\epsilon_{1}<\eta(t_{0})/2$. Since the case $v_1 \equiv b$ gives an identical estimate, the claim follows provided
	\begin{equation}
	0<\delta<(b-a)\frac{\eta(t_{0})}{2}\r_{\epsilon_1}. \label{delta condition 1}%
	\end{equation}
	
	\noindent\textbf{Step 2: }We claim that $v_{1}$ has no jump other than $t_{1}$
	in $I_{0}$. Indeed, assume that there is a second jump $t_{2}\neq t_{1}$ in
	$I_{0}$. Then by \eqref{continuity eta} and Step 1,
	\begin{align*}
	G^{(1)}(v_{1})  &  \geq 2c_W(\eta(t_{1})+\eta(t_{2}))\\
	&  \geq 2c_W(\eta(t_{0})-\epsilon_{1}+\min_{I_{0}}\eta
	)>2c_W\eta(t_{0})=G^{(1)}(v_{0}),
	\end{align*}
	where in the last inequality we used the fact that $0<\epsilon_{1}<\min
	_{I_{0}}\eta$. This is impossible since we are assuming that $G^{(1)}%
	(v_{1})\le G^{(1)}(v_{0})$.
	
	\noindent\textbf{Step 3: } We claim that $v_1$ jumps from $a$ to $b$ at $t_1$. 	Suppose not, and suppose that $t_1 \leq t_0$. Then
	\[
	\delta\geq\int_{B(t_0,\r_{\epsilon_1})}|v_{1}-v_{0}|\eta~dt \geq  (b-a)\frac{\eta(t_{0})}{2}%
	\r_{\epsilon_1},
	\]
	which again leads to a contradiction if $\delta$ is chosen small enough. The case $t_1 > t_0$ is analogous.

	\noindent\textbf{Step 4: } We claim that $t_{1}=t_{0}$.  Indeed, if $t_{1}>t_{0}$, then
	\[
	0=\int_I (v_{1}-v_{0})\eta~dt=\int_{-T}^{-T+\r_0}(v_{1}-a)\eta
	~dt+\int_{t_{0}}^{t_{1}}(a-b)\eta~dt+\int_{T-\r_0}^{T}(v_{1}-b)\eta~dt,
	\]
	which implies, as the last two terms are negative, that there must be a jump $t_{3}$ that belongs to $(-T,-T+\r_0)$, with
	\begin{equation}
	0<\frac{\eta(t_{0})}{2}(b-a)(t_{1}-t_{0})\leq\int_{t_{0}}^{t_{1}}%
	(b-a)\eta~dt\leq(b-a)\int_{-T}^{t_{3}}\eta~dt \leq d_{2}(b-a)\frac
	{(T+t_{3})^{n_1}}{n_1}, \label{ineq 100}%
	\end{equation}
	where in the last equality we used \eqref{etaTail1}, in conjunction with \eqref{choice h0}. By the mean
	value theorem and inequality \eqref{ineq 100}, for some $\theta \in (t_0,t_1)$,
	\begin{align*}
	\eta(t_{1})  &  =\eta(t_{0})+\eta^{\prime}(\theta)(t_{1}-t_{0})\geq\eta
	(t_{0})-M_{0}|t_{1}-t_{0}|\\
	&  \geq\eta(t_{0})-\frac{2M_{0}d_{2}}{n_1\eta(t_{0})}(T+t_{3})^{n_1}.
	\end{align*}
	Hence by \eqref{choice h0},%
	\begin{align*}
	G^{(1)}(v_{1})  &  \geq2c_W(\eta(t_{1})+\eta(t_{3}))\\
	&  \geq2c_W\eta(t_{0})-2c_W\frac{2M_{0}d_{2}%
	}{n_1\eta(t_{0})}(T+t_{3})^{n_1}+2c_Wd_{1}(T+t_{3})^{n_1-1}\\
	&  >2c_W\eta(t_{0})=G^{(1)}(v_{0}),
	\end{align*}
	which violates our assumption. The case  $t_1<t_0$ is analogous. This proves that $t_1=t_0$, and so $G^{(1)}(v_{1}) \ge 2c_W\eta(t_0)=G^{(1)}(v_{0})$, which implies that $G^{(1)}(v_{1})=G^{(1)}(v_{0})$. In particular, $v_1$ 	has no jumps in $I \backslash I_0$. But then $v_1 = v_0$, which is a contradiction. This completes the proof.
\end{proof}

We have seen in Theorem \ref{theorem local minimizer G0} that $v_0$ is a local minimizer for $G^{(1)}$. In general $v_0$ will not be a global minimizer without further assumptions on $\eta$ (e.g., $\eta \equiv$ constant). However, for the applications in the $n$-dimensional case later on it will be important to study a type of second-order asymptotic development of $G_\e$ where in the definition of $G_\e^{(2)}$ (see \eqref{higherOrderFunctionalDefinition}) in place of $\inf G^{(1)}$ we take $G^{(1)}(v_0)$. We will see that this corresponds to studying the second-order asymptotic development of the localized functional
\begin{equation} \label{def:JEps}
J_\e(v) := \begin{cases}
G_\e(v) &\text{if } \|v-v_0\|_{L_\eta^1}\leq \delta, \\
\infty &\text{ otherwise}.
\end{cases}
\end{equation}
When we apply the following theorem in $n$-dimensions we will need slightly weaker assumptions on $\eta$, and thus this theorem differs in its assumptions.

\begin{theorem} \label{1DLimsupOrder2}
	Assume that $W$ satisfies \eqref{W_Smooth}-\eqref{WGurtin_Assumption}, and that $\eta:I \to [0,\infty)$ is measurable, bounded, differentiable at $t_0$, $\eta(t_0)>0$ and
	\begin{equation}\label{eqn:limsupEstimate6}
	|\eta(t) - \eta(t_0) - \eta'(t_0)(t-t_0)| \leq C|t-t_0|^{1+\beta}
	\end{equation}
	for some constant $C>0$ and for all $t$ in a neighborhood of $t_0$. Then there exists a sequence $\{v_\e\}$ converging to $v_0$ in $L_\eta^1$ so that 
	\begin{equation}\label{eqn:limsup}
	\begin{aligned}
	\limsup_{\e \to 0^+} \frac{G_\e^{(1)}(v_\e) - 2c_W \eta(t_0)}{\e} &\leq 2\eta'(t_0)(  \tau_0c_W + c_{\operatorname*{sym}}) \\
	&\quad+ \begin{cases}
	\frac{\lambda_0^2}{2W''(a)} \int_I \eta  \ds \quad &\text{ if } q = 1, \\
	0 &\text{ if } q<1,
	\end{cases}
	\end{aligned}
	\end{equation}
	where $c_W$ and $c_{\operatorname*{sym}}$ are given by \eqref{c0Definition}, \eqref{c1Definition}, $\tau_0$ is determined by the equation
	\begin{equation}\label{limSupDelta0Definition}
	\eta(t_0) \int_\mathbb{R} (z(s-\tau_0) - \operatorname*{sgn}\nolimits_{a,b})  \ds = \begin{cases}
	 \frac{\lambda_0}{W''(a)} \int_I \eta  \dt &\text{ if } q =1, \\
	0 &\text{ if } q<1 ,
	\end{cases}
	\end{equation}
	where $z$ is the solution to \eqref{profileCauchyProblem} and $\lambda_0$ is defined by
	\begin{equation} \label{def:rhoNot}
	\lambda_0 := \frac{2\eta'(t_0) c_W}{(b-a)\eta(t_0)}.
	\end{equation}
\end{theorem}

\begin{proof}
\textbf{ Step 1: }Assume $q = 1$. Define $z_\e(t) := z(\frac{t-t_0}{\e})$ and then define
	\begin{equation}\label{eqn:RecoverySequenceDefinition}
	v_\e(t) := z_\e(t-\e\tau_\e) - \frac{\lambda_0 \e}{W''(a)},
	\end{equation}
	where $\tau_\e$ is selected so that \eqref{1DMassConstraint} is satisfied. We first claim that 
	\begin{equation} \label{Eqn:LimTau}
	\lim_{\e\to 0^+} \tau_\e = \tau_0.
	\end{equation}
	 To this end, we can write, via \eqref{1DMassConstraint},
	\[
	\int_I v_\e \eta \dt = \int_I v_0 \eta \dt = m.
	\]
	In turn this implies that
	\begin{equation} \label{eqn:LimsupMass}
	\begin{aligned}
	\int_I (z_\e(t-\e \tau_\e) - z_\e(t-\e \tau_0)) \eta \dt &= \int_I (\operatorname*{sgn}\nolimits_{a,b}(t-t_0)-z_\e(t-\e\tau_0)) \eta \dt \\
	&\quad+  \frac{\e \lambda_0}{W''(a)} \int_I \eta \dt.
	\end{aligned}
	\end{equation}
	After the change of variables $s = \frac{t-t_0}{\e}$ we can write the right-hand side as
	\begin{equation}\label{eqn:limsupEstimate0}
	\e \int_{\frac{-T-t_0}{\e}}^{\frac{T-t_0}{\e}} (\operatorname*{sgn}\nolimits_{a,b}(s) - z(s-\tau_0)) \eta(\e s+t_0) \ds + \frac{\e \lambda_0}{W''(a)} \int_I \eta \dt.
	\end{equation}
	By our choice of $\tau_0$ (via \eqref{limSupDelta0Definition}) and \eqref{sgnabDefinition} this is equal to
	\begin{equation} \label{eqn:limsupDeltaEstimate} \begin{aligned}
	&\e \int_{\frac{-T-t_0}{\e}}^{\frac{T-t_0}{\e}} (\operatorname*{sgn}\nolimits_{a,b}(s) - z(s-\tau_0)) (\eta(\e s+t_0)-\eta(t_0)) \ds \\
	&-  \e\eta(t_0)\int_{-\infty}^{\frac{-T-t_0}{\e}} (a - z(s-\tau_0)) \ds - \e\eta(t_0)\int_{\frac{T-t_0}{\e}}^\infty (b - z(s-\tau_0))\ds. \end{aligned}
	\end{equation}
	By \eqref{eqn:limsupEstimate6} there exists a $R_0>0$ such that $|\eta(t) - \eta(t_0)| \leq (|\eta'(t_0)|+1)|t-t_0|$ for all $t \in B(t_0,R_0)$. Since $\eta$ is bounded by assumption, we thus have for all $t \in I\backslash B(t_0,R_0)$,
	\[
	|\eta(t)-\eta(t_0)| \leq 2\|\eta\|_\infty \leq 2\frac{\|\eta\|_\infty}{R_0} |t-t_0|.
	\]
	Hence for all $t \in I$ we have that $|\eta(t)-\eta(t_0)| \leq C_\eta|t-t_0|$ for some $C_\eta>0$. Thus, using \eqref{CPDecay2}, the first term in \eqref{eqn:limsupDeltaEstimate} can be bounded by
\begin{equation}\label{eqn:limsupEstimate1}
	2(b-a) \e \int_{\frac{-T-t_0}{\e}}^{\frac{T-t_0}{\e}} e^{-c_1 |s|} |\eta(\e s+t_0)-\eta(t_0)| \ds \leq 2(b-a) C_\eta \e^2 \int_\mathbb{R} e^{-c_1 |s|} |s| \ds.
\end{equation}
  By \eqref{CPDecay2} we know that the last two terms of \eqref{eqn:limsupDeltaEstimate} are bounded from above by $\frac{(b-a)}{c_1}\|\eta\|_\infty \e^2 e^{-\frac{c_1T_1}{\e}}$, where $T_1 := \min(T-t_0,T+t_0) > 0$.  Hence, the right-hand side of \eqref{eqn:LimsupMass} is bounded from above by $C\e^2$ for all $\e>0$ sufficiently small.

Now assume that the $\tau_\e$ do not converge to $\tau_0$. Assume without loss of generality that for some subsequence (not relabeled) the $\tau_\e \leq \tau_0 - k_0$ for some $k_0>0$ (the case where $\tau_\e \geq \tau_0 + k_0$ is similar). Since $z$ is increasing (see \eqref{profileCauchyProblem}), by \eqref{eqn:LimsupMass} and what we just proved, 
	\begin{align*}
	C\e^2&\geq\int_I (z_\e(t-\e \tau_\e) - z_\e(t-\e \tau_0)) \eta(t) \dt \geq \inf_{B(t_0+\e \tau_0,k_1\e)} \eta \int_{B(t_0 +\e \tau_0, k_1\e)} \int_{t-\e \tau_0}^{t-\e \tau_\e} z_\e'(s) \ds \dt \\
	&\geq \inf_{B(t_0+\e \tau_0,k_1\e)} \eta \int_{B(t_0 +\e \tau_0, k_1\e)} \int_{t-\e \tau_0}^{t-\e (\tau_0-k_0)} \e^{-1}\sqrt{W(z(\e^{-1}(s-t_0))}\ds \dt \\
	&\geq  2k_1k_0 \e \inf_{ t\in B(0,k_1+k_0)} \sqrt{W(z(t))} \inf_{B(t_0+\e \tau_0,k_1\e)} \eta,
	\end{align*}
where $0<k_1<1$ and where we have used the facts that $\eta$ is continuous at $t_0$ and that $\eta(t_0)>0$. Since $z(0) = c$, by taking $k_0$ and $k_1$ sufficiently small we can assume that $z(t) \in B(c,\min \{\frac{c-a}{2},\frac{b-c}{2}\})$ for all $t \in B(0,k_0+k_1)$. In turn the right-hand side of the previous inequality is bounded from below by $C_1 \e$ for some $C_1>0$. This is a contradiction, which proves our claim.

Next we prove \eqref{eqn:limsup}. We will write $R_\e := C_k\e |\log \e|$, with $C_k>0$ to be chosen later. We then write
\begin{equation}\label{eqn:limsupEstimate3}
\begin{aligned}
\frac{G_\e^{(1)}(v_\e) - 2c_W \eta(t_0)}{\e} &= \e^{-1}\left(\int_{B(t_0,R_\e)} (\e^{-1} W(v_\e) + \e(v_\e')^2) \eta \dt - 2c_W \eta(t_0)\right) \\
&\quad+ \int_{I \backslash B(t_0,R_\e)} (\e^{-2} W(v_\e) + (v_\e')^2) \eta \dt.
\end{aligned}
\end{equation}
First we examine the second term, namely the tail integral. We first note that by \eqref{CPDecay2} and the fact that the $\tau_\e\to \tau_0$ we then have that
\[
b-z_\e(t-\e \tau_\e) \leq \frac{b-a}{2}e^{c_1(1+|\tau_0|)}\e^{c_1C_k} \leq \e^k
\]
for $t \in [t_0+R_\e,T]$ and for $\e$ small, provided $C_k \geq 2\frac{k}{c_1}$. Similarly, $z_\e(t-\e\tau_\e)-a < \e^k$ for $t \in [-T,t_0 - R_\e]$. Thus for $t \in I\backslash B(t_0,R_\e)$ we have that
\begin{equation}\label{eqn:limsupEstimate5}
|z_\e(t-\e \tau_\e) - v_0(t)| \leq \e^k
\end{equation}
which in turn implies, after recalling \eqref{eqn:RecoverySequenceDefinition}, that, for $k$ large,
\begin{equation}\label{eqn:tailEstimate2}
(v_\e(t)-v_0)^2 \leq \frac{\lambda_0^2 \e^2}{W''(a)^2} +  C \e^{k+1}
\end{equation}
for all $t \in I\backslash B(t_0,R_\e)$ and for some fixed $C >0$.

We then fix $\gamma>0$. By \eqref{W_Limits} there exists $s_\gamma$ such that
\begin{equation}\label{eqn:WellEstimate1}
W(s) \leq \left(\frac{W''(a)}{2} + \gamma\right) (s-a)^2
\end{equation}
for all $s$ with $|s-a| \leq s_\gamma$, and
\begin{equation}\label{eqn:WellEstimate2}
W(s) \leq \left(\frac{W''(a)}{2} + \gamma\right) (s-b)^2
\end{equation}
for all $s$ with $|s-b| \leq s_\gamma$. By \eqref{eqn:tailEstimate2}, \eqref{eqn:WellEstimate1} and \eqref{eqn:WellEstimate2} we then have for $\e$ sufficiently small that
\begin{equation}
\int_{I\backslash B(t_0,R_\e)} W(v_\e) \eta \dt \leq \left(\frac{W''(a)}{2} + \gamma\right)\e^2\lambda_0^2 W''(a)^{-2} \int_I \eta \dt + O(\e^{k+1}).
\end{equation}

On the other hand, using \eqref{profileCauchyProblem}, \eqref{eqn:limsupEstimate5}, \eqref{eqn:WellEstimate1}, and \eqref{eqn:WellEstimate2},
\[
(v_\e'(t))^2 = \frac{1}{\e^2}W(z_\e(t + \e \tau_\e)) \leq \frac{C}{\e^2}(z_\e(t  + \e \tau_\e) - v_0(t))^2 \leq C\e^{2k-2}
\]
for $t \in I \backslash B(t_0,R_\e)$. After taking limits (first as $\e \to 0^+$ and then as $\gamma \to 0^+$) we thus find that
\begin{equation}\label{eqn:limsupEstimate9}
\limsup_{\e \to 0^+}\int_{I \backslash B(t_0,R_\e)} (\e^{-2} W(v_\e) + (v_\e')^2) \eta \dt \leq \frac{\lambda_0^2}{2W''(a)} \int_I \eta  \dt.
\end{equation}

Next we estimate the energy in the region $B(t_0,R_\e)$. We will define $s_1^\e := v_\e(t_0 - R_\e)$ and $s_2^\e := v_\e(t_0+R_\e)$. Note that by \eqref{eqn:tailEstimate2}, $s_1^\e = a + O(\e)$ and $s_2^\e = b+O(\e)$. Thus recalling the definition of $c_W$, \eqref{c0Definition}, and \eqref{W_Limits},  we find that
	\[
	c_W =  \int_{s_1^\e}^{s_2^\e} W^{1/2}(s) \ds + O(\e^2) = \int_{B(t_0,R_\e)} W^{1/2}(v_\e) v_\e'  \dt + O(\e^2),
	\]
	where we have used the change of variables $s = v_\e(t)$. Thus we have that
	\begin{equation} \label{116}
	\begin{aligned}
	&\int_{B(t_0,R_\e)}(\e^{-1}W(v_\e) + \e  (v_\e')^2)\eta  \dt - 2c_W \eta(t_0) \\
	&= \int_{B(t_0,R_\e)} (\e^{-1/2}W^{1/2}(v_\e) - \e^{1/2}v_\e')^2\eta  + W^{1/2}(v_\e) v_\e' (2\eta - 2\eta(t_0))  \dt + O(\e^2).
	\end{aligned}
	\end{equation}	
   We now estimate the terms on the right-hand side of \eqref{116}. Recalling the fact that $|W^{1/2}(s_1) - W^{1/2}(s_2)| \leq C|s_1-s_2|$ for all $s_1,s_2 \in [a-1,b+1]$ (see \eqref{W_Smooth} and \eqref{WPrime_At_Wells}), it follows from \eqref{profileCauchyProblem}, \eqref{eqn:RecoverySequenceDefinition}, and the boundedness of $\eta$, that
	\begin{align}
	\int_{B(t_0,R_\e)} (\e^{-1/2}W^{1/2}(v_\e)-\e^{1/2}v_\e')^2\eta \dt &\leq \e^{-1}\int_{B(t_0,R_\e)} (W^{1/2}(v_\e(t)) - W^{1/2}(z_\e(t-\e \tau_\e)))^2 \eta(t)  \dt \nonumber\\
	&\leq C \e^{-1}\int_{B(t_0,R_\e)} \left(\frac{\e \lambda_0}{ W''(a)}\right)^{2} \eta  \dt \leq C \e^2 |\log \e| .\label{radialTransitionCost1}
	\end{align}
	Next we will use \eqref{profileCauchyProblem}, \eqref{eqn:limsupEstimate6} and \eqref{eqn:RecoverySequenceDefinition} to obtain:
	\begin{align*}
	&2\int_{B(t_0,R_\e)} W^{1/2}(v_\e) v_\e' (\eta - \eta(t_0)) \dt \\
	&= 2\int_{B(t_0,R_\e)} W^{1/2}(v_\e(t)) v_\e'(t) (\eta'(t_0) (t-t_0) + O(|t-t_0|^{1+\beta}))  \dt\\
	&= 2\eta'(t_0) \int_{B(t_0,R_\e)} W^{1/2}(v_\e(t)) v_\e'(t) (t-t_0)  \dt + O(\e^{1+\beta} |\log\e|^{2+\beta}) .
	\end{align*}
	Changing variables to $s = \frac{t-t_0 - \e \tau_\e}{\e}$ we can then write
	\begin{align}
	&2\int_{B(t_0,R_\e)} W^{1/2}(v_\e) v_\e' (\eta - \eta(t_0))  \dt \nonumber \\
	&= 2\eta'(t_0) \e \int_{B(\tau_\e,C_k |\log \e|)} W^{1/2}(z(s)-\lambda_0 W''(a)^{-1} \e) z'(s) (\tau_\e+s) \ds + O(\e^{1+\beta} |\log \e|^{2+\beta}). \label{radialTransitionCost2.1}
	\end{align}
	We remark that, by \eqref{c0Definition} and \eqref{c1Definition} and \eqref{Eqn:LimTau},  the integral on the right-hand side of the previous equality converges to
	\[
	\int_{\mathbb{R}} W^{1/2}(z(s)) z'(s) (\tau_0+s) \ds = \tau_0 c_W + c_{sym}.
	\]
	By then combining estimates \eqref{eqn:limsupEstimate3}, \eqref{eqn:limsupEstimate9}, \eqref{116},  \eqref{radialTransitionCost1}, \eqref{radialTransitionCost2.1}, to find that
	\begin{align*}
	\limsup_{\e \to 0^+} \frac{G_\e^{(1)}(v_\e) - 2c_W\eta(t_0)}{\e} &\leq 2\eta'(t_0)\left(\tau_0c_W + c_{\operatorname*{sym}}\right) +  \frac{\lambda_0^2}{2W''(a)} \int_I \eta  \dt ,
	\end{align*}
	which is the desired conclusion.
	
	\textbf{Step 2: } The case $q<1$ is simpler since by \eqref{CPFiniteWidth} the function $z $ in \eqref{profileCauchyProblem} satisfies $z(t) \equiv b$ for $t \geq t_b$ and $z(t) \equiv a$ for $t \leq t_a$. We define $v_\e(t) := z_\e(t-\e \tau_\e)$. Then the second term in the right-hand side of \eqref{eqn:LimsupMass} should be replaced by $0$, while \eqref{eqn:limsupEstimate0} becomes
	\[
	\e\int_{t_a+\tau_0}^{t_b + \tau_0} (\operatorname*{sgn}\nolimits_{a,b}(s) - z(s-\tau_0)) \eta(\e s + t_0) \ds.
	\]
	In turn, in \eqref{eqn:limsupDeltaEstimate} the first integral is over $[t_a + \tau_0,t_b+\tau_0]$, while the other two integrals vanish. Using the regularity of $\eta$ near $t_0$ we can bound the integral in the new \eqref{eqn:limsupDeltaEstimate} by $2(b-a)C_\eta \e^2(t_b-t_a)$. We can continue as before to conclude that $\tau_\e \to \tau_0$.
	
	By \eqref{c0Definition} and \eqref{profileCauchyProblem}, in place of \eqref{eqn:limsupEstimate3} we now have
	\[
	\frac{G_\e^{(1)}(v_\e)-2c_W \eta(t_0)}{\e} = \e^{-1}\int_{t_0 + \e \tau_\e + \e t_a}^{t_0 + \e \tau_\e + \e t_b}W^{1/2}(v_\e(t)) v_\e'(t) (\eta(t) - \eta(t_0)) \dt.
	\]
	Using \eqref{eqn:limsupEstimate6} and the fact that $\tau_\e \to \tau_0$, the right-hand side can be bounded from above by
	\begin{align*}
	&\leq 2\e^{-1} \eta'(t_0)  \int_{t_0 + \e \tau_\e + \e t_a}^{t_0 + \e \tau_\e + \e t_b}W^{1/2}(v_\e(t)) v_\e'(t) (t-t_0) \dt + O(\e^{\beta}) \\
	&= 2 \eta'(t_0) \int_{t_a}^{t_b} W^{1/2}(z(s)) z'(s)(s + \tau_\e) \ds + O(\e^{\beta}),
	\end{align*}
	where we have used a change of variables $s = \frac{t-t_0-\e \tau_\e}{\e}$. It now suffices to let $\e \to 0^+$.

\end{proof}

\subsection{Local Minimizers of $G_\e$}
In this subsection we prove the existence of certain types of local minimizers of $G_\e$ and study their qualitative properties. In the next subsection these properties will enable us to characterize the second-order asymptotic development of the family $J_\e$ defined in \eqref{def:JEps}. 
We begin with the following proposition, which is based on an argument from \cite{KohnSternberg} (see also \cite{BraidesLocalMinNotes}). We include the proof for completeness.

\begin{proposition} \label{JLocalMinimizers}
	Assume that $W$ satisfies \eqref{W_Smooth}-\eqref{WGurtin_Assumption} and that $\eta$ satisfies \eqref{etaSmooth}-\eqref{etaPrimeRatio}. Then for all $\e >0$ there exists a global minimizer $v_\e$ of the functional $J_\e$. Furthermore, the functions $v_\e$ must converge to $v_0$ in $L_\eta^1$, and thus for $\e$ small enough $v_\e$ is a local minimizer of $G_\e$. Additionally, the following equality holds:
	\begin{equation} \label{eqn:JepsLimitEnergy}
	\lim_{\e \to 0^+} J_\e^{(1)}(v_\e) = G^{(1)}(v_0).
	\end{equation}
\end{proposition}

\begin{proof}
	First we prove the existence of a global minimizer. Fix $\e > 0$ and suppose that $\{f_k\}$ is a minimizing sequence in the sense that
	\begin{equation}\label{GminimizingSequence}
	\lim_{k \to \infty} J_\e(f_k) = \inf_v J_\e(v) < \infty.
	\end{equation}
	In particular, $\|f_k-v_0\|_{L_\eta^1} \leq \delta$ for all $k$ sufficiently large. By \eqref{Def:G1} and \eqref{def:JEps} it follows that $\{f_k'\}$ is bounded in $L_\eta^2$. Since $\{f_k\}$ is bounded in $L_\eta^1$, by \eqref{etaSmooth} and a diagonal argument, we may find a function $v_\e \in H_{\eta,\operatorname*{loc}}^1$ such that $f_k' \rightharpoonup v_\e'$ in $L_\eta^2$ and $f_k \to v_\e$ in $L_{\eta,\operatorname*{loc}}^1$, and pointwise a.e.. By Fatou's lemma and the weak lower semi-continuity of the $L_\eta^2$ norm, we then have, provided that $v_\e \in H_\eta^1$ (see \eqref{GDefinition}), that
	\[
	G_\e(v_\e) \leq \liminf_{k \to \infty} G_\e(f_k) = \inf_v J_\e(v)
	\]
	and that $\|v_\e - v_0\|_{L_\eta^1} \leq \delta$. Thus it remains to show that $v_\e \in L_\eta^2$. Since $v_\e$ is locally absolutely continuous, by H\"older's inequality, for $-T<t< -T+t^*$ we have
	\begin{align}
	v_\e^2(t) \eta(t) &= \eta(t) \left(v_\e(-T+t^*) - \int_t^{-T+t^*} v_\e'(s) \ds \right)^2 \\
	&\leq 2\eta(t) v_\e^2(-T+t^*) + 2\eta(t) \left( \int_t^{-T+t^*}v_\e'(s) \frac{\eta^{1/2}(s)}{\eta^{1/2}(s)}\ds \right)^2 \\
	&\leq 2\eta(t) v_\e^2(-T+t^*) + 2\eta(t) \int_t^{-T+t^*} \frac{1}{\eta(s)}\ds \int_t^{-T+t^*} |v_\e'(s)|^2 \eta(s) \ds \\
	&\leq 2\eta(t) v_\e^2(-T+t^*) + 2 \frac{d_2}{d_1}t^* \int_I |v_\e'(s)|^2 \eta(s) \ds,
	\end{align}
	where we have used the fact that if $t<s<-T+t^*$ then $\eta(s) \geq \frac{d_1}{d_2}\eta(t)$ (see \eqref{etaTail1}). By integrating in $t$ over $(-T,-T+t^*)$ we observe that $v_\e \in L_\eta^2((-T,-T+t^*))$. A similar estimate can be obtained on the interval $(T-t^*,T)$. On the other hand, by \eqref{etaSmooth}, we have that $\eta \geq \eta_0 >0$ in $[-T+t^*,T-t^*]$, and thus $v_\e \in L^2((-T+t^*,T-t^*))$, which then implies that $v_\e \in L_\eta^2$, as desired. This establishes the existence of a global minimizer, $v_\e$.

%

 By Theorem \ref{1DFirstGammaLimit} we know that there exists a sequence $\{\tilde v_\e\}$ converging to $v_0$ in $L_\eta^1$ with $G_\e^{(1)}(\tilde v_\e) \to G^{(1)}(v_0)$. In particular $\|\tilde v_\e -v_0\|_{L_\eta^1} \leq \delta$ for $\e$ sufficiently small. Since $v_\e$ is a global minimizer of $J_\e$ we then know that $G_\e(v_\e) \leq G_\e(\tilde v_\e)$ for $\e$ small. Thus
\begin{equation}
\limsup_{\e \to 0^+} G^{(1)}_\e(v_\e) \leq \limsup_{\e \to 0^+} G_\e^{(1)}(\tilde v_\e) \leq G^{(1)}(v_0). \nonumber
\end{equation}
By Proposition \ref{1DCompactnessLemma} we then have that (up to a subsequence, not relabeled), $v_\e \to \tilde v$ in $L_\eta^1$, with $\tilde v \in \mathcal{C}$ and with $\|\tilde v - v_0\|_{L_\eta^1} \leq \delta$. By again applying Theorem \ref{1DFirstGammaLimit} we find that
\begin{equation}  \label{eqn:IsolatedMin}
G^{(1)}(\tilde v) \leq \liminf_{\e \to 0^+} G^{(1)}_\e(v_\e) \leq \limsup_{\e \to 0^+} G^{(1)}_\e(v_\e) \leq G^{(1)}(v_0).
\end{equation}
Theorem \ref{theorem local minimizer G0} then implies that $\tilde v = v_0$, which along with \eqref{eqn:IsolatedMin} implies \eqref{eqn:JepsLimitEnergy}. As $v_\e \to v_0$ in $L_\eta^1$ we then have that the $v_\e$ must be local minimizers of $G_\e$, for $\e$ sufficiently small. This completes the proof.
\end{proof}

In light of the fact that the global minimizers of $J_\e$ are local minimizers of $G_\e$ for $\e$ sufficiently small we can then establish the Euler--Lagrange equations.

\begin{theorem} \label{1DBVP}
	Under the hypotheses of Proposition \ref{JLocalMinimizers} the sequence $\{v_\e\}$ of global minimizers of the functionals $J_\e$ will satisfy the following Euler--Lagrange equations (for $\e$ sufficiently small):
	\begin{equation} \label{1DEulerLagrange}
	2\e^2(v_\e'(t) \eta(t))' - W'(v_\e(t)) \eta(t) = \e \lambda_\e \eta(t),
	\end{equation}
	where $\lambda_\e \in \mathbb{R}$. Moreover the Lagrange multipliers $\lambda_\e$ satisfy
	\begin{equation} \label{1DMultiplierLimit}
	\lim_{\e\to 0^+}\lambda_\e=  \lambda_0, 
	\end{equation}
	where $\lambda_0$ is the number given in \eqref{def:rhoNot}.
\end{theorem}

\begin{proof}
	Reasoning somewhat as in the proof of step 4 in \cite{DalMasoFonsecaLeoni} we have that $v_\e \in C^2(I)$ and satisfies \eqref{1DEulerLagrange}. Next, we will prove \eqref{1DMultiplierLimit}, namely the limit of the Lagrange multipliers $\lambda_\e$. The argument here follows \cite{LuckhausModica}, with the necessary adaptations to the weighted setting.
	
	To prove \eqref{1DMultiplierLimit}, fix some $\psi \in C_c^\infty (I)$. We multiply the Euler--Lagrange equations \eqref{1DEulerLagrange} by $\psi v_\e'$ and integrate to obtain
	\[
	\e \lambda_\e \int_I \psi v_\e' \eta \dt = \int_I (2\e^2(v_\e''\eta + v_\e' \eta') - W'(v_\e) \eta) \psi v_\e'  \dt.
	\]
	Integrating by parts, we find that
	\begin{equation} \label{1DIntegratedForm1}
	\e \lambda_\e \int_I \psi v_\e' \eta \dt = \int_I (W(v_\e)-\e^2 v_\e'^2)(\eta \psi)' + 2\e^2 (v_\e')^2 \eta' \psi  \dt .
	\end{equation}
	By Theorem \ref{1DFirstGammaLimit} and Proposition \ref{JLocalMinimizers}  we know that
	\[
	\lim_{\e \to 0^+} \int_I (\e^{-1}W(v_\e) + \e (v_\e')^2) \eta  \dt = 2 c_W \eta(t_0).
	\]
Furthermore, as in the proof of \eqref{firstOrderLimInf}, by lower semicontinuity
	\begin{equation}\label{100lower}
	\liminf_{\e \to 0^+} 2\int_I W^{1/2}(v_\e) |v_\e'| \eta \dt = \liminf_{\e \to 0^+} 2\int_I |(\Phi(v_\e))'| \eta  \dt   \geq  2 c_W \eta(t_0) ,
	\end{equation}
	where we recall that $\Phi(t) := \int_a^t W^{1/2}(s) \ds$. 
	These together give the following:
	\begin{align*}
	0 &\leq \limsup_{\e \to 0^+} \int_I (\e^{-1/2}W^{1/2}(v_\e) - \e^{1/2} (v_\e'))^2 \eta  \dt \\
	& = \limsup_{\e \to 0^+} \int_I (\e^{-1}W(v_\e) + \e (v_\e')^2 - 2W^{1/2}(v_\e)|v_\e'|)\eta  \dt \leq 0 .
	\end{align*}
	We thus have that $\e^{-1/2}W^{1/2}(v_\e) - \e^{1/2} |v_\e'|$ goes to zero in $L^2_\eta$. Moreover, the liminf in \eqref{100lower} is actually a limit and equality holds, so that
	\begin{equation}\label{101lower}
	\lim_{\e \to 0^+} \int_I W^{1/2}(v_\e) |v_\e'| \eta \dt    =   c_W \eta(t_0) .
		\end{equation}

	Additionally, we can write the following:
	\begin{align*}
	&\lim_{\e \to 0^+} \int_I |\e^{-1}W(v_\e) - \e (v_\e')^2|\eta  \dt \\
	&= \lim_{\e \to 0^+} \int_I \left|\e^{-1/2}W^{1/2}(v_\e) - \e^{1/2} |v_\e'|\right|\left|\e^{-1/2}W^{1/2}(v_\e) + \e^{1/2} |v_\e'|\right|\eta  \dt \\
	& \leq \lim_{\e \to 0^+}\left(\int_I\left(\e^{-1/2}W^{1/2}(v_\e) - \e^{1/2} |v_\e'|\right)^2\eta  \dt\right)^{1/2}\\
	&\quad\times \left(\int_I\left(\e^{-1/2}W^{1/2}(v_\e) + \e^{1/2} |v_\e'|\right)^2\eta  \dt\right)^{1/2} \\
	&\leq \lim_{\e \to 0^+}C\left(\int_I\left(\e^{-1/2}W^{1/2}(v_\e) - \e^{1/2} |v_\e'|\right)^2\eta  \dt\right)^{1/2} =0,
	\end{align*}
	where we have used H\"{o}lder's inequality in the first inequality, Young's inequality and the boundedness of $G_\e^{(1)}(v_\e)$ in the second. By \eqref{etaSmooth} we can deduce that $\e^{-1}W(v_\e) - \e (v_\e')^2$ goes to zero in $L^1_{\operatorname*{loc}}(I)$. Thus by dividing \eqref{1DIntegratedForm1} by $\e$, and recalling that $\psi$ is compactly supported in $I$, we obtain
	\[
	\lim_{\e \to 0^+} \lambda_\e \int_I \psi v_\e' \eta \dt =  \lim_{\e \to 0^+} 2\int_I \e (v_\e')^2 \eta' \psi  \dt .
	\]
	We then use the $L^2$ convergence shown above to estimate the following
	\begin{align*}
	&\lim_{\e \to 0^+} \left|\int_I (\e (v_\e')^2- W^{1/2}(v_\e)|v_\e'|) \eta' \psi  \dt\right| \\
	&= \lim_{\e \to 0^+} \left|\int_I \e^{1/2} |v_\e'|(\e^{1/2}|v_\e'| - \e^{-1/2}W^{1/2}(v_\e) )\eta'\psi  \dt\right|\\
	&\leq \lim_{\e \to 0^+}\left(\int_I \e (v_\e')^2 \left(\frac{\eta'\psi}{\eta}\right)^2 \eta \dt\right)^{1/2} \left(\int_I (\e^{1/2}|v_\e'| - \e^{-1/2}W^{1/2}(v_\e) )^2  \eta \dt\right)^{1/2} = 0 ,
	\end{align*}
	where we have used the fact that $\frac{\psi \eta'}{\eta}$ is uniformly bounded, since $\psi$ has compact support in $I$.
	
	Thus we can write the following:
	\begin{equation}\label{multiplier0}
	\lim_{\e \to 0^+} \lambda_\e \int_I \psi v_\e' \eta \dt =  \lim_{\e \to 0^+} 2\int_I  W^{1/2}(v_\e)|v_\e'| \eta' \psi   \dt.
	\end{equation}
We know that $ v_\e'\mathcal{L}^1\lfloor I\stackrel{*}{\rightharpoonup} Dv_0=(b-a)\delta_{t_0}$ and $ W^{1/2}(v_\e) v_\e'\mathcal{L}^1\lfloor I \stackrel{*}{\rightharpoonup} D(\Phi \circ v_0)=c_W\delta_{t_0}$, both in $(C_0(\overline{I}))'$. In turn, 
	$ W^{1/2}(v_\e) v_\e'\eta\mathcal{L}^1\lfloor I \stackrel{*}{\rightharpoonup} c_W\eta(t_0)\delta_{t_0}$. In view of \eqref{101lower}, it follows from   Proposition 4.30 in \cite{MaggiBook} that $ W^{1/2}(v_\e) |v_\e'|\eta\mathcal{L}^1\lfloor I \stackrel{*}{\rightharpoonup} c_W\eta(t_0)\delta_{t_0}$. 
	Hence,
	\[
	\lim_{\e \to 0^+} \int_I  W^{1/2}(v_\e)|v_\e'| \eta' \psi   \dt=\lim_{\e \to 0^+} \int_I  W^{1/2}(v_\e)|v_\e'|\eta \frac{\eta'}{\eta} \psi   \dt=c_W\eta(t_0)\frac{\eta'(t_0)}{\eta(t_0)} \psi(t_0).
	\] 
We thus take limits in \eqref{multiplier0}  to find that
	\[
	\lim_{\e \to 0^+} \lambda_\e (b-a) \psi(t_0) \eta(t_0) = 2 \eta'(t_0) c_W \psi(t_0).
	\]
	This then gives the desired conclusion, namely that \eqref{1DMultiplierLimit} holds.
	
\end{proof}

Next we establish tight bounds on the functions $v_\e$, as well as a Neumann condition.
	
	\begin{theorem}\label{Thm:1DBounds}
	Under the hypotheses of Proposition \ref{JLocalMinimizers}, for all $\e>0$ sufficiently small the minimizers $v_\e$ of $J_\e$ satisfy
	\begin{align}
	a_\e &\leq v_\e(t) \leq b_\e, \quad t \in I, \label{1DTightBounds}\\
	v_\e'(-T) &= v_\e'(T) = 0\label{1DNeumannCondition},
	\end{align}
	where $a_\e < c_\e<b_\e$ are the only zeros of $W' + \lambda_\e \e$. Moreover 
	\begin{align}
    a_\e &= a - \lambda_\e |\lambda_\e|^{1/q-1}  (q/\ell)^{1/q}\e^{1/q}+ o(\e^{1/q}), \label{aEpsilonForm}\\
    c_\e &= c - \lambda_\e W''(c)^{-1} \e + o(\e), \label{cEpsilonForm}\\
	b_\e &= b - \lambda_\e  |\lambda_\e|^{1/q-1}(q/\ell)^{1/q} \e^{1/q} + o(\e^{1/q}) \label{bEpsilonForm},
\end{align}
where $\ell$ is given in \eqref{WPrime_At_Wells}.
	\end{theorem}
	
	\begin{proof}
	By hypothesis \eqref{WGurtin_Assumption}, $|W'(s)| \geq w_0 > 0$ for all $|s| \geq  C$. Since $W'$ has only three zeros at $a,b,c$ and is strictly monotonic in a ball centered at each of these points with radius $\zeta_0 > 0$ (see \eqref{WPrime_At_Wells} and \eqref{W_Number_Zeros}), by taking $w_0$ smaller we can assume that $|W'(s)| \geq w_0$ for all $s \in \mathbb{R}\setminus( B(a,\zeta_0)\cup B(c,\zeta_0)\cup B(b,\zeta_0))$. By \eqref{1DMultiplierLimit}, $|\e \lambda_\e| \leq w_0/2$ for all $\e>0$ small. Hence $W'+\e \lambda_\e$ has only three zeros
\begin{equation} \label{80}
a_\e<b_\e<c_\e,
\end{equation}
for all $\e>0$ small. Furthermore by \eqref{W_Number_Zeros} and \eqref{WPrimeLimits} we can derive the explicit forms in \eqref{aEpsilonForm}-\eqref{bEpsilonForm}.

Next, consider the open set $U_\e := \{t \in I : v_\e(t) < a_\e\}$. We claim that $U_\e$ is empty. Indeed, if not, let $I_\e$ be a maximal subinterval of $U_\e$, and since $W'(v_\e) + \e \lambda_\e < 0$ for all $t \in I_\e$ by \eqref{1DEulerLagrange} we have that $(v_\e'(t) \eta(t))'<0$ for all $t \in I_\e$. Since $\eta>0$ on $I$ by \eqref{etaSmooth}, this implies that $v_\e'$ has at most one zero in $\overline{I_\e}$. Hence there exist $\lim_{t \to t_\e^+}v_\e(t)= \ell_\e$ and $\lim_{t \to T_\e^-}v_\e(t) = L_\e$, where $t_\e,T_\e$ are the left and right endpoints of $I_\e$, respectively. Note that $\ell_\e, L_\e$ could be infinite if one of the endpoints is $-T$ or $T$. Consider $\inf_{I_\e} v_\e$. If there exists $s_\e \in I_\e^\circ$ such that $v_\e(s_\e) = \inf_{I_\e} v_\e$, then $v_\e'(s_\e) = 0$ and $v_\e''(s_\e) \geq 0$. This is impossible, as $(v_\e' \eta)' < 0$ on $I_\e$. Thus it follows that $\inf_{I_\e} v_\e$ is either $\ell_\e$ or $L_\e$. Assume first that $\inf_{I_\e} v_\e = \ell_\e$. By the definition of $I_\e$ it cannot be that $\ell_\e = a_\e$,  but then, by the maximality of $I_\e$, necessarily $t_\e = -T$. By \eqref{1DEulerLagrange} for all $t_1,t_2 \in I_\e$, with $t_1 < t_2$:
\begin{equation}\label{Eqn:ELIntegrated}
2\e^2 v_\e'(t_2)\eta(t_2) - 2\e^2 v_\e'(t_1)\eta(t_1 ) = \int_{t_1}^{t_2} (W'(v_\e(s)) + \e \lambda_\e)\eta(s) \ds.
\end{equation}
Since $W'(v_\e(t)) + \e \lambda_\e < 0$ for all $t \in I_\e$, the integral $\int_{-T}^{t_2} (W'(v_\e(s)) + \e \lambda_\e) \eta(s) \ds$ is well-defined in $\mathbb{R} \cup \{-\infty\}$. Hence, letting $t_1 \to -T^+$ in \eqref{Eqn:ELIntegrated}, it follows that there exists
\begin{equation}\label{Eqn:LimitVPrime}
\lim_{t \to -T^+} v_\e'(t) \eta(t) = M_\e \in \mathbb{R} \cup \{\infty\}.
\end{equation}
Assume, for the sake of contradiction, that $M_\e \neq 0$. Then by \eqref{etaTail1} and \eqref{Eqn:LimitVPrime}, $|v_\e'(t)| \geq C_0 (T+t)^{-n_1+1}$ for all $t \in (-T,-T+\delta_\e)$, for some $\delta_\e > 0$. It would then follow that
\[
\int_{-T}^{-T+\delta_\e} |v_\e'|^2 \eta \dt \geq d_1 \int_{-T}^{-T+\delta_\e} C_0^2 (T+t)^{-n_1+1} \dt = \infty
\]
if $n_1 \geq 2$. On the other hand, if $n_1 = 1$ then $v_\e'(-T) = 0$, since $v_\e$ is a minimizer. Thus in both cases we must have that $M_\e = 0$. In turn, letting $t_1 \to -T^+$ in \eqref{Eqn:ELIntegrated} it follows that $v'_\e(t) < 0$ for all $t \in I_\e$, which contradicts the fact that $\ell_\e = \inf_{I_\e} v_\e$. Using a similar argument we can exclude the case that $L_\e = \inf_{I_\e} v_\e$. This proves that $I_\e$, and in turn $U_\e$, is empty. Thus $v_\e \geq a_\e$ in $I$. Similarly, we can show that $v_\e \leq b_\e$ in $I$.

It remains to prove the Neumann boundary condition \eqref{1DNeumannCondition}. If $n_i = 1$ then this comes from the minimality of $v_\e$. When $n_i \geq 2$, since $v_\e$ is bounded by what we just proved, it follows that the integral on the right-hand side of \eqref{Eqn:ELIntegrated} is bounded for all $t \in I$. Hence as in the first part of the proof we can conclude that the limit $M_\e$ in \eqref{Eqn:LimitVPrime} exists and must be zero. Hence letting $t_1 \to -T^+$ in \eqref{Eqn:ELIntegrated} we obtain
\[
2\e^2v_\e'(t)\eta(t) = \int_{-T}^t (W'(v_\e) + \lambda_\e \e) \eta(s) \ds.
\]
Using again the fact that $v_\e$ is bounded, along with \eqref{W_Smooth} and \eqref{etaTail1}, we have that 
\[
0 \leq 2\e^2 |v_\e'(t)| \leq \frac{C}{d_1(T+t)^{n_1-1}} \int_{-T}^t d_2(T+s)^{n_1-1} \ds = \frac{Cd_2}{d_1 n_1}(T+t) \to 0
\]
as $t \to -T^+$. A similar estimate holds near $T$. This completes the proof.
	\end{proof}

In the following theorem we specify the qualitative behavior of $v_\e$, which are global minimizers of $J_\e$. Despite the fact that $v_\e \to v_0 \in L_\eta^1$ by Proposition \ref{JLocalMinimizers}, $v_\e$ need not be increasing. Indeed in the radial case $\eta (t) \equiv (t+T)^{n-1}$, on an unbounded domain and for $n$ large, Ni \cite{Ni1983} has shown that all positive solutions of \eqref{1DEulerLagrange} approach $b_\e$ as $t \to \infty$ in an oscillatory way. The presence of possible oscillations makes the analysis significantly more involved. However, the overall idea of the proof is the same as the proof of Theorem \ref{theorem local minimizer G0}.

Fix 
\begin{equation}\label{def: theta i}
\theta_i \in \left(\frac{1}{n_i},\frac{1}{n_i-1}\right),\quad i = 1,2,
\end{equation}
where $n_i$ are the exponents given in \eqref{etaTail1} and \eqref{etaTail2}. Let $k\in\mathbb{N}$ and define
\begin{equation}
O_{\e}:=\{t\in [-T+c(n_1)\e^{\theta_1},T-c(n_2)\e^{\theta_2}]:~a_{\e}+\e^{k}\leq
v_{\e}(t)\leq b_{\e}-\e^{k}\},
\label{set O epsilon}%
\end{equation}
with $c(n_i): = 0$ if $n_i = 1$ and $1$ otherwise.

\begin{theorem}\label{thm:1DMinProperties}
Assume that $W$ satisfies \eqref{W_Smooth}-\eqref{WGurtin_Assumption}, and that $\eta$ satisfies \eqref{etaSmooth}-\eqref{etaPrimeRatio}.  Let $v_\e$ be a minimizer of $J_\e$. Write $I_0 := [-T+\r_0, T-\r_0]$, with $\r_0>0$ a constant to be defined. Then for $\delta$ sufficiently small in \eqref{def:JEps} and for all $\e>0$ sufficiently small the following properties hold:
\begin{enumerate}
\item $\Gamma_\e := O_\e \cap I_0$ has exactly one component $[T_1^\e,T_2^\e]$, with $v_\e(T_1^\e)=a_\e + \e^k$ and $v_\e(T_2^\e)=b_\e -\e^k$. Moreover, there exists $0<\r_1<\r_0$ so that $\Gamma_\e \subset B(t_0,\r_1)$.
\item For every fixed $\e$, the points in $\Gamma_\e$ where $v_\e = c_\e$ are at most distance $C\e$ apart, for some $C>0$ independent of $\e$.
\item For $t\in(-T,T_1^\e)$ we have that $v_\e(t) \in [a_\e, a_\e + \e^k)$ except on a set of $\eta \mathcal{L}^1$ measure $o(\e)$. Similarly for $t\in(T_2^\e,T)$ we have that  $v_\e(t) \in (b_\e-\e^k, b_\e]$ except on a set of $\eta \mathcal{L}^1$ measure $o(\e)$.
\end{enumerate}
\end{theorem}

We delay the proof of this theorem until after we establish some preliminary results. Let $\r_0>0$ be chosen as in \eqref{choice h0}. As $v_{\e}\rightarrow v_{0}$ in  $L^1_\eta$, by
selecting a subsequence, we can assume that $v_{\e}(t)\rightarrow
v_{0}(t)$ for $\mathcal{L}^{1}$ a.e. $t\in I$. Hence, given 
\begin{equation}\label{rho-chosen}
0<\rho<\frac12\min\{c-a,b-c\},
\end{equation} there exists $\e_\rho>0$ such that 
\begin{equation}
|v_{\e}(T_{1})-a|<\rho,\quad|v_{\e}(T_{2})-a|<\rho,\quad|v_{\e}(T_{3})-b|<\rho,\quad|v_{\e}(T_{4}%
)-b|<\rho \label{limits}%
\end{equation}
for all $0<\e\leq\e_\rho$ sufficiently small and some $T_1 \in (-T,-T+\r_0)$, $T_2 \in (-T+2\r_0,t_0-\r_0)$, $T_3 \in (t_0+\r_0,T-2\r_0)$ and $T_4 \in (T-\r_0,T)$. Fix $\e>0$ sufficiently small so that \eqref{limits} holds.

First, we prove adaptations of two lemmas from \cite{SternbergZumbrunBarrier}.

\begin{lemma}
	\label{lemma c epsilon}Let $s_{0},s_{1}>0$ be such that $a_{\e}%
	+s_{0}<c_{\e}<b_{\e}-s_{1}$ for all $\e>0$
	sufficiently small. Fix any such $\e$. Let $I_\e\subseteq I$ be a non-empty
	maximal interval such that $a_{\e}+s_{0}<v_{\e
	}(t)<b_{\e}-s_{1}$ for all $t\in I_\e$. Then there exists $t_\e\in\overline{I_\e}$ such
	that $v_{\e}(t_\e)=c_{\e}$.
\end{lemma}

\begin{proof}
	If not, then either $a_{\e}+s_{0}\leq v_{\e}(t)<c_{\e}$
	for all $t\in\overline{I_\e}$ or $c_{\e}<v_{\e}%
	(t)\leq b_{\e}-s_{1}$ for all $t\in\overline{I_\e}$. Consider the second
	case. Then $W^{\prime}(v_{\e}(t))+\e\lambda_{\e}<0$
	for all $t\in I_\e$, and so by \eqref{1DEulerLagrange} we have that $(v_{\e
	}^{\prime}(t)\eta(t))^{\prime}<0$ for all $t\in I_\e$. Let $\tilde t\in\overline{I_\e}$
	be the point of minimum of $v_{\e}$ in $\overline{I_\e}$. Reasoning as
	in the proof of \eqref{1DTightBounds}, we have that $\tilde t$ cannot belong
	to $I_\e$, and so $\tilde t\in\partial I_\e$. If $\tilde t\in I$, then necessarily,
	$v_{\e}(\tilde t)=c_{\e}$, which contradicts the fact that
	$c_{\e}<v_{\e}(t)<b_{\e}-s_{1}$ for all
	$t\in\overline{I_\e}$. it follows that $\tilde t\in\{-T,T\}$. We can now continue as
	in the proof of \eqref{1DTightBounds} to exclude this possibility.
\end{proof}

\begin{lemma}\label{BarrierLemma}
	Let $\rho$ be as in \eqref{rho-chosen} and suppose that $I_\e$ is a maximal subinterval of the set $\{t\in [-T+c(n_1)\e^{\theta_1},T-c(n_2)\e^{\theta_2}]:\,  v_\e(t) \geq c+\rho\}$. Then there exists a $\mu > 0$ such that we have the following estimate for all $t \in I_\e$:
	\[
	b_{\e}-v_{\e}(t)\leq2(b_{\e}-c-\rho)e^{-\mu d(t,I_\e^c)\e^{-1}}.%
	\]
	In addition an analogous bound holds for the set $\{t \in [-T+c(n_1)\e^{\theta_1},T-c(n_2)\e^{\theta_2}]:\, v_\e(t) \leq c-\rho\}$.
\end{lemma}

We recall that $d(t,E)$ is the distance from $t$ to the set $E$ and $E^c$ is the complement of $E$ (see Section \ref{notationSection}).

\begin{proof}

	First, we claim that there exists a $\mu$ such that for any $s \in [c+\rho,b_\e]$ the following inequality holds
	\begin{equation} \label{eqn:WPrimeLinearized}
		-(W'(s) + \e \lambda_\e) \geq 2 \mu^2 (b_\e - s).
	\end{equation}
	If $q = 1$ in \eqref{WPrime_At_Wells}, then also by \eqref{W_Smooth} we have that $W \in C^2(\mathbb{R})$. Since $W''(b) > 0$ by continuity we have that $W''(s) \geq 2\mu^2 > 0$ for all $s \in B(b,R_1)$, for some $\mu \neq 0$, and $R_1 >0$. It follows from \eqref{80} that
	\[
	W'(s) + \e \lambda_\e = -\int_s^{b_\e} W''(r) \,dr \leq -2\mu^2 (b_\e -s)
	\]
	for all $s \in B(b,R_1)$, with $s < b_\e$. Using the fact that $W'+\e\lambda_\e < 0$ in $(c_\e,b_\e)$ (see Theorem \ref{Thm:1DBounds}), and by taking $\mu$ smaller, if necessary, we can assume that
	\[
	W'(s) +\e \lambda_\e \leq -2\mu^2 (b_\e - s)
	\]
	for all $s \in [c+\rho,b_\e]$. Note that $\mu$ depends upon $\rho$ but not on $\e$. On the other hand, if $0<q<1$ then since $\lim_{s \to b} W''(s) = \infty$ by \eqref{WPrime_At_Wells}, we can still assume that $W''(s) \geq \mu^2 >0$ near $b$. Hence we can continue as before to conclude that \eqref{eqn:WPrimeLinearized} holds even in this case. This proves the claim.

	Write $I_\e = [t_1,t_2]$  and define
	\begin{equation} \label{phiBarrierDefinition}
	\phi(t) := (b_\e - v_\e(t_1))e^{-\mu (t-t_1)\e^{-1}} + (b_\e - v_\e(t_2))e^{-\mu (t_2-t)\e^{-1}}
	\end{equation}
	with $\mu$ fixed by \eqref{eqn:WPrimeLinearized}. We note that $\phi$ satisfies the following differential inequality:
	\begin{align*}
	(\phi' \eta) ' &= \frac{\mu^2}{\e^2} \phi \eta + \frac{\mu}{\e}\eta' \left(-(b_\e - v_\e(t_1))e^{-\mu (t-t_1)\e^{-1}} + (b_\e - v_\e(t_2))e^{-\mu (t_2-t)\e^{-1}}\right) \\
	&\leq \frac{1}{\e^2}\left(\mu^2 + \e\frac{ |\eta'|}{\eta}\mu\right) \phi \eta.
	\end{align*}
	If $n_1>1$ in \eqref{etaTail1}, then $c(n_1)=1$ in \eqref{set O epsilon} and so by \eqref{etaPrimeRatio},
	\[\e\frac{|\eta'(t)|}{\eta(t)}\leq \frac{\e d_5}{t+T}
	\leq d_5\e^{1-\theta_1}\leq \mu
		\]
	for all $t\in [-T+\e^{\theta_1},0]$ and all $\e$ sufficiently small. On the other hand, if $n_1=1$ in \eqref{etaTail1}, then $c(n_1)=0$ in \eqref{set O epsilon} and so by  \eqref{etaSmooth}  and \eqref{etaTail2}, $\eta(t)\ge \eta_0>0$ for all $t\in [-T,0]$. Thus, \[\e\frac{|\eta'(t)|}{\eta(t)}\leq \e\frac{\max|\eta'|}{\eta_0}\leq \mu
	\]
	for all $t\in [-T,0]$  and all $\e$ sufficiently small. 	Similar inequalities hold in $[0,T-c(n_2)\e^{\theta_2}]$. Thus in $I_\e$, 
	\begin{equation}\label{eqn:Barrier1}
	(\phi' \eta) ' \leq 2\e^{-2}\mu^2  \phi \eta.
	\end{equation}
	We then set $g(t): = b_\e - v_\e(t)$ and using \eqref{1DEulerLagrange} and \eqref{eqn:WPrimeLinearized} we have that
	\begin{equation} \label{eqn:Barrier2}
	(g' \eta)' = -\e^{-2}(W'(v_\e) + \e \lambda_\e) \eta \geq 2\e^{-2}\mu^2 g \eta.
	\end{equation}
	We define $\Psi := g-\phi$. By \eqref{phiBarrierDefinition}, \eqref{eqn:Barrier1} and \eqref{eqn:Barrier2}, for $\e$ small we have the following:
	\begin{align*}
	&(\Psi' \eta)' \geq 2 \e^{-2}\mu^2 \Psi \eta, \\
	&\Psi(t_1)\leq  0,\quad \Psi(t_2) \leq  0.
	\end{align*}
	The maximum principle implies that $\Psi\leq 0$ for all $t \in I_\e$. Thus
	\begin{equation}\label{Eqn:preciseDecayBarrier}
	b_\e - v_\e(t)\leq (b_\e - v_\e(t_1))e^{-\mu (t-t_1)\e^{-1}} + (b_\e - v_\e(t_2))e^{-\mu (t_2-t)\e^{-1}} \leq 2(b_\e - c-\rho)) e^{-\mu \e^{-1} d(t,I_\e^c)},
	\end{equation}
	which is the desired result.
\end{proof}

\begin{corollary}
	\label{corollary diameter} Let $\rho$ be as in \eqref{rho-chosen} and let
	\begin{align}
	A_{\e} &  :=\{t\in[-T+c(n_1)\e^{\theta_1},T-c(n_2)\e^{\theta_2
	}]:~a_{\e}+\e^{k}\leq v_{\e}(t)\leq c-\rho
	\},\label{set A epsilon}\\
	B_{\e} &  :=\{t\in [-T+c(n_1)\e^{\theta_1},T-c(n_2)\e^{\theta_2
	}]:~c+\rho\leq v_{\e}(t)\leq b_{\e}-\e^{k}%
	\}.\label{set B epsilon}%
	\end{align}
	Then for any maximal interval $I_\e$ contained in
	$A_{\e}\cup B_{\e}$,
	\begin{equation}
	\operatorname*{diam}I_{\e}\leq C\e|\log\e|\label{diam A epsilon}%
	\end{equation}
for
all $\e>0$ sufficiently small and 	for some constant $C>0$ depending only on $W$, $k$, $\mu$, $\rho$, where $\mu$ is given in Lemma \ref{BarrierLemma}.
\end{corollary}

\begin{proof}
	Assume $(t_1,t_2) = I_{\e}^\circ \subset B_{\e}$. By Lemma \ref{BarrierLemma} we have that for $t = \frac{t_1+t_2}{2}$:
	\[
	\e^{k}\leq b_{\e}-v_{\e}(t) \leq2(b_{\e
	}-c-\rho)e^{-\mu 2^{-1}(t_2-t_1)\e^{-1}},
	\]
	which implies that $-\frac{\mu}{2}(t_2-t_1)\e^{-1}\geq k\log\e
	-\log 2(b_{\e}-c-\rho)$, that is,
	\[
	0\leq t_2-t_1\leq2\mu^{-1}k\e|\log\e|+2\mu^{-1}\e
	\log 2(b_{\e}-c-\rho).
	\]
	This shows that $\operatorname*{diam}I_{\e}\leq C\e
	|\log\e|$. The proof for the case $I_{\e}\subset
	A_{\e}$ is similar, and we omit it.
\end{proof}

Next we state a lemma from \cite{SternbergZumbrunBarrier}, which allows us to estimate the size of certain sets. In what follows given a set $E$ and $s > 0$ we define the set
\begin{equation} \label{defSetWidening}
E^s := \{ x \in \mathbb{R}^n : d(x,E) \leq s \}
\end{equation}

\begin{lemma} \label{SetSizeLemma}
Given a measurable set $A \subset \mathbb{R}^n$, for all numbers $0<s_1<s_2$ we have that
\begin{equation}
\frac{\mathcal{L}^n(A^{s_2})}{\mathcal{L}^n(A^{s_1})} \leq C_n \left(\frac{s_2}{s_1}\right)^n,
\end{equation}
where we are using the notation \eqref{defSetWidening}.
\end{lemma}

Next we establish an estimate on the derivative of $v_\e$.

\begin{lemma}\label{DerivativeBounds}
There exists a constant $C>0$ such that
	\[
	|v_{\e}^{\prime}(t)|\leq C\e^{-1}%
	\]
	for all $t\in I$.
\end{lemma}

\begin{proof}
	By \eqref{1DEulerLagrange} and the fact that $v_{\e}^{\prime}(-T)=0$,
	\begin{equation}
	2\e^{2}v_{\e}^{\prime}(t)\eta(t)=\int_{-T}^{t}(W^{\prime
	}(v_{\e}(s))+\e\lambda_{\e})\eta(s
	)~ds\label{integral identity 1}%
	\end{equation}
	for every $t\in\overline{I}$. In light of \eqref{etaSmooth}-\eqref{etaTail1} we know that that there exist constants $c_1, c_2 > 0$ so that $c_1(T+t)^{n_1-1} \leq \eta(t) \leq c_2(T+t)^{n_1-1}$ for all $t \in [-T,T-t^*]$.  Since $v_{\e}$ is bounded by \eqref{1DTightBounds}, this implies that
	\begin{align*}
	2\e^{2}|v_{\e}^{\prime}(t)|  &  \leq\frac{C}{\eta(t)}%
	\int_{-T}^{t}\eta(s)~ds\leq\frac{C}{c_{1}(T+t)^{n_1-1}}\int_{-T}^{t}%
	c_{2}(T+s)^{n_1-1}~ds\\
	&  =\frac{Cc_{2}}{c_{1}n_1}(T+t)
	\end{align*}
	for all $t\in(-T,T-t^*)$. Using a similar argument in $(-T+t^*,T)$,
	we conclude that
	\[
	\e^{2}|v_{\e}^{\prime}(t)|\leq C\min\{T+t,T-t\}
	\]
	for all $t\in I$. By \eqref{1DEulerLagrange}, $v_{\e}$ satisfies%
	\[
	2\e^{2}v_{\e}^{\prime\prime}(t)+2\e^{2}\frac
	{\eta^{\prime}(t)}{\eta(t)}v_{\e}^{\prime}(t)=W^{\prime
	}(v_{\e}(t))+\e\lambda_{\e}.
	\]
	Using \eqref{etaPrimeRatio}, \eqref{1DTightBounds} and the previous inequality we get
	\[
	2\e^{2}|v_{\e}^{\prime\prime}(t)|\leq\left\vert \frac
	{\eta^{\prime}(t)}{\eta(t)}\right\vert 2\e^{2}|v_{\e
	}^{\prime}(t)|+C\leq C.
	\]
	Next we use a classical interpolation result. Let $t\in I$ and consider
	$t_{1}\in I$ with $|t-t_{1}|=\e$. By the mean value theorem
	$v_{\e}(t)-v_{\e}(t_{1})=v_{\e}^{\prime}%
	(\theta)(t-t_{1})$ and so by the fundamental theorem of calculus
	\[
	v_{\e}^{\prime}(t)=v_{\e}^{\prime}(\theta)+\int_{\theta}%
	^{t}v_{\e}^{\prime\prime}(s)~ds=\frac{v_{\e
		}(t)-v_{\e}(t_{1})}{t-t_{1}}+\int_{\theta}^{t}v_{\e}%
	^{\prime\prime}(s)~ds.
	\]
	Again by \eqref{1DTightBounds} it follows that
	\[
	|v_{\e}^{\prime}(t)|\leq\frac{C}{\e}+\sup|v_{\e
	}^{\prime\prime}||t-\theta|\leq\frac{C}{\e}+\frac{C}{\e^{2}%
}\e.
\]
This concludes the proof.
\end{proof}


We are now prepared to prove Theorem \ref{thm:1DMinProperties}. By way of notation, for every measurable subset $E\subset I$ and for every $v\in H_{\eta}^{1}$
satisfying $\|v-v_0\|_{L_\eta^1} \leq \delta$ and \eqref{1DMassConstraint} we define the
localized energy
\begin{equation}
J_{\e}^{(1)}(v;E):=\int_{E}\left(  \frac{1}{\e
}W(v)+\e (v^{\prime})^{2}\right)  \eta~dt.
\label{G1 epsilon local}%
\end{equation}
Figure \ref{Fig:ProofExplanation} gives a visual representation of the notation used in the following proof.

\begin{figure}
\centering
  \includegraphics[width = .8\linewidth]{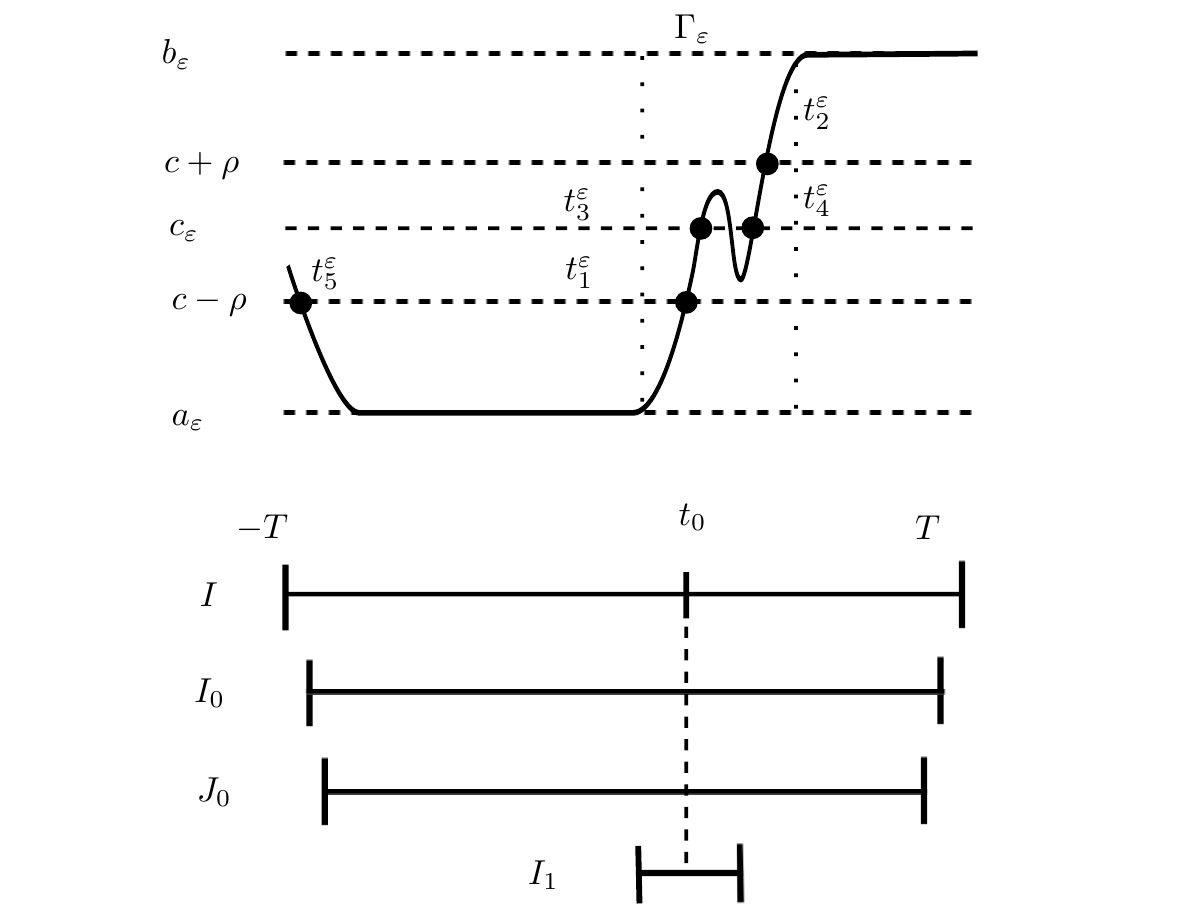}
  \caption{Important intervals and points for the proof of Theorem \ref{thm:1DMinProperties}}
  \label{Fig:ProofExplanation}
\end{figure}

\begin{figure}
\centering
\begin{tabular}{|l| p{5.5cm}| p{7cm} |}
\hline \textbf{Symbol} & \textbf{Definition} & \textbf{Characteristics} \\
\hline $O_\e$ &  (\ref{set O epsilon}) & Step 1 proves that $\mathcal{L}^1(O_\e) = o(1)$. \\
\hline $I_0$ & $[-T + \r_0, T-\r_0]$ (see statement of Theorem \ref{thm:1DMinProperties}) & \\
\hline $J_0$ & $[-T + 2\r_0, T-2\r_0]$ (see Step 2) & \\
\hline $I_1$ & $[t_0 - \hat \r, t_0 + \hat \r]$ (see \eqref{Def:I1})& \\
\hline $\Gamma_\e$ & A maximal subinterval of $O_\e$ which intersects $\overline{B(t_0,\r_1/2)}$ & Existence proved in Step 3, uniqueness, endpoint values and width estimate in Step 4. \\
\hline $t_1^\e, t_2^\e$ & (\ref{times c0}) & \\
\hline $t_3^\e, t_4^\e$ & The first and last time in $\Gamma_\e$ where $v_\e = c_\e$ (see Step 3) & Step 3 proves that these are $O(\e)$ distance apart. \\
\hline $t_5^\e$ & The last point to the left of $\Gamma_\e$ where $v_\e(t_5^\e) = c-\rho$ & Step 5 proves that $t_5^\e$, if it exists, must be in $[-T, -T + c(n_1)\e^{\theta_1}]$. \\ \hline
\end{tabular}
\caption{Explanations of some of the notation in the proof of Theorem \ref{thm:1DMinProperties}.}
\end{figure}

\begin{proof}[ Proof of Theorem \ref{thm:1DMinProperties}]
	
	By Theorem \ref{1DLimsupOrder2} there exists $\tilde{v}_{\e}$ converging to $v_0$ in $L_\eta^1$ such
	that
	\begin{equation}
	G_{\e}^{(1)}(v_{\e}) = J_\e^{(1)}(v_\e) \leq J_{\e}^{(1)}(\tilde
	{v}_{\e})  \leq G_{\e}^{(1)}(\tilde
	{v}_{\e})\leq G^{(1)}(v_{0})+C\e=2c_W%
	\eta(t_{0})+C\e, \label{energies bounded}%
	\end{equation}
	where we have used the fact that $v_{\e}$ is a
	minimizer of $J_\e$. We fix
	\begin{equation}\label{Eqn:Eps1Bound}
	0<\epsilon_1 < \min \left\{\frac{\eta(t_0)}{2},\frac{\eta(t_0)}{2c_W}\int_c^{c+\rho}W^{1/2}(s)\ds,\frac{\min \{c_-,c_+\}}{2c_W}\min_{I_0}\eta \right\},
	\end{equation}
	where
		\begin{equation}
		c_{-}:=\int_{a}^{c}W^{1/2}(s)~ds,\quad c_{+}:=\int_{c}^{b}W^{1/2}(s)~ds.
		\label{c- and c+}%
		\end{equation}
	By the continuity of $\eta$ there exists $\r_{\epsilon_1}>0$ so that 
	\begin{equation}\label{Def:Oldr1}
	|\eta(t)-\eta(t_0)| \leq \epsilon_1
	\end{equation}
	for all $t \in [t_0 - \r_{\epsilon_1},t_0+\r_{\epsilon_1}]$. Pick $\hat \r>0$ so that
	\begin{equation} \label{Def:I1}
	I_1 := [t_0-\hat{\r},t_0+\hat{\r}] \subset I,
	\end{equation}
	and let
	\begin{equation}\label{Eqn:EtaMinBound}
	\eta_1 := \min_{I_1} \eta > 0.
	\end{equation}
	Choose $\r_1$ so that
	\begin{equation} \label{Def:r1}
	0<\r_1 < \min \{\r_{\epsilon_1},\hat{\r}\}.
	\end{equation}	
	Fix $\delta$ so that
	\begin{equation}\label{Def:delta}
	0<\delta < (c-a-\rho)\frac{\eta(t_0)}{2}\r_1.
	\end{equation}

	\noindent\textbf{Step 1: } We claim that $\mathcal{L}^1(O_\e) = o(1)$ (see \eqref{set O epsilon}). Define the set 
	\begin{equation}
	D_\e := 	O_{\e} \cap v_\e^{-1}([c-\rho,c+\rho]\}).
	\end{equation}
	By Lemma \ref{DerivativeBounds}, $|v_\e'| \leq C_0 \e^{-1}$, and so, using the notation in \eqref{defSetWidening}, $(D_\e)^{l \e} \subset v_\e^{-1}([c-2\rho,c+2\rho])$, provided $0<l \leq \rho C_0^{-1}$. In turn 
	\begin{align}
	\mathcal{L}^1((D_\e)^{l \e}) &\leq  \int_{\{c-2\rho \leq v_\e \leq c+2\rho\}} 1 \dt \nonumber\\
	&\leq \e^{\theta_1} +\e^{\theta_2}+ \left(\min_{[c-2\rho,c+2\rho]} W\right)^{-1} \int_{-T+\e^{\theta_1}}^{T-\e^{\theta_2}} W(v_\e) \dt \label{1003}\\
	&\leq \e^{\theta_1} +\e^{\theta_2} + C\left(\e^{-\theta_1(n_1-1)}+\e^{-\theta_2(n_2-1)}\right)\int_{-T+\e^{\theta_1}}^{T-\e^{\theta_2}} W(v_\e) \eta \dt\nonumber\\& \leq \e^{\theta_1} +\e^{\theta_2} + C\left(\e^{1-\theta_1(n_1-1)}+\e^{1-\theta_2(n_2-1)}\right),\nonumber
	\end{align}
	where we have used \eqref{W_Smooth}, \eqref{etaSmooth}-\eqref{etaTail2}, \eqref{rho-chosen} and \eqref{energies bounded}.
	
	Next we claim that
	\begin{equation}\label{1004}
	O_\e\subset (D_\e)^{C\e |\log \e|} \cup [-T,-T+c(n_1)\e^{\theta_1}+C\e |\log \e|] \cup [T-c(n_2)\e^{\theta_2}-C\e |\log \e|, T].
	\end{equation}
	Indeed, as $O_\e = A_\e \cup B_\e \cup D_\e$, it suffices to consider $\tilde t \in A_\e$, as the case $\tilde t \in B_\e$ is analogous. Let $I_\e$ be the maximal subinterval of $A_\e$ containing $\tilde t$. By Corollary \ref{corollary diameter}, $\operatorname*{diam} I_\e \leq C \e |\log \e|$. If $I_\e$ intersects $D_\e$,  then $d(\tilde t,D_\e) \leq \operatorname*{diam} I_\e \leq C\e |\log \e|$. Otherwise, since reasoning as in the proof of \eqref{1DTightBounds} and Lemma \ref{lemma c epsilon} it cannot happen that $v_\e$ takes the value $b_\e-\e^k$ at both endpoints of $I_\e$, it follows that one of the endpoints  of $I_\e$ is $-T+c(n_1)\e^{\theta_1}$ or $T-c(n_2)\e^{\theta_2}$, say, $-T+c(n_1)\e^{\theta_1}$.  Thus
	\[
	d(\tilde t, [-T,-T+c(n_1)\e^{\theta_1}] ) \leq C \e|\log \e|.
	\]
 This proves \eqref{1004}.

	By Lemma \ref{SetSizeLemma} and \eqref{1003} we have that
	\[
	\mathcal{L}^1((D_\e)^{C\e|\log \e|}) \leq C|\log \e| \mathcal{L}^1((D_\e)^{l\e}) \leq C|\log \e| \left(\e^{\theta_1} +\e^{\theta_2} + \e^{1-\theta_1(n_1-1)}+\e^{1-\theta_2(n_2-1)}\right).
	\]
	Hence by \eqref{1004} we have that
	\begin{align}
	\mathcal{L}^1(O_\e) &\leq \e^{\theta_1} +\e^{\theta_2}+ C \e|\log \e|+ \mathcal{L}^1((D_\e)^{C \e|\log \e|})\\
	&\leq C_1 |\log \e| \left(\e^{\theta_1} +\e^{\theta_2} + \e^{1-\theta_1(n_1-1)}+\e^{1-\theta_2(n_2-1)}\right),
	\end{align}
	where $C_1>0$ is \emph{independent of $\r_0$}.

	\noindent\textbf{Step 2: }We claim if $I_{\e}$ is a maximal
	subinterval of the set $O_{\e}$ (see \eqref{set O epsilon}) that
	intersects the interval $J_{0}:=[-T+2\r_0,T-2\r_0]$, then $I_{\e}$
	is contained in $I_{0}$ for all $\e>0$
	sufficiently small, with
	\begin{equation}
	\mathcal{L}^{1}(I_{\e})\leq C\e|\log\e|.
	\label{diam I epsilon}%
	\end{equation}
	The first part of the claim, namely, that $I_\e \subset I_0$, follows immediately from Step 1.  Lemma \ref{lemma c epsilon} then implies that $I_\e \cap D_\e \neq \emptyset$.  Reasoning as in the proof of \eqref{1003} but using the fact that $\eta\ge \eta_0>0$ in $I_0$ we find that $\mathcal{L}^1((I_\e\cap D_\e)^{C\e}) < C\e$. 
	Again due to the fact that $I_{\e}\subset I_0$, reasoning as in the proof of \eqref{1004} 
	we can show that $I_\e \subset (I_\e\cap D_\e)^{C\e|\log \e|}$. Using Lemma \ref{SetSizeLemma} once more gives \eqref{diam I epsilon}.

	\noindent\textbf{Step 3: }We claim that there exist $t_{1}^{\e}$,
	$t_{2}^{\e}\in \overline{B(t_0,\r_1/2)}$ such that
	\begin{equation}
	v_{\e}(t_{1}^{\e})\leq c-\rho,\quad v_{\e}%
	(t_{2}^{\e})\geq c+\rho\label{times c0}%
	\end{equation}
	provided $\e>0$ is sufficiently small. Indeed, if $t_{1}^{\e
	}$ does not exist, then $c-\rho<v_{\e}$ in $\overline{B(t_0,\r_1/2)}$, and so by
	\eqref{continuity eta},
	\[
	\delta\geq\int_{\overline{B(t_0,\r_1/2)}}|v_{\e}-v_{0}|\eta~dt\geq(c-a-\rho)\frac
	{\eta(t_{0})}{2}\r_1,
	\]
	where we used \eqref{Eqn:Eps1Bound}. This contradicts \eqref{Def:delta}. Hence the $t_{1}^{\e}$ in \eqref{times c0}  exists, and with a similar argument we can prove
	the existence of $t_{2}^{\e}$.
	
	Since $v_{\e}$ is continuous, by the intermediate value theorem it
	will take all values between $c-\rho$ and $c+\rho$ in $\overline{B(t_0,\r_1/2)}$. Let
	$\Gamma_{\e}^{-}$ be a maximal subinterval of $O_{\e}$
	intersecting $\overline{B(t_0,\r_1/2)}$ such that $v_{\e}(\Gamma_{\e}^{-}%
	)\supset\lbrack c-\rho,c]$ and let $\Gamma_{\e}^{+}$ be a maximal
	subinterval of $O_{\e}$ intersecting $\overline{B(t_0,\r_1/2)}$ such that
	$v_{\e}(\Gamma_{\e}^{+})\supset\lbrack c,c+\rho]$. By
	Step 1, for $\e$ small enough, both intervals are contained in the interval $I_{1}$ given by \eqref{Def:I1}.
	
	We claim that either $v_{\e}(\Gamma_{\e}^{-})=[a_{\e
	}+\e^{k},b_{\e}-\e^{k}]$ or $v_{\e
}(\Gamma_{\e}^{+})=[a_{\e}+\e^{k},b_{\e
}-\e^{k}]$. Indeed, if this is not the case, then by the maximality
of $\Gamma_{\e}^{-}$ and $\Gamma_{\e}^{+}$, Lemma \ref{lemma c epsilon} and the definition of
$O_{\e}$ (see \eqref{set O epsilon}) $v_{\e}=a_{\e
}+\e^{k}$ at both endpoints of $\Gamma_{\e}^{-}$ and
$v_{\e}=b_{\e}-\e^{k}$ at both endpoints of
$\Gamma_{\e}^{+}$. Let $t_\e \in \Gamma_\e^-$ be such that $v_\e(t_\e) = c$. Hence, by \eqref{G1 epsilon local}, \eqref{Eqn:EtaMinBound}, Young's inequality and a change of
variables,%
\begin{align}
J_{\e}^{(1)}(v_{\e};\Gamma_{\e}^{-}) &  \geq
2 \eta_1 \int_{\Gamma_{\e}^{-}}W^{1/2}%
(v_{\e})|v_{\e}^{\prime}|~dt \nonumber\\
&= 2 \eta_1\int_{\Gamma_\e^- \cap (-T,t_\e]}W^{1/2}%
(v_{\e})|v_{\e}^{\prime}|~dt + 2 \eta_1\int_{\Gamma_\e^- \cap (t_\e,T)}W^{1/2}%
(v_{\e})|v_{\e}^{\prime}|~dt\nonumber\\
&  \geq 4 \eta_1 \int_{a_{\e
	}+\e^{k}}^{c}W^{1/2}(s)~ds \geq 4c_{-}\eta_1-C\e
^{(q+3)/{2q}},\label{energy G1 epsilon I epsilon -}
\end{align}
where we have used \eqref{c- and c+} and the fact that%
\[
\int_{a}^{a_{\e}+\e^{k}}W^{1/2}(s)~ds\leq C|a-a_{\e
}-\e^{k}|^{(q+3)/{2}}\leq C\e^{(q+3)/{2q}}%
\]
by \eqref{W_Limits} and \eqref{1DTightBounds} where here $C$ is independent of $\r_0$. A similar inequality
holds for $J_{\e}^{(1)}(v_{\e};\Gamma_{\e}^{+})$ with
the only difference that $c_{-}$ should be replaced by $c_{+}$. Hence, also by
\eqref{continuity eta} and \eqref{energies bounded},
\[
2c_W\eta(t_{0})+C\e  \geq J_{\e}%
^{(1)}(v_{\e};\Gamma_{\e}^{-})+J_{\e}^{(1)}%
(v_{\e};\Gamma_{\e}^{+}) \geq 4c_W(\eta(t_{0})-\epsilon_{1})-C\e
^{(q+3)/{2q}},
\]
which gives
\[
C\e \geq2(\eta(t_{0})-2\epsilon_{1})c_W.
\]
This contradicts \eqref{Eqn:Eps1Bound} provided $\e$ is sufficiently small. This proves the claim. We denote by
$\Gamma_{\e}$ a maximal subinterval of $O_{\e}$ intersecting
$\overline{B(t_0,\r_1/2)}$ such that $v_{\e}(\Gamma_{\e})=[a_{\e
}+\e^{k},b_{\e}-\e^{k}]$.

First we claim that $v_\e$ takes the values $a_\e + \e^k$ and $b_\e - \e^k$ on the endpoints of $\Gamma_\e$. If not then reasoning as in \eqref{energy G1 epsilon I epsilon -}  we would have
\[
J_{\e}^{(1)}(v_\e;\Gamma_\e) \geq 4c_W \eta_1 -C\e^{(q + 3)/2}
\]
which is a contradiction. Next let $t_{3}^{\e}$ and $t_{4}^{\e}$ be the first time and
last time in $\Gamma_{\e}$ that $v_{\e}$ equals $c_{\e}%
$. We claim that
\begin{equation}
t_{4}^{\e}-t_{3}^{\e}\leq C_2\e
,\label{estimate oscillation interval}%
\end{equation}
for some constant $C_2>0$ \emph{independent of } $\r_0$, for all $\e$ sufficiently small.
Indeed, if $v_{\e}(t)\in\lbrack c-\rho,c+\rho]$ for all $t\in\lbrack
t_{3}^{\e},t_{4}^{\e}]$, then by \eqref{continuity eta},
\[
J_{\e}^{(1)}(v_{\e};[t_{3}^{\e},t_{4}^{\e
}])\geq\e^{-1}\frac{\eta(t_{0})}{2}(t_{4}^{\e}-t_{3}^{\e})\min_{[c-\rho,c+\rho]}%
W,
\]
and so \eqref{estimate oscillation interval} follows from
\eqref{energies bounded}, where all the constants appearing are independent of $\r_0$. On the other hand if there exists $\tilde t^{\e}\in\lbrack
t_{3}^{\e},t_{4}^{\e}]$ such that $|v_{\e}%
(\tilde t^{\e})-c|\geq\rho$, say, $v_{\e}(\tilde t^{\e
})\geq c+\rho$, then by Young's inequality, Step 1, \eqref{Eqn:Eps1Bound}, \eqref{Def:Oldr1} and a change of variables we
get%
\[
J_{\e}^{(1)}(v_{\e};[t_{3}^{\e},t_{4}^{\e
}])\geq 2\frac{\eta(t_{0})}{2}\int_{c}^{c+\rho}W^{1/2}%
(s)~ds-C\e^{(q+3)/{2q}}.
\]
Furthermore, by again reasoning as in \eqref{energy G1 epsilon I epsilon -}, and using the fact that $v_\e$ takes the values $a_\e + \e^k$ and $b_\e - \e^k$ on the endpoints of $\Gamma_\e$ we have that
\begin{equation}
J_{\e}^{(1)}(v_{\e};\Gamma_{\e}\backslash [t_3^\e,t_4^\e])\geq 2%
\eta_1 \int_{a_{\e}+\e^{k}}%
^{b_{\e}-\e^{k}}W^{1/2}(s)~ds\geq 2c_W%
\eta_1-C\e^{(q+3)/{2q}}%
,\label{energy G1 epsilon I epsilon}%
\end{equation}
with $C$ independent of $\r_0$.

Hence, by \eqref{continuity eta}, \eqref{energies bounded}, and \eqref{energy G1 epsilon I epsilon},
\begin{align*}
2c_W\eta(t_{0})+C\e &  \geq J_{\e}%
^{(1)}(v_{\e};\Gamma_{\e}\backslash[t_3^\e,t_4^\e])+J_{\e}^{(1)}(v_{\e
};[t_{3}^{\e},t_{4}^{\e}])\\
&  \geq 2c_W(\eta(t_{0})-\epsilon_{1})+\eta(t_{0})\int%
_{c}^{c+\rho}W^{1/2}(s)~ds-C\e^{(q+3)/{2q}},
\end{align*}
which gives%
\[
C\e\geq\eta(t_{0})\int_{c}^{c+\rho}W^{1/2}(s)~ds-2c_W\epsilon
_{1},
\]
which contradicts \eqref{Eqn:Eps1Bound}, provided $\e$ is
sufficiently small. The case where $v_\e(\tilde t^\e)\leq c-\rho$ is analogous.

\noindent\textbf{Step 4: }We claim that for all $\e>0$ sufficiently
small, $\Gamma_{\e}$ is the only maximal subinterval of the set
$O_{\e}$ that intersects the interval $J_{0}$ defined in Step $2$. Indeed, assume that
there exists another maximal subinterval $I_{\e}$ of $O_{\e
}$ that intersects $J_{0}$. By Step 1, $I_{\e}\subset I_{0}$ and
\eqref{diam I epsilon} holds. In view of Lemma \ref{lemma c epsilon} there
exists $t_{\e}\in I_{\e}$ such that $v_{\e
}(t_{\e})=c_{\e}$. Since $I_{\e}$ is a maximal
interval of $O_{\e}$ at one of the endpoints it attains either the
value $a_{\e}+\e^{k}$ or $b_{\e}-\e^{k}$.
In the first case, reasoning as in \eqref{energy G1 epsilon I epsilon -}, we
get
\begin{align*}
J_{\e}^{(1)}(v_{\e};I_{\e})  &  \geq2\min
_{I_{\e}}\eta\int_{I_{\e}}W^{1/2}(v_{\e
})|v_{\e}^{\prime}|~dt\geq2\min_{I_{\e}}\eta
\int_{a_{\e}+\e^{k}}^{c_{\e}}W^{1/2}(s)~ds\\
&  \geq 2c_{-}\min_{I_{\e}}\eta-C|c-c_{\e
}|-C\e^{(q+3)/{2q}}.
\end{align*}
A similar inequality holds in the second case, with $c_{+}$ in place of
$c_{-}$. Hence, by \eqref{continuity eta}, \eqref{energies bounded},
 and by \eqref{energy G1 epsilon I epsilon},%
\begin{align*}
2c_W\eta(t_{0})+C\e &  \geq J_{\e}%
^{(1)}(v_{\e};\Gamma_{\e})+J_{\e}^{(1)}(v_{\e
};I_{\e})\\
&  \geq2c_W\min_{\Gamma_{\e}}\eta+2\min\{c_{-}%
	,c_{+}\}\min_{I_{\e}}\eta-C\e\\
&  \geq2c_W(\eta(t_{0})-\epsilon_{1})+2\min\{c_{-}%
	,c_{+}\}\min_{I_{0}}\eta-C\e,
\end{align*}
which gives%
\[
C\e\geq2\min\{c_{-},c_{+}\}\min_{I_{0}}\eta-2c_W\epsilon
_{1},
\]
which contradicts \eqref{Eqn:Eps1Bound} provided $\e$ is sufficiently small.

This proves that $\Gamma_\e$ is the only maximal subinterval of $O_\e$ that intersects $J_0$. In view of \eqref{limits} it follows that $v_\e$ takes the value $a_\e +   \e^k$ on its left endpoint of $\Gamma_\e$ and $b_\e-\e^k$ on the right endpoint. Indeed, if $v_\e$ takes the value $b_\e - \e^k$ at the left endpoint of $\Gamma_\e$ then since $v_\e(T_2)<a+\rho$ by \eqref{limits}, then $\Gamma_\e$ could not be the only maximal subinterval of $O_\e$ intersecting $J_0$. At this point we have established parts (i) and (ii) of our theorem.

Next we show that
\begin{equation}\label{Eqn:GammaIndependent}
\mathcal{L}^1(\Gamma_\e) \leq C_3 \e |\log \e|,
\end{equation}
for some constant $C_3>0$ \emph{independent of } $\r_0$. By Step 1, and the fact that $\Gamma_\e$ intersects $\overline{B(t_0,\r_1/2)}$, we have that $\Gamma_\e \subset B(t_0,\r_1)$ for $\e$ sufficiently small, where $\r_1$ is given in \eqref{Def:r1}. By \eqref{Eqn:EtaMinBound} and \eqref{Def:r1}, we have that $\eta \geq \eta_1 > 0$ on $\Gamma_\e$, with $\eta_1$ independent of $\r_0$. The argument in Step 2 then implies \eqref{Eqn:GammaIndependent}.

\noindent\textbf{Step 5: }We claim that $v_{\e}< c-\rho$ in
$[-T+c(n_1)\e^{\theta_1},-T+2\r_0]$. We first consider the case where $n_1 > 1$ in \eqref{etaTail1}. Suppose the claim does not hold. By \eqref{limits}, $v_\e(T_1)<a+\rho$ for $\e$ sufficiently small and where $T_1 \in (-T,-T+\r_0)$. By the intermediate value theorem there exists a point in $(T_1,-T+2\r_0)$ where $v_\e$ takes the value $c-\rho$. Since $-T+\e^{\theta_1}<T_1$ for $\e$ sufficiently small, we have that $v_\e$ takes the value $c-\rho$ in  $[-T+\e^{\theta_1},-T+2\r_0]$. Let $t_{5}^{\e}$
be the last time in $[-T+\e^{\theta_1},-T+2\r_0]$ such that
$v_{\e}(t_{5}^{\e})=c-\rho$. We claim that
\begin{equation}
|t_{3}^{\e}-t_{0}|\leq C_4(\e|\log\e|+%
(T+t_{5}^{\e})^{n_1}),\label{t3 -t0}%
\end{equation}
for some $C_4>0$ independent of $\r_0$, where we recall that $t_{3}^{\e}$ and $t_{4}^{\e}$ are the
first time and last time in $\Gamma_{\e}$ that $v_{\e}$ equals
$c_{\e}$. If $t_{3}^{\e}\leq t_{0}\leq t_{4}^{\e}$,
then this follows from \eqref{estimate oscillation interval}. Assume next that
$t_{0}<t_{3}^{\e}$. Then from \eqref{1DMassConstraint},
\begin{equation}
0=\int_I (v_{\e}-v_{0})\eta~dt=\int_{-T}^{t_{0}}(v_{\e
}-a)\eta~dt+\int_{t_{0}}^{t_{3}^{\e}}(v_{\e}-b)\eta
~dt+\int_{t_{3}^{\e}}^{T}(v_{\e}-b)\eta~dt.\label{601}%
\end{equation}
By \eqref{continuity eta},
\begin{align}
0 &  <\frac{\eta(t_{0})}{2}(b-c_{\e})(t_{3}^{\e}-t_{0}%
)\leq\int_{t_{0}}^{t_{3}^{\e}}(b-v_{\e})\eta~dt\label{602}\\
&  = \int_{-T}^{t_{0}}(v_{\e}-a)\eta~dt+\int_{t_{3}^{\e}%
}^{T}(v_{\e}-b)\eta~dt.\nonumber
\end{align}
We now estimate the two terms on the right-hand side of \eqref{602}. By \eqref{1DTightBounds} and \eqref{bEpsilonForm},%
\begin{equation}
\int_{t_{3}^{\e}}^{T}(v_{\e}-b)\eta~dt\leq|b_{\e
}-b|2T\max\eta\leq C\e^{1/q},\label{607}%
\end{equation}
where $C$ is independent of $\r_0$. We decompose the interval $[-T,t_{0}]$ as follows
\begin{equation}
\lbrack-T,t_{0}]=[-T,t_{5}^{\e}]\cup\lbrack t_{5}^{\e
},-T+2\r_0]\cup(\lbrack-T+2\r_0,t_{0}]\setminus \Gamma_{\e})\cup
([-T+2\r_0,t_{0}]\cap \Gamma_{\e}),\label{intervals subdivisions}%
\end{equation}
and estimate the integrals over each of these subintervals. By
\eqref{etaTail1}, \eqref{1DTightBounds}, and \eqref{bEpsilonForm},%
\begin{equation}
\int_{-T}^{t_{5}^{\e}}(v_{\e}-a)\eta~dt\leq(b_{\e
}-a)d_{2}\int_{-T}^{t_{5}^{\e}}(T+t)^{n_1-1}~dt\leq2(b-a)d_{2}%
(T+t_{5}^{\e})^{n_1}.\label{603}%
\end{equation}
Let $Q_{\e} :=[t_{5}^{\e},-T+2\r_0]\cap O_{\e}$.
Since $v_{\e}(t_{5}^{\e})=c-\rho$, we have that
$t_{5}^{\e}\in Q_{\e}$. Since $t_5^\e$ is the last time in $[-T+\e^{\theta_1},-T+2\r_0]$ such that $v_\e$ takes the value $c-\rho$, and since, by Step 4, $v_\e(-T+2\r_0)\leq a_\e+\e^k$ for $\e$ small, it must be that $v_{\e}<
c-\rho$ in $(t_{5}^{\e},-T+2\r_0]$. By
Corollary \ref{corollary diameter}, we get that
\begin{equation}\label{Q epsilon}
\mathcal{L}^{1}(Q_{\e})\leq
C\e|\log\e|,
\end{equation}
with $C$ independent of $\r_0$. Thus by \eqref{etaSmooth} and \eqref{1DTightBounds},
\begin{equation}
\int_{Q_{\e}}(v_{\e
}-a)\eta~dt\leq C\e|\log\e|\label{604}%
\end{equation}
with $C$ independent of $\r_0$. On the other hand, since $v_{\e}\leq a_{\e}+\e^{k}$
in $[t_{5}^{\e},-T+2\r_0]\setminus Q_{\e}$, by \eqref{1DTightBounds} and \eqref{aEpsilonForm},
\begin{equation}\label{6042}
\int_{\lbrack t_{5}^{\e},-T+2\r_0]\setminus Q_{\e}%
}(v_{\e}-a)\eta~dt\leq|a_{\e}+\e^{k}-a|d_{2}%
\int_{-T}^{-T+2\r_0}(T+t)^{n_1-1}~dt\leq C \r_0^{n_1} \e^{1/q},
\end{equation}
with $C$ independent of $\r_0$. Since the set $O_{\e}$ intersects the interval $J_{0}$ only in
$\Gamma_{\e}$ by Step 3, and as $t_0 < t_3^\e$, we have that $v_{\e}\leq
a_{\e}+\e^{k}$ in $[-T+2\r_0,t_{0}]\setminus \Gamma_{\e
}$. Hence, by \eqref{1DTightBounds} and \eqref{aEpsilonForm},
\begin{equation}
\int_{\lbrack-T+2\r_0,t_{0}]\setminus \Gamma_{\e}}(v_{\e}%
-a)\eta~dt\leq|a_{\e}+\e^{k}-a|2T\max\eta\leq C\e^{1/q}
,\label{605}%
\end{equation}
with $C$ again independent of $\r_0$. Again by Step 3, $[-T+2\r_0,t_{0}]\cap \Gamma_{\e}=[t_{0}-\r_1,t_{0}]\cap
\Gamma_{\e}$. Hence, by \eqref{1DTightBounds}  and \eqref{Eqn:GammaIndependent},
\begin{equation}
\int_{\lbrack t_{0}-\r_1,t_{0}]\cap \Gamma_{\e}}(v_{\e}%
-a)\eta~dt\leq C\e|\log\e|,\label{606}%
\end{equation}
for $C$ independent of $\r_0$. Combining the inequalities \eqref{602}, \eqref{607}, \eqref{intervals subdivisions}, \eqref{603}, \eqref{Q epsilon}, \eqref{604}, \eqref{6042},  \eqref{605} and \eqref{606} gives%
\[
\frac{\eta(t_{0})}{2}(b-c_{\e})(t_{3}^{\e}-t_{0})\leq
C\e|\log\e|+2(b-a)d_{2}(T+t_{5}^{\e})^{n_1},
\]
with $C$ independent of $\r_0$, which implies \eqref{t3 -t0} in the case $t_{0}<t_{3}^{\e}$.

It remains to prove \eqref{t3 -t0} in the case $t_{4}^{\e}<t_{0}$.
Then \eqref{601} should be replaced by
\begin{equation}
0=\int_{-T}^{T}(v_{\e}-v_{0})\eta~dt=\int_{-T}^{t_{4}^{\e}%
}(v_{\e}-a)\eta~dt+\int_{t_{4}^{\e}}^{t_{0}}(v_{\e
}-a)\eta~dt+\int_{t_{0}}^{T}(v_{\e}-b)\eta~dt\label{701}%
\end{equation}
and \eqref{602} by
\begin{equation}
0 <\frac{\eta(t_{0})}{2}(c_{\e}-a)(t_{0}-t_{4}^{\e}%
)\leq\int_{t_{4}^{\e}}^{t_{0}}(v_{\e}-a)\eta~dt \leq\int_{t_{0}}^{T}(b-v_{\e})\eta~dt+\int_{-T}^{t_{4}%
	^{\e}}(a-v_{\e})\eta~dt.\label{702}
\end{equation}
By \eqref{1DTightBounds} and \eqref{aEpsilonForm},
\begin{equation}
\int_{-T}^{t_{4}^{\e}}(a-v_{\e})\eta~dt\leq|a-a_{\e
}|2T\leq C\e^{1/q},\label{707}%
\end{equation}
with $C$ independent of $\r_0$. The integral $\int_{t_{0}}^{T}(b-v_{\e})\eta~dt$ can be estimated as
in the case $t_{0}<t_{3}^{\e}$. We omit the details.
Hence, we have shown that \eqref{t3 -t0} holds in all cases. 

Since $t_{3}^{\e}\in \Gamma_{\e}$, by \eqref{Eqn:GammaIndependent} and \eqref{t3 -t0}, it
follows that for any $t\in \Gamma_{\e}$,%
\[
\left\vert t-t_{0}\right\vert \leq\left\vert t-t_{3}^{\e}\right\vert
+ |t_{3}^{\e}-t_{0}| \leq C_5(\e|\log\e|+%
(T+t_{5}^{\e})^{n_1}),
\]
where $C_5>0$ is independent of $\r_0$.
In turn, by the mean value theorem
\begin{align*}
\eta(t) &  =\eta(t_{0})+\eta^{\prime}(\theta)(t-t_{0})\geq\eta(t_{0}%
)-M_{0}|t-t_{0}|\\
&  \geq\eta(t_{0})-C_5M_0(\e|\log\e|+(T+t_{5}%
^{\e})^{n_1}),
\end{align*}
where we recall that $M_{0}=\max|\eta^\prime|+1$. 
Hence, also by \eqref{energy G1 epsilon I epsilon} we get
\[
J_{\e}^{(1)}(v_{\e};\Gamma_{\e})   \geq2c_W\min_{\Gamma_{\e}}\eta-C\e^{(q+3)/{2q}} \geq2c_W\eta(t_{0})-C_6(\e|\log\e
|+(T+t_{5}^{\e})^{n_1})
\]
with $C_6>0$ independent of $\r_0$. On the other hand, since $v_{\e}(t_{5}^{\e})=c-\rho$, there
exists a maximal subinterval $S_{\e}$ of $Q_{\e}$ that
contains $t_{5}^{\e}$. As argued just before \eqref{Q epsilon}, it must be that $v_{\e}(S_{\e
})\supset\lbrack a_{\e}+\e^{k},c-\rho]$, and so reasoning as in
\eqref{energy G1 epsilon I epsilon -}, by \eqref{etaTail1}, which can be applied since $2\r_0<t^*$ by \eqref{choice h0} and \eqref{Q epsilon} holds,%
\begin{align*}
J_{\e}^{(1)}(v_{\e};S_{\e}) &  \geq2\min_{S_{\e}}\eta\int_{a_{\e}+\e^{k}}^{c-\rho
}W^{1/2}(s)~ds\\
&  \geq2d_{1}(T+t_{5}^{\e})^{n_1-1}\int_{a+\rho}^{c-\rho}%
W^{1/2}(s)~ds,
\end{align*}
for $\e>0$ small enough. Combining these last two estimates, it follows from \eqref{energies bounded}
that%
\begin{align*}
2c_W\eta(t_{0})+C\e &  \geq J_{\e}^{(1)}(v_{\e};\Gamma_{\e})+J_{\e}^{(1)}(v_{\e};S_{\e})\geq2c_W\eta
(t_{0})-C_6(\e|\log\e|+(T+t_{5}^{\e})^{n_1})\\
&  \quad+2d_{1}(T+t_{5}^{\e})^{n_1-1}\int_{a+\rho}^{c-\rho
}W^{1/2}(s)~ds,
\end{align*}
which gives
\begin{equation}
C\e|\log\e|\geq (T+t_{5}^{\e}%
)^{n_1-1}\left(  2d_{1}\int_{a+\rho}^{c-\rho}W^{1/2}(s)~ds-C_{6}%
(T+t_{5}^{\e})\right)  .\label{600}%
\end{equation}
Since $-T+\e^{\theta_1}\leq t_{5}^{\e}\leq-T+2\r_0$, by
taking
\[
0<\r_0<\frac{d_{1}}{C_{6}}\int_{a+\rho}^{c-\rho}W^{1/2}(s)~ds,
\]
we get a contradiction, since $\theta_1(n_1-1) < 1$ by \eqref{def: theta i}.

Finally we consider the case where $n_1 = 1$. In this case we can use energy estimates, as in Step 4, the fact that $\eta \geq C > 0$ on $[-T,-T+2\r_0]$, and Lemma \ref{lemma c epsilon} to show that $v_\e(t) < a_\e + \e^k$ on the interval $[-T,-T+2\r_0]$. We omit the details.

\noindent \textbf{Step 6: } Finally, we prove the last claim in our theorem. We write $\Gamma_\e = [T_1^\e,T_2^\e]$. By the remark at the end of Step 5, in the case $n_1 = 1$ we are already done, so we only need to consider the case $n_1 > 1$. In view of Step 5 we can  use the barrier method in Lemma \ref{BarrierLemma} to show that for $t \in [-T + \e^{\theta_1},T_1^\e]$
\[
|v_\e(t) - a_\e| \leq Ce^{-\mu \e^{-1}d(t,\{-T+\e^{\theta_1},T_1^\e\})}
\]
This clearly implies that $v_\e (t)\in [a_\e, a_\e + \e^k)$ for all $t \in (-T + \e^{\theta_1} + 2k\mu^{-1}\e|\log \e|,T_1^\e)$. Using \eqref{etaTail1} we then estimate the $\eta$ measure of the remaining set as follows:
\[
\int_{-T}^{-T + \e^{\theta_1} + 2k\mu^{-1}\e|\log \e|} \eta\, dt \leq \frac{d_2}{n_1} (\e^{\theta_1} + C\e|\log\e|)^{n_1} \leq C\e^{n_1\theta_1}
\]
Since $n_1 \theta_1 > 1$ by \eqref{def: theta i}, then we have the desired estimate. Thus the result holds to the left of $T_1^\e$. We can use the same argument to the right of $T_2^\e$ to obtain the desired result.
\end{proof}

\subsection{Second-Order $\Gamma$-limit}

In this subsection we prove the $\liminf$ counterpart of Theorem \ref{1DLimsupOrder2}.
\begin{theorem} \label{1DLiminfOrder2}
Assume that $W$ satisfies \eqref{W_Smooth}-\eqref{WGurtin_Assumption} and that $\eta$ satisfies \eqref{etaSmooth}-\eqref{etaPrimeRatio} and let $v_0$ and $v_\e$ be given in Theorems \ref{theorem local minimizer G0} and \ref{JLocalMinimizers} respectively. Then
\begin{equation}\label{eqn:liminfOrder2}
\begin{aligned}
\liminf_{\e \to 0^+} \frac{G_\e^{(1)}(v_\e)-2c_W\eta(t_0)}{\e} &\geq 2\eta'(t_0)(  \tau_0c_W + c_{\operatorname*{sym}}) \\
	&\quad + \begin{cases}
	 \frac{\lambda_0^2}{2W''(a)} \int_I \eta  \ds \quad &\text{ if } q = 1, \\
	0 &\text{ if } q<1.
	\end{cases}
\end{aligned}
\end{equation}
\end{theorem}

Note that Theorems \ref{1DLimsupOrder2}  and \ref{1DLiminfOrder2} together provide a second-order asymptotic development by $\Gamma$-convergence for the functionals $J_\e$ defined in \eqref{def:JEps}. To prove Theorem \ref{1DLiminfOrder2} it is convenient to rescale the functionals $G_\e$. We define
\begin{equation} \label{rescaledProblemFormulation}
H_\e (w) := \int_{A\e^{-1}}^{B\e^{-1}} (W(w(s)) + (w'(s))^2) \eta_\e(s)  \ds
\end{equation}
for all $w\in H^1_{\eta_\e}((A\e^{-1},B\e^{-1}))$ 
such that
\begin{equation}
\int_{A\e^{-1}}^{B\e^{-1}} |w(s)-\operatorname*{sgn}\nolimits_{a,b}(s)| \eta_\e(s) \ds \leq \frac{\delta}{\e},\quad \int_{A\e^{-1}}^{B\e^{-1}} (w(s)-\operatorname*{sgn}\nolimits_{a,b}(s)) \eta_\e(s)  \ds = 0, \label{eqn:RescaledMassConstraint}
\end{equation}
where $A = -T-t_0$, $B = T - t_0$ and 
\begin{equation} \label{rescaledEtaDefinition}
\eta_\e(s) := \eta(t_0+\e s).
\end{equation}

Observe that we have shifted our variables so that $t_0$ moves to zero and then scaled by $\e^{-1}$, which in view of \eqref{eqn:RescaledMassConstraint} implies that minimizers of $H_\e$ are precisely rescaled versions of minimizers of $J_\e$. Here we study the behavior of minimizers $w_\e$ of $H_\e$.  First we prove a bound on the locations where $w_\e = c_\e$, in the region close to $t=0$.

\begin{lemma} \label{rescaledZerosBounded}
Let $w_\e$ be a minimizer of $H_\e$, and let $\tau_\e \in B(0,\r_1 \e^{-1})$ satisfy $w_\e(\tau_\e) = c_\e$, with $\r_1$ as in Theorem \ref{thm:1DMinProperties} (i). Then we have that
	\[
	|\tau_\e| \leq C
	\]
	for all $\e>0$ sufficiently small and for some constant $C>0$ independent of $\e$.
\end{lemma}

\begin{proof}
	This proof essentially combines the mass constraint with the exponential decay to obtain the desired bounds.
	
Let $s_1^\e$ be the first time in $[-\r_1\e^{-1},\r_1\e^{-1}]$ so that $w_\e(s_1^\e) = c-\rho$, and $s_4^\e$ be the last time in $[-\r_1\e^{-1},\r_1\e^{-1}]$ so that $w_\e(s_4^\e) = c+\rho$. Then let $s_2^\e$ and $s_3^\e$ be the first and last times in $[-\r_1\e^{-1},\r_1\e^{-1}]$ where $w_\e$ takes the value $c_\e$. We note that such points exist by Theorem \ref{thm:1DMinProperties} (i).  Furthermore, by Theorem \ref{thm:1DMinProperties} (ii) we know that $s_3^\e-s_2^\e \leq C$ and that $-\r_1\e^{-1} < s_1^\e<s_2^\e\leq s_3^\e<s_4^\e < \r_1 \e^{-1}$. Furthermore, using the same argument from the proof of \eqref{1DTightBounds} we know that $w_\e([s_1^\e,s_2^\e]) = [c-\rho,c_\e]$, and that $w_\e([s_3^\e,s_4^\e]) = [c_\e,c+\rho]$. We can then estimate the following:
\[
(s_2^\e - s_1^\e)\inf_{B(t_0,\r_1)} \eta \inf_{(c-\rho,c+\rho)} W  \leq \int_{s_1^\e}^{s_2^\e} W(w_\e)\eta_\e \ds \leq C.
\]
This, along with a similar estimate for $s_4^\e - s_3^\e$, then implies that $s_4^\e-s_1^\e \leq C$. Thus if we can prove that the $s_1^\e$ are bounded above and that the $s_4^\e$ are bounded below then we are done.

Suppose, for the sake of contradiction that the $s_1^\e$ are not bounded above. By taking a subsequence as necessary we may assume that $s_1^\e \to \infty$.
	
	By \eqref{1DTightBounds} and Lemma \ref{BarrierLemma} we have the following bounds
	\begin{align}
	0<w_\e(s)- a_\e &\leq 2(c-\rho - a_\e) e^{-\mu|s-s_1^\e|}\quad \text{for } s \in [-\r_1\e^{-1},s_1^\e], \label{rescaledExpBound2} \\
	 0<b_\e - w_\e(s) &\leq 2( b_\e - c - \rho) e^{-\mu(s-s_4^\e)}\quad \text{for } s \in [s_4^\e,\r_1\e^{-1}]. \label{rescaledExpBound1}	
	\end{align}
	By our mass constraint \eqref{eqn:RescaledMassConstraint} we can write:
	\begin{align}\label{w1}
0 &= \int_{A\e^{-1}}^{B\e^{-1}} (w_\e - \operatorname*{sgn}\nolimits_{a,b}) \eta_\e \ds = \int_{A\e^{-1}}^{s_1^\e} (w_\e - \operatorname*{sgn}\nolimits_{a,b}) \eta_\e  \ds \\
	&\quad+ \int_{s_1^\e}^{s_4^\e}(w_\e - \operatorname*{sgn}\nolimits_{a,b}) \eta_\e \ds + \int_{s_4^\e}^{B\e^{-1}}(w_\e - \operatorname*{sgn}\nolimits_{a,b}) \eta_\e  \ds .\nonumber
	\end{align}	
	We will estimate these terms to obtain a contradiction. By \eqref{1DTightBounds} and the fact that $0<s_4^\e-s_1^\e\leq C$ we have that
	\begin{equation}
	\left|\int_{s_1^\e}^{s_4^\e} (w_\e-\operatorname*{sgn}\nolimits_{a,b}) \eta_\e \ds\right| \leq C.\label{w2}
	\end{equation}
	We can also calculate
	\begin{align}
	&\int_{A\e^{-1}}^{s_1^\e} (w_\e - \operatorname*{sgn}\nolimits_{a,b}) \eta_\e  \ds \label{w3}\\
	&= \int_{A\e^{-1}}^{s_1^\e} (w_\e - a_\e) \eta_\e \ds + \int_{A\e^{-1}}^{s_1^\e} (a_\e - \operatorname*{sgn}\nolimits_{a,b}) \eta_\e  \ds.\nonumber
	\end{align}
	By \eqref{rescaledExpBound2} we have that
	\begin{equation}
	0 \leq \int_{-\r_1\e^{-1}}^{s_1^\e} (w_\e -   a_\e) \eta_\e \ds \leq 2(c-\rho -   a_\e)\max \eta  \int_{-\r_1\e^{-1}}^{s_1^\e} e^{-\mu |s-s_1^\e|}  \ds \leq C,\label{w4}
	\end{equation}
	whereas by Theorem \ref{thm:1DMinProperties} (iii) and \eqref{1DTightBounds} we know that
	\begin{equation}
		\left|\int_{A\e^{-1}}^{-\r_1\e^{-1}} (w_\e - a_\e) \eta_\e \ds \right| \leq C\e^{k-1} + o(1).
	\label{w5}
	\end{equation}
	Furthermore as $  a_\e = a + O(\e^{1/q})$ by Theorem \ref{Thm:1DBounds}, we may estimate that
	\begin{equation}
	\left|\int_{A\e^{-1}}^{0} (  a_\e - \operatorname*{sgn}\nolimits_{a,b}) \eta_\e  \ds\right| \leq C\e^{\frac{1-q}{q}}.
		\label{w6}
		\end{equation}
  A similar argument, and the fact that $0<s_1^\e<s_4^\e$ shows that
	\begin{equation}
	\left|\int_{s_4^\e}^{B\e^{-1}} (w_\e- \operatorname*{sgn}\nolimits_{a,b}) \eta_\e  \ds\right| \leq C.
		\label{w7}
		\end{equation}
			Now as $s_1^\e \to \infty$ we then have that
	\begin{equation}
	\lim_{\e \to 0^+}\left|\int_{0}^{s_1^\e} (  a_\e - \operatorname*{sgn}\nolimits_{a,b}) \eta_\e  \ds\right| \geq \lim_{\e \to 0^+}\inf_{B(t_0,\r_1)} \eta \left|\int_{0}^{s_1^\e} (  a_\e - b)  \ds\right| = \infty.
		\label{w8}
		\end{equation}
Combining \eqref{w1}--\eqref{w8} gives
	\[
	\lim_{\e \to 0^+}\left|\int_{A\e^{-1}}^{B\e^{-1}} (w_\e - \operatorname*{sgn}\nolimits_{a,b}) \eta_\e  \ds\right| = \infty .
	\]
	This violates the mass constraint. Thus we must have that the $s_1^{\e}$ are bounded above.
	
	A similar argument shows that $s_4^\e$ is bounded below. As $\tau_\e \in (s_1^\e,s_4^\e)$ and $s_4^\e-s_1^\e \leq C$, we then have that $|\tau_\e|\leq C$, which is the desired conclusion.
	
\end{proof}

We then prove that the functions $w_\e$ necessarily converge.
\begin{lemma}\label{profilesConverge}
	Let $w_\e$ be as in Lemma \ref{rescaledZerosBounded}. Then (up to a subsequence, not relabeled) $\{w_\e\}$ converges weakly in $H^1((-l,l))$ for every $l\in\mathbb{N}$ to the profile $w_0(s): = z(s - \tau_0)$, where $\tau_0$ is determined by \eqref{limSupDelta0Definition}. Moreover, the family $\{w_\e'\}$ is bounded in $L^\infty((A\e^{-1},B\e^{-1}))$.
\end{lemma}

\begin{proof}
Throughout this proof we let $w_\e$ be associated with its extension by constants outside of $[A\e^{-1},B\e^{-1}]$. The fact that the family $\{w_\e'\}$ is uniformly bounded in $L^\infty(\mathbb{R})$ follows immediately from Lemma \ref{DerivativeBounds}. Furthermore, we have that the $w_\e$ are bounded in $L^\infty(\mathbb{R})$ by \eqref{1DTightBounds}. After a diagonalization argument, this implies that for some $w_0 \in H_{\operatorname*{loc}}^1(\mathbb{R})$,
	\begin{equation} \label{rescaledWeakConvergence}
	w_\e \rightharpoonup w_0 \text{ in } H_{\operatorname*{loc}}^1(\mathbb{R}).
	\end{equation}
	 By \eqref{1DEulerLagrange}  and \eqref{1DNeumannCondition} we have that
	\begin{equation} \label{rescaledELEquation}
	\begin{cases} 2(w_\e' \eta_\e)' - W'(w_\e)\eta_\e = \e \lambda_\e \eta_\e \quad\text{on } (A\e^{-1},B\e^{-1}), \\
	w_\e'(A\e^{-1}) = w_\e'(B\e^{-1}) = 0.\end{cases}
	\end{equation}
	Hence for every $\phi \in C_c^\infty(\mathbb{R})$ for $\e$ small enough we find that
	\[
	\int_{A\e^{-1}}^{B\e^{-1}} 2w'_\e \eta_\e \phi' + W'(w_\e)\eta_\e \phi \ds = -\int_{A\e^{-1}}^{B\e^{-1}} \e \lambda_\e \eta_\e \phi  \ds.
	\]
	Letting $\e \to 0$ and using \eqref{rescaledEtaDefinition} and \eqref{rescaledWeakConvergence} gives
	\[
	\int_{\mathbb{R}} 2 w_0' \eta(t_0) \phi'+ W'(w_0) \eta(t_0) \phi  \ds = 0,
	\]
	which then shows that $w_0$ satisfies the differential equation
	\begin{equation} \label{1DLimitODE}
	2w_0'' = W'(w_0).
	\end{equation}
	Furthermore, by \eqref{1DTightBounds}  we know that $a\le w_0 \le b$, which by \eqref{1DLimitODE} implies that $|w_0''|\leq  C$. Also, by \eqref{eqn:JepsLimitEnergy} and the fact that $H_\e(w_\e) = J_\e(v_\e)$, where $v_\e$ is a minimizer of $J_\e$,
	\[
	\eta(t_0) \int_{-l}^l (w_0')^2 + W(w_0) \ds \leq \lim_{\e \to 0 } \int_{-l}^l ((w_\e')^2 + W(w_\e))\eta_\e  \ds \leq  \lim_{\e \to 0^+} H_\e(w_\e) = 2c_W\eta(t_0)
	\]
	for every $l \in \mathbb{N}$, and thus
	\begin{equation} \label{rescaledLimitEnergyBound}
	\eta(t_0) \int_\mathbb{R} (w_0')^2 + W(w_0)  \ds \leq 2c_W \eta(t_0).
	\end{equation}
	This combined with the fact that $|w_0''|\leq C$ (by \eqref{1DLimitODE}) implies that $\lim_{s \to \pm \infty} w_0'(s) = 0$. By then using \eqref{rescaledExpBound2} and \eqref{rescaledExpBound1} along with Lemma \ref{rescaledZerosBounded} we have that $\lim_{s \to -\infty} w_0(s) = a$, and that $\lim_{s \to \infty} w_0(s) = b$. Thus by integrating \eqref{1DLimitODE} we find that
	\begin{equation}\label{rescaledODE2}
	(w_0')^2 = W(w_0).
	\end{equation}
	We next claim that $w_0$ is increasing. Suppose not. Then by \eqref{rescaledODE2} there exists critical points $t_1<t_2$ of $w_0$, with $w_0(t_1) = b$ and $w_0(t_2) = a$. This then implies, by Young's inequality, \eqref{rescaledLimitEnergyBound} and a change of variables that
	\[
	6c_W\eta(t_0) \leq  2c_W\eta(t_0).
	\]
	This is impossible and thus $w_0$ is increasing. Moreover, by \eqref{cEpsilonForm}, \eqref{rescaledWeakConvergence}, and Lemma \ref{rescaledZerosBounded}, up to a subsequence, $\tau_\e\to \tau_0$ with  $w_0(\tau_0) = c$. 
	This then implies that $w_0(s) = z(s-\tau_0)$, where $z$ is the solution of the Cauchy problem \eqref{profileCauchyProblem}.
	
	The only thing left to prove is that $\tau_0$ is determined by equation \eqref{limSupDelta0Definition}. To this end, fix $l$ large enough that $(s_1^\e,s_4^\e) \subset (-l,l)$ for all $\e$, where $s_1^\e$ and $s_4^\e$ are as in the proof of  Lemma \ref{rescaledZerosBounded}. Then by the mass constraint \eqref{eqn:RescaledMassConstraint} we have that
	\begin{align*}
	0 &=  \int_{A\e^{-1}}^{B\e^{-1}} (w_\e-\operatorname*{sgn}\nolimits_{a,b})\eta_\e  \ds =  \int_{-l}^l (w_\e - \operatorname*{sgn}\nolimits_{a,b})\eta_\e  \ds \\
	&\quad+  \int_{-\r_1\e^{-1}}^{-l} (w_\e -  a_\e + a_\e -\operatorname*{sgn}\nolimits_{a,b})\eta_\e  \ds +  \int_l^{\r_1\e^{-1}} (w_\e -  b_\e + b_\e - \operatorname*{sgn}\nolimits_{a,b}) \eta_\e  \ds \\
	&\quad+ \int_{A\e^{-1}}^{-\r_1\e^{-1}} (w_\e -  a_\e + a_\e -\operatorname*{sgn}\nolimits_{a,b})\eta_\e  \ds +  \int_{\r_1\e^{-1}} ^{B\e^{-1}} (w_\e -  b_\e + b_\e - \operatorname*{sgn}\nolimits_{a,b}) \eta_\e  \ds.
	\end{align*}
	By the definitions of $s_1^\e$ and $s_4^\e$ it must be that $v_\e \leq c-\rho$ in the interval $[-\r_1\e^{-1},-l]$ and $v_\e \geq c+\rho$ in the interval $[l,\r_1\e^{-1}]$. Hence by \eqref{1DTightBounds} and \eqref{Eqn:preciseDecayBarrier} we have that
	\begin{align*}
	0 \leq  \int_l^{\r_1\e^{-1}} (b_\e - w_\e)\eta_\e \ds &\leq 2\left((b_\e- w_\e(l)) + (b_\e - w_\e(\r_1\e^{-1}) \right) \max \eta \int_{0}^\infty e^{-\mu s} \ds \\
	&\leq C(b_\e- w_\e(l) + \e^k),
	\end{align*}
	where in the last inequality we have used \eqref{set O epsilon} and Theorem \ref{thm:1DMinProperties}. Similarly, we have
	\begin{equation}
	0 \leq \int_{-\r_1\e^{-1}}^{-l} (w_\e -  a_\e)\eta_\e \ds \leq C(w_\e(-l) - a_\e + \e^k).
	\end{equation}

	By \eqref{1DTightBounds} we can write:
	\begin{align*}
	\int_{A\e^{-1}}^{-l} ( a_\e - \operatorname*{sgn}\nolimits_{a,b})\eta_\e \ds&= -\lambda_\e |\lambda_\e|^{1/q-1}  (q/\ell)^{1/q}\e^{1/q-1}\int_{-T}^{t_0} \eta  \dt\, 
	 +o(\e^{1/q-1}), \\
	\int_l^{B\e^{-1}}( b_\e - \operatorname*{sgn}\nolimits_{a,b})\eta_\e \ds &= -\lambda_\e |\lambda_\e|^{1/q-1}  (q/\ell)^{1/q}\e^{1/q-1}\int_{t_0}^T \eta  \dt + o(\e^{1/q-1}) .
	\end{align*}
	Furthermore by Theorem \ref{thm:1DMinProperties} along with \eqref{1DTightBounds} we have that
	\begin{align*}
	\int_{A\e^{-1}}^{-\r_1\e^{-1}} (w_\e-a_\e) \eta_\e\ds &= o(1), \\
	\int_{\r_1\e^{-1}}^{B\e^{-1}} (b_\e-w_\e) \eta_\e \ds &= o(1).
	\end{align*}
	Utilizing these estimates, and taking $\e \to 0$ we find that
	\begin{align*}
	0 = \eta(t_0)\int_{-l}^l w_0 - \operatorname*{sgn}\nolimits_{a,b} \ds 
	&- \lambda_0 |\lambda_0|^{1/q-1}  (q/\ell)^{1/q} \lim_{\e \to 0^+}\e^{1/q-1}\int_I \eta  \dt\\
	&+ O(|a - w_0(-l)|)+ O(|b - w_0(l)|) .
	\end{align*}
		Taking $l$ to infinity, and using \eqref{WPrime_At_Wells} then implies that
	\[
	\eta(t_0) \int_\mathbb{R} w_0 - \operatorname*{sgn}\nolimits_{a,b} \ds =\begin{cases} \frac{\lambda_0}{W''(a)} \int_I \eta  \ds &\text { if } q = 1, \\ 0 &\text{ if } q < 1, \end{cases}
	\]
	which then implies that $\tau_0$ has the desired form. This completes the proof.

\end{proof}

Next we will use the previous lemmas to derive a second-order liminf inequality, which immediately implies Theorem \ref{1DLiminfOrder2}.

\begin{lemma} \label{1DLiminf}
	Let $\{w_\e\}$ be minimizers of the functionals $\{H_\e\}$. Then we have the following:
	\begin{align}\label{liminf H epsilon}
	\liminf_{\e \to 0^+} \frac{H_\e(w_\e) - 2c_W \eta(t_0)}{\e} &\geq 2\eta'(t_0)( \tau_0c_W + c_{\operatorname*{sym}}) \\
	&\quad+ \begin{cases}
	 \frac{\lambda_0^2}{2W''(a)} \int_I\eta(s)  \ds \hspace{2mm} &\text{ if } q = 1, \\
	0 &\text{ if } q < 1,
	\end{cases} \nonumber
	\end{align}
	where $c_W$, $c_{\operatorname*{sym}}$, $\tau_0$,  $\lambda_0$ are given by \eqref{c0Definition},  \eqref{c1Definition}, \eqref{limSupDelta0Definition} and \eqref{1DMultiplierLimit} respectively.
\end{lemma}

\begin{proof}
	Fix $k$ to be a large integer. By \eqref{rescaledExpBound2} and \eqref{rescaledExpBound1} and the fact that $s_1^\e$ and $s_4^\e$ are bounded we can find $l_\e \in (s_2^\e, \r_1\e^{-1})$ such that $b_\e - w_\e(l_\e) < \e^k$ and $w_\e(-l_\e) -  a_\e < \e^k$ for $\e>0$ sufficiently small. Recall that by Corollary \ref{corollary diameter} we can take 
	\begin{equation} \label{KepsilonBound}
	l_\e< C|\log \e|. 
	\end{equation}
	By \eqref{rescaledProblemFormulation} we can compute
	\begin{align*}
	&\frac{H_\e(w_\e) - 2c_W \eta(t_0)}{\e}  \\
	&= \e^{-1}\int_{-l_\e}^{l_\e} (W^{1/2}(w_\e) - w_\e')^2 \eta_\e \ds +2\e^{-1}\int_{-l_\e}^{l_\e} W^{1/2}(w_\e) w_\e' (\eta_\e - \eta(t_0))  \ds \\
	&\quad+ \e^{-1}\int_{[A \e^{-1},B \e^{-1}] \backslash (-l_\e,l_\e)} \left(W(w_\e) + (w_\e')^2\right) \eta_\e \ds + \e^{-1}2\eta(t_0)\left(\int_{-l_\e}^{l_\e} W^{1/2}(w_\e) w_\e'  \ds - c_W\right) \\
	&\geq 2\e^{-1}\int_{-l_\e}^{l_\e} W^{1/2}(w_\e) w_\e' (\eta_\e - \eta(t_0))  \ds \\
	&\quad+ \e^{-1}\int_{[A \e^{-1},B \e^{-1}] \backslash (-l_\e,l_\e)} W(w_\e) \eta_\e \ds + \e^{-1}2\eta(t_0)\left(\int_{-l_\e}^{l_\e} W^{1/2}(w_\e) w_\e'  \ds - c_W\right) .
	\end{align*}
	We will examine the individual terms. The last term goes to zero as
	\begin{align}
	\e^{-1}\left| \int_{-l_\e}^{l_\e} W^{1/2}(w_\e) w_\e' \, \ds - c_W\right| &\leq \e^{-1} \left|\int_{w_\e(-l_\e)}^{w_\e(l_\e)} W^{1/2}(r) \,dr - \int_a^b W^{1/2}(r)\, dr \right| \nonumber \\
	&\leq  \e^{-1}\left|\int_{a_\e}^{b_\e} W^{1/2}(r)\, dr - \int_a^b W^{1/2}(r)\, dr \right|+C\e^{k-1}  \nonumber\\
	&\leq C \e^{-1}\int_0^{\e^{1/q}} t^{\frac{1+q}{2}} \dt+C\e^{k-1}= o(1)  \label{1DGeodesicError},
	\end{align}
	where we have used \eqref{c0Definition}, \eqref{W_Limits} and \eqref{1DTightBounds}.
	
	For $s \in [l_\e,B \e^{-1}] \cap \{w_\e \geq b_\e - \e^k\}$ by the mean value theorem we can write
	\[
	W(w_\e(s)) = W(b_\e) + W'(\zeta_\e)(w_\e(s) - b_\e),
	\]
	where $\zeta_\e \in [w_\e(s),b_\e]$. By \eqref{WPrimeLimits} and \eqref{bEpsilonForm} for such $s$ we then have that
	\begin{align*}
	|W'(\zeta_\e)|(b_\e - w_\e(s)) &\leq C|\zeta_\e - b|^q (b_\e - w_\e(s)) \\
	&\leq C(|\zeta_\e -b_\e|^q + |b_\e - b|^q) (b_\e - w_\e(s)) \\
	&\leq C(\e^{q k}  + \e) \e^k \leq C\e^{k+1} .
	\end{align*}
	Thus we can write, after applying \eqref{W_Limits},  part (iii) of Theorem \ref{thm:1DMinProperties},  \eqref{bEpsilonForm}, and \eqref{KepsilonBound},
	\begin{align*}
	&\e^{-1} \int_{l_\e}^{B\e^{-1}} W(w_\e) \eta_\e \ds \geq \e^{-1} W(b_\e) \int_{l_\e}^{B\e^{-1}} \eta_\e  \ds + O(\e^{k -1}) \\
	&= \e^{-1} \left( \frac{\ell}{q(1+q)}|b_\e-b|^{1+q} + o(|b_\e - b|^{1+q}) \right)\left( \e^{-1} \int_{t_0}^T \eta  \dt + O(|\log \e|)\right) + O(\e^{k-1}) \\
	&=  \left( \frac{q^{1/q}|\lambda_\e |^{1+1/q}}{(1+q)\ell^{1/q}} + o(1) \right)\left( \e^{1/q-1} \int_{t_0}^T \eta  \dt + O(\e^{1/q}|\log \e|)\right) + O(\e^{k-1}) .
	\end{align*}
	An analogous bound will hold on the interval $[A\e^{-1},-l_\e]$. Hence
	\begin{align} 
	\lim_{\e \to 0^+} \e^{-1}\int_{[A \e^{-1},B \e^{-1}] \backslash (-l_\e,l_\e)} W(w_\e)\eta_\e  \ds = \begin{cases}
	\frac{\lambda_0^2}{2W''(a)}  \int_I \eta  \dt \hspace{2mm} &\text{ if } q = 1, \\
	0 &\text{ if } q<1.
	\end{cases} \label{1DTailCost}
	\end{align}
	
	For the first term we use assumption \eqref{etaSmooth} to estimate $\eta_\e(s) -\eta(t_0) = \e s \eta'(t_0) + O(\e^{1+\beta} |s|^{1+\beta})$. Using \eqref{1DTightBounds}, Lemma \ref{profilesConverge} and \eqref{KepsilonBound} we have that
	\[
	\left|\e^{-1} \int_{-l_\e}^{l_\e} W^{1/2}(w_\e)w'_\e O(\e^{1+\beta} |s|^{1+\beta})  \ds\right| \leq C\e^{\beta}|\log \e|^{2+\beta} \to 0 .
	\]
	Thus we  find that:
	\[
	\lim_{\e \to 0^+} 2\e^{-1}\int_{-l_\e}^{l_\e} W^{1/2}(w_\e) w_\e' (\eta_\e - \eta(t_0)) \ds = 2\eta'(t_0)\lim_{\e \to 0^+}  \int_{-l_\e}^{l_\e} W^{1/2}(w_\e) w_\e' s  \ds .
	\]
	Now for any fixed $l$ by \eqref{rescaledWeakConvergence} and the fact that $w_0(s) = z(s-\tau_0)$, we can write
	\begin{align*}
	\lim_{\e \to 0^+} \int_{-l}^l W^{1/2}(w_\e) w_\e ' s \ds &= \int_{-l}^l W^{1/2}(w_0) w_0' s  \ds \\
	&= \int_{-l-\tau_0}^{l-\tau_0} W^{1/2}(z(t)) z'(t) (t+\tau_0) \,dt \\
	&= \tau_0\Phi(z(l-\tau_0)) - \tau_0\Phi(z(-l-\tau_0)) + \int_{-l-\tau_0}^{l-\tau_0} W^{1/2}(z(t)) z'(t) t  \dt,
	\end{align*}
	where we recall that $\Phi(s) = \int_a^s W^{1/2}(r)\, dr$. Furthermore we can establish the following bound using \eqref{W_Limits}, \eqref{rescaledExpBound1} and Lemma \ref{profilesConverge}:
	\begin{align*}
	&\left|\int_l^{l_\e} W^{1/2}(w_\e) w_\e ' s \ds\right| \leq C \int_l^{l_\e} |b- w_\e|^{\frac{1+q}{2}} s  \ds \\
	&\leq C(|b_\e-c - \rho|^\frac{1+q}{2} + |b_\e - b|^\frac{1+q}{2}) \int_l^\infty e^{-\frac{1+q}{2}\mu (s-s_4^\e)} s  \ds ,
	\end{align*}
	provided $l > s_4^\e$. Thus we can write
	\begin{align*}
	\lim_{\e \to 0^+} \int_{-l_\e}^{l_\e} W^{1/2}(w_\e) w_\e' s  \ds &=  \tau_0\Phi(z(l-\tau_0)) - \tau_0\Phi(z(-l-\tau_0)) \\
	&\quad+  \int_{-l-\tau_0}^{l-\tau_0} W^{1/2}(z(s)) z'(s) s  \ds + O(le^{-\frac{1+q}{2}\mu l}) .
	\end{align*}
	Taking $l$ to $\infty$, combined with \eqref{1DGeodesicError} and \eqref{1DTailCost} gives the desired claim, namely, \eqref{liminf H epsilon}.
\end{proof}

We now give the proof of Theorem \ref{1DLiminfOrder2}.

\begin{proof}[Proof of Theorem \ref{1DLiminfOrder2}]
By changing variables it is immediate that $H(w_\e) = G^{(1)}_\e(v_\e)$. Lemma \ref{1DLiminf} then immediately implies \eqref{eqn:liminfOrder2}. This concludes the proof.
\end{proof}

%

\section{Proofs of Main Theorems} \label{mainTheorem}
With our tools in hand, we now can approach the problem of proving Theorems \ref{mainThm1} and \ref{mainThm2}. We begin by proving the $\Gamma$-$\liminf$ inequalities from Theorems \ref{mainThm1} and \ref{mainThm2}. Precisely, we prove the following theorem:

\begin{theorem} \label{thm:nDLiminf}
Assume that $\Omega$ satisfies \eqref{domainAssumptions}, $m$ satisfies \eqref{originalMassRange}, $\I$ satisfies \eqref{isoFunctionSmooth} and that $W$ satisfies \eqref{W_Smooth}-\eqref{WGurtin_Assumption}. Let $\{u_n\} \subset L^1(\Omega)$ converge to $u$ in $L^1(\Omega)$. Then
\[
\liminf_{\e \to 0^+} \mathcal{F}_\e^{(2)}(u_\e) \geq \mathcal{F}^{(2)}(u),
\]
where $\mathcal{F}_\e^{(2)}$ is defined by \eqref{higherOrderFunctionalDefinition} and $\mathcal{F}^{(2)}$is defined in Theorems \ref{mainThm1} and \ref{mainThm2}.
\end{theorem}

\begin{proof}
If $\liminf_{\e\to0^+}\mathcal{F}_\e^{(2)}(u_\e) = \infty$ then there is nothing to prove. Thus, passing to a subsequence, if necessary, we can assume that
\begin{equation}
\label{88} \sup_{\e} \mathcal{F}_\e^{(2)}(u_\e) < \infty.
\end{equation}
By standard results on compactness and lower semicontinuity for the Cahn--Hilliard functional $\mathcal{F}_\e^{(1)}$ (see, e.g.,  \cite{LeoniCompactness} and the references therein), it follows from \eqref{higherOrderFunctionalDefinition}, \eqref{higherOrderGammaConvergenceDefinition} and \eqref{88} that $u$ must be a minimizer of $\mathcal{F}^{(1)}$. This implies that the set $E:= \{u=a\}$ is a minimizer of \eqref{partitionProblemDefinition}, and its mean curvature is given by \eqref{isoPerDerivative}. Using $\I^*$ from Proposition \ref{prop:IStar} as in Section 3, we then have that
\[
\mathcal{F}_\e(u_\e) \geq \int_I (W(f_{u_\e}) + \e^2( f_{u_\e}')^2)\I^*(V_\Omega) \dt, \quad m
 = \int_\Omega u_\e \dx = \int_I f_{u_\e} \I^*(V_\Omega) \dt,
\]
where $V_\Omega$ and $f_u$ are defined in Section 3 (see \eqref{DomainRearrangeODE}, \eqref{uStarDefinition} and Remark \ref{increasingRearrangementRemark}). We then set $\eta := \I^*(V_\Omega)$. This $\eta$ will satisfy all of the assumptions in Section 4. Indeed, since $V_\Omega>0$ in $(-T,\infty)$ and $V_\Omega(-T) = 0$, by \eqref{eqn:IStarTail} and \eqref{DomainRearrangeODE}, $V_\Omega(t) = [C_0/n (t+T)]^n$ near $-T$, and so $\eta = C_0^n[\frac{1}{n}(t+T)]^{n-1}$, which shows that \eqref{etaTail1} and \eqref{etaPrimeRatio} hold for $t$ close to $-T$. On the other hand, since $V_\Omega(t) = 1-V_\Omega(-t)$ (by \eqref{eqn:IStarSymmetric} and \eqref{DomainRearrangeODE}), for $t$ close to $T$ we have that $\eta(t) = C_0^n[\frac{1}{n}(T-t)]^{n-1}$ and thus  \eqref{etaTail2} and \eqref{etaPrimeRatio} hold close to $T$. Since $\I^* \in C_{\operatorname*{loc}}^{1,\beta}(0,1)$, by \eqref{DomainRearrangeODE} we have that $V_\Omega \in C_{\operatorname*{loc}}^{2,\beta}(I)$, and in turn $\eta \in C_{\operatorname*{loc}}^{1,\beta}(I)$. Thus \eqref{etaSmooth} is satisfied. Finally, since $\I^*>0$ in $(0,1)$ we have by \eqref{eqn:IStarPositive} that $\eta > 0$ in $I$, and thus \eqref{etaPrimeRatio} holds on any compact subset of $I$ by uniform continuity.

Next observe that since $u \in BV(\Omega,\{a,b\})$ and \eqref{massConstraintEquation} holds, by Lemma \ref{Equimeasurable} we have that $f_u$ only takes the values $a$ and $b$ and $\int_I f_u\eta \dt = \int_\Omega u \dx = m$. Since $f_u$ is increasing, this implies that $f_u(t) = \operatorname*{sgn}\nolimits_{a,b}(t-t_0)$ for some $t_0 \in I$ and all $t\in I$. It follows from Theorem \ref{theorem local minimizer G0} that $f_u$ is a local minimizer of the functional $G^{(1)}$ defined in \eqref{G0Definition}. Moreover, by Lemma \ref{rearrangementContraction} we have that $u_\e \to u$ in $L^1(\Omega)$ implies that $f_{u_\e} \to f_u$ in $L_\eta^1(I)$. Hence, $\|f_{u_\e} - f_u\|_{L_\eta^1} \leq\delta$ for all $\e$ sufficiently small, where $\delta>0$ is the number given in Theorem \ref{theorem local minimizer G0} (with $v_0 = f_u$). In turn choosing $v_\e$ to be minimizers of the function $J_\e$ defined in \eqref{def:JEps}, by Corollary \ref{Cor:EnergyRearrangement} we have that
\begin{equation} \label{eqn:RearrangeEnergy}
\mathcal{F}_\e(u_\e) \geq G_\e(f_{u_\e}) = J_\e(f_{u_\e}) \geq J_\e(v_\e).
\end{equation}

Since $\int_I f_u \eta \dt = m$, it follows from the fact that (see \eqref{domainAssumptions} and Lemma \ref{Equimeasurable})
\begin{equation} \label{101}
1 = \mathcal{L}^n(\Omega) = \int_I \eta \dt
\end{equation}
and \eqref{DomainRearrangeODE} that
\begin{equation} \label{102}
\V_m = \frac{b-m}{b-a} = \mathcal{L}^n(\{u = a\}) = \int_{-T}^{t_0} \eta \dt = \int_{-T}^{t_0} \frac{d}{dt	} V_\Omega \dt = V_\Omega(t_0).
\end{equation}
In turn, by \eqref{eqn:IStarTouches},
\[
\eta(t_0) = \I^*(\V_m) = \I(\V_m) = \operatorname*{P}(\{u=a\};\Omega),
\]
which shows that $\mathcal{F}^{(1)}(u) = G^{(1)}(f_u)$. Hence by \eqref{eqn:RearrangeEnergy} we have that
\[
\mathcal{F}_\e^{(2)}(u_\e) = \frac{\mathcal{F}_\e^{(1)}(u_\e)-\mathcal{F}^{(1)}(u)}{\e} \geq \frac{J_\e^{(1)}(v_\e)-J^{(1)}(f_u)}{\e} = J_\e^{(2)}(v_\e).
\]
By applying Lemma \ref{1DLiminf} we thus have that
\begin{equation} \label{102B}
\liminf_{\e \to 0^+} \mathcal{F}_\e^{(2)}(u_\e) \geq 2\eta'(t_0) (\tau_0 c_W + c_{\operatorname*{sym}}) + \begin{cases} \frac{\lambda_0^2}{2W''(a)} &\text{ if } q =1, \\
0 &\text{ if } q<1 ,\end{cases}
\end{equation}
where we have used \eqref{101}. By \eqref{DomainRearrangeODE} we have that $\eta'(t) = (\I^* )'(V_\Omega(t)) \eta(t)$, and so by \eqref{isoPerDerivative}, \eqref{eqn:IStarTouches} and \eqref{102},
\begin{equation}
\eta'(t_0) = \I'(\V_m) \I(\V_m) = (n-1)\kappa_u \operatorname*{P}(\{u=a\};\Omega).
\end{equation}
In turn by \eqref{def:lambdaNot} and \eqref{def:rhoNot},
\begin{equation} \label{104}
\lambda_0 = \frac{2(n-1)c_W }{(b-a)}\kappa_u = \Lambda_u,
\end{equation}
and so by \eqref{limSupDelta0Definition} the number $\tau_0$ coincides with the number $\tau_u$ in \eqref{deltaUDefinition}. Combining \eqref{102B}-\eqref{104} gives
\[
\liminf_{\e \to 0^+} \mathcal{F}_\e^{(2)}(u_\e) \geq 2(\tau_u c_W + c_{\operatorname*{sym}})(n-1)\kappa_u\operatorname*{P}(\{u=a\};\Omega) + \begin{cases}
 \frac{\Lambda_u^2}{2W''(a)} &\text{ if } q =1,\\
0 &\text{ if } q<1.
\end{cases}
\]
This completes the proof.

\end{proof}

\begin{remark}
We note that the assumption that $\Omega$ is $C^2$ is not truly necessary to prove the previous theorem. A Lipschitz domain satisfying \eqref{isoFunctionSmooth} would actually be sufficient. However, for a Lipschitz domain it is not clear that \eqref{isoFunctionSmooth} need hold, and furthermore for the $\Gamma$-$\limsup$ inequality we directly utilize the fact that $\Omega \in C^2$.
\end{remark}

Now we prove the corresponding $\Gamma$-$\limsup$ inequality. 
\begin{theorem} \label{thm:nDLimsup}
Assume that $\Omega$ satisfies \eqref{domainAssumptions}, $m$ satisfies \eqref{originalMassRange} and that $W$ satisfies \eqref{W_Smooth}-\eqref{WGurtin_Assumption}. Let $u \in L^1(\Omega)$. Then there exists a sequence $\{u_n\} \subset L^1(\Omega)$ that converges to $u$ in $L^1(\Omega)$ such that
\[
\limsup_{\e \to 0^+} \mathcal{F}_\e^{(2)}(u_\e) \leq \mathcal{F}^{(2)}(u),
\]
where $\mathcal{F}_\e^{(2)}$ is defined by \eqref{higherOrderFunctionalDefinition} and $\mathcal{F}^{(2)}$is defined in Theorems \ref{mainThm1} and \ref{mainThm2}.
\end{theorem}

We note that Theorems \ref{thm:nDLiminf} and \ref{thm:nDLimsup} together establish Theorems \ref{mainThm1} and \ref{mainThm2}. To prove this theorem, we primarily utilize the approach from previous works (see \cite{Modica87,Sternberg88}), while leveraging the insight regarding skew in transition layers that we have developed over the course of this work. We begin with the following lemma:
\begin{lemma} \label{levelSetLemma}
Suppose that $E\subset \Omega$ is a volume-constrained perimeter minimizer in $\Omega$. Define the function $\eta(s) := \mathcal{H}^{n-1}(\{d_E(x) = s\})$, where $d_E$ is the signed distance function (see \eqref{def:SignedDistance}). Then $\eta$ is twice differentiable at zero and satisfies
\begin{align}
\eta(0) &= \operatorname*{P}(E;\Omega),\label{levelSetAtZero}\\
\eta'(0) &= (n-1)\kappa_E\operatorname*{P}(E;\Omega),  \label{levelSetDerivative}
\end{align}
where $\kappa_E$ is the mean curvature of $E$. Furthermore, the function $\eta$ is bounded.
\end{lemma}

\begin{proof}
By classical results (see \cite{EvansGariepy}), we have that \eqref{levelSetAtZero} holds. To prove that $\eta$ is twice differentiable at $0$ and that \eqref{levelSetDerivative} holds we are primarily concerned with possible interactions with $\partial\Omega$. By \cite{GruterBoundaryRegularity} we know that $\partial E$ is a $C^{2,\alpha}$ surface, that intersects $\partial \Omega$ orthogonally. By appropriately reflecting $E$ outside of $\Omega$ we may assume that $E$ is a $C^{2,\alpha}$ set in $\mathbb{R}^n$. Since $\partial E$ is of class $C^{2,\alpha}$ for every $x\in\partial E$ there
exist a ball $B\left(  x,r_{x}\right)  $, with local coordinates $y=(y^{\prime
},y_{n})\in\mathbb{R}^{n-1}\times\mathbb{R}$ such that $x$ corresponds to
$y=0$, and a function $g$ of class $C^{2,\alpha}(\mathbb{R}^{n-1})$ such that
$g(0)=0$, $\frac{\partial g}{\partial y_{i}}(0)=0$ for all $i=1,\ldots,n-1$,
and
\begin{align*}
E\cap B(x,r_{x})  &  =\{y\in B(0,r_{x}):\,y_{n}<g(y^{\prime})\},\\
\partial E\cap B(x,r_{x})  &  =\{y\in B(0,r_{x}):\,y_{n}=g(y^{\prime})\}.
\end{align*}
In what follows we use local coordinates and we set $y^{\prime\prime}%
:=(y_{2},\ldots,y_{n-1})$. In particular, since $\partial E$ meets the
boundary of $\Omega$ transversally, if $x\in\partial E\cap\partial\Omega$, by
a rotation and by taking $r_{x}$ smaller, we can assume that
\begin{align*}
\Omega\cap B(x,r_{x})  &  =\{y\in B(0,r_{x}):\,y_{1}<f(y^{\prime\prime}%
,y_{n})\},\\
\partial\Omega\cap B(x,r_{x})  &  =\{y\in B(0,r_{x}):\,y_{1}=f(y^{\prime
\prime},y_{n})\}
\end{align*}
for some function $f \in C^{2,\alpha}(\mathbb{R}^{n-1})$. Setting
$F(y):=y_{1}-f(y^{\prime\prime},y_{n})$ and $G(y):=y_{n}-g(y_1,y^{\prime\prime
})$, by the transversality condition we have that
\begin{align}
0  &  =\nabla F(y)\cdot\nabla G(y)=-\frac{\partial g}{\partial y_{1}%
}(y^{\prime})\label{transversal 1}\\
&  \quad+\sum_{k=2}^{n-1}\frac{\partial f}{\partial y_{k}}(y^{\prime\prime
},y_{n})\frac{\partial g}{\partial y_{k}}(y^{\prime})-\frac{\partial
f}{\partial y_{n}}(y^{\prime\prime},y_{n})\nonumber
\end{align}
for all $y\in\partial E\cap\partial\Omega\cap B(0,r_{x})$. For $|s|$
sufficiently small the set
\[
E_{s}:=\{\zeta \in\mathbb{R}^{n}:\,\operatorname*{d}\nolimits_{E}(\zeta)=s\}
\]
in a neighborhood of $x$ is given by an $n-1$-dimensional manifold
parametrized by
\begin{align}
\varphi_{i}(y^{\prime},s)  &  =y_{i}-s\frac{\partial g}{\partial y_{i}%
}(y^{\prime})\left(  1+\left| \nabla g(y^{\prime})\right| _{n-1}%
^{2}\right)  ^{-1/2}\label{chart1}\\
\varphi_{n}(y^{\prime},s)  &  =g(y^{\prime})+s\left(  1+\left| \nabla
g(y^{\prime})\right|_{n-1}^{2}\right)  ^{-1/2} \label{chart2}%
\end{align}
for all $i=1,\ldots,n-1$ and for $y^{\prime}\in Q_{n-1}(0,r_{x}^{\prime})$,
where $0<r_{x}^{\prime}<r_{x}$ and we are using local coordinates (see \cite{krantz1981distance}). Here $Q_{n-1}(0,r)$ denotes the cube $(-r/2,r/2)^{n-1}$ and $|\cdot|_{n-1}$ denotes the norm in $\mathbb{R}^{n-1}$.  We
are interested in in the surface area of $E_{s}$ contained in $\Omega$, or in other words, the area of the region characterized by
\begin{equation}
\varphi_{1}(y^{\prime},s)>f(\varphi_{2}(y^{\prime},s),\ldots,\varphi
_{n}(y^{\prime},s)). \label{inside Omega}%
\end{equation}
Consider the function%
\begin{equation}
H(y^{\prime},s):=\varphi_{1}(y^{\prime},s)-f(\varphi_{2}(y^{\prime}%
,s),\ldots,\varphi_{n}(y^{\prime},s)). \label{function H}%
\end{equation}
By \eqref{chart1} and \eqref{chart2} we have%
\[
\frac{\partial H}{\partial y_{1}}(y^{\prime},s)=\frac{\partial\varphi_{1}%
}{\partial y_{1}}(y^{\prime},s)-\sum_{k=2}^{n}\frac{\partial f}{\partial
y_{k}}(\varphi_{2}(y^{\prime},s),\ldots,\varphi_{n}(y^{\prime},s))\frac
{\partial\varphi_{k}}{\partial y_{1}}(y^{\prime},s).
\]
Taking $y^{\prime}=0$ and recalling that $\frac{\partial g}{\partial y_{i}%
}(0)=0$ for all $i=1,\ldots,n-1$ gives%
\begin{equation}
\frac{\partial H}{\partial y_{1}}(0,s)=1-s\frac{\partial^{2}g}{\partial
y_{1}^{2}}(0)+s\sum_{k=2}^{n-1}\frac{\partial f}{\partial y_{k}}%
(0)\frac{\partial^{2}g}{\partial y_{k}^{2}}(0)>0 \label{d100}%
\end{equation}
provided we take
\[
1>s(1+\left\Vert \nabla f\right\Vert _{L^{\infty}(Q_{n-1}(0,r_{x}^{\prime}%
))})\left\Vert \nabla^{2}g\right\Vert _{L^{\infty}(Q_{n-1}(0,r_{x}^{\prime}%
))}.
\]
By taking $r_{x}^{\prime}>0$ smaller, it follows by the implicit function
theorem that there exists a function $h(y^{\prime\prime},s)$ of class
$C^{1,\alpha}$ such that the condition \eqref{inside Omega} is equivalent to%
\begin{equation}\label{Eqn:ImplicitFunctionArg}
y_{1}>h(y^{\prime\prime},s)
\end{equation}
for all for $y^{\prime}=(y_{1},y^{\prime\prime})\in Q_{n-1}(0,r_{x}^{\prime})$
and all $s>0$ sufficiently small. Moreover, since $f$ is of class
$C^{2,\alpha}$ and the functions $\varphi_{i}$ and $\frac{\partial\varphi_{i}%
}{\partial y_{k}}$ are infinitely differentiable in the variable $s$ (by
\eqref{chart1} and \eqref{chart2}), we have that $\frac{\partial^{2}%
h}{\partial s^{2}}$ and $\frac{\partial^{2}h}{\partial s\partial y_{i}}$ exist
and are continuous. By \eqref{transversal 1}, \eqref{chart1}, \eqref{chart2},
and \eqref{function H},%
\begin{align*}
\frac{\partial H}{\partial s}(y^{\prime},0)  &  =\frac{\partial\varphi_{1}%
}{\partial s}(y^{\prime},0)-\sum_{k=2}^{n}\frac{\partial f}{\partial y_{k}%
}(\varphi_{2}(y^{\prime},0),\ldots,\varphi_{n}(y^{\prime},0))\frac
{\partial\varphi_{k}}{\partial s}(y^{\prime},0)\\
&  =-\left(  1+\left| \nabla g(y^{\prime})\right| _{n-1}^{2}\right)
^{-1/2}\frac{\partial g}{\partial y_{1}}(y^{\prime})\\
&  \quad+\left(  1+\left| \nabla g(y^{\prime})\right| _{n-1}%
^{2}\right)  ^{-1/2}\sum_{k=2}^{n-1}\frac{\partial f}{\partial y_{k}%
}(y^{\prime\prime},g(y^{\prime}))\frac{\partial g}{\partial y_{k}}(y^{\prime
})\\
&  \quad-\left(  1+\left| \nabla g(y^{\prime})\right|_{n-1}%
^{2}\right)  ^{-1/2}\frac{\partial f}{\partial y_{n}}(y^{\prime\prime
},g(y^{\prime}))=0.
\end{align*}
It follows by the implicit function theorem that
\begin{equation}
\frac{\partial h}{\partial s}(y^{\prime\prime},0)=-\frac{\partial H}{\partial
s}(h(y^{\prime\prime},0),y^{\prime\prime},0)\left(  \frac{\partial H}{\partial
y_1}(h(y^{\prime\prime},0),y^{\prime\prime},0)\right)  ^{-1}=0. \label{Dh Ds=0}%
\end{equation}

By \eqref{Eqn:ImplicitFunctionArg}, in a neighborhood of $x$ the surface area of $E_{s}$ inside $\Omega$ is given
by the surface integral
\[
A_{x}(s):=\int_{Q_{n-2}(0,r_{x}^{\prime})}\int_{h(y^{\prime\prime},s)}%
^{r_{x}^{\prime}/2}\sqrt{\sum_{\alpha\in\Xi}\left[  \det\frac
{\partial\left(  \varphi_{\alpha_{1}},\ldots,\varphi_{\alpha_{n-1}}\right)
}{\partial\left(  y_{1},\ldots,y_{n-1}\right)  }\left(  y^{\prime},s\right)
\right]  ^{2}}\,dy_{1}dy^{\prime\prime},
\]
where
\[
\Xi:=\left\{  \alpha\in\mathbb{N}^{n-1}:\,1\leq\alpha_{1}<\alpha
_{2}<\cdots<\alpha_{n-1}\leq n\right\}  .
\]
By standard theorems of differentiation under the integral sign, we have that
$A_{x}$ is of class $C^{2}$. 

Moreover, by \eqref{Dh Ds=0},
\begin{align*}
A_{x}^{\prime}(0) &  =\int_{Q_{n-1}(0,r_{x}^{\prime})}\frac{\partial}{\partial
s}\left(  \sqrt{\sum_{\alpha\in\Xi}\left[  \det\frac{\partial\left(
\varphi_{\alpha_{1}},\ldots,\varphi_{\alpha_{n-1}}\right)  }{\partial\left(
y_{1},\ldots,y_{n-1}\right)  }\left(  y^{\prime},s\right)  \right]  ^{2}%
}\right)  _{s=0}\,dy^{\prime}\\
&  \quad-\frac{\partial h}{\partial s}(y^{\prime\prime},0)\int_{Q_{n-2}%
(0,r_{x}^{\prime})}\sqrt{\sum_{\alpha\in\Xi}\left[  \det\frac
{\partial\left(  \varphi_{\alpha_{1}},\ldots,\varphi_{\alpha_{n-1}}\right)
}{\partial\left(  y_{1},\ldots,y_{n-1}\right)  }\left(  h(y^{\prime\prime
},0,),y^{\prime\prime},0\right)  \right]  ^{2}}\,dy^{\prime\prime}\\
&  =\int_{Q_{n-1}(0,r_{x}^{\prime})}\frac{\partial}{\partial s}\left(
\sqrt{\sum_{\alpha\in\Xi}\left[  \det\frac{\partial\left(  \varphi
_{\alpha_{1}},\ldots,\varphi_{\alpha_{n-1}}\right)  }{\partial\left(
y_{1},\ldots,y_{n-1}\right)  }\left(  y^{\prime},s\right)  \right]  ^{2}%
}\right)  _{s=0}\,dy^{\prime}.
\end{align*}

As the boundary term has dropped out we can then obtain \eqref{levelSetDerivative} using a partition of unity and classical formulas (see, e.g., \cite{MaggiBook}).

The fact that $\eta$ is bounded follows from \cite{Oleksiv}. This completes the proof.

\end{proof}

We then prove our $\Gamma$-$\limsup$ inequality.

\begin{proof}[Proof of Theorem \ref{thm:nDLimsup}]
If $u \notin \mathcal{U}_1$ the inequality is trivial. Thus, let $u \in \mathcal{U}_1$. Then we have that $u$ is of the form $a\chi_{E} + b\chi_{E^c}$. Define
\begin{equation} \label{etaChoice}
\eta(t) := \mathcal{H}^{n-1}(\{x: d_E(x) = t\}).
\end{equation}
By Lemma \ref{levelSetLemma} we have that $\eta$ satisfies the assumptions of Theorem \ref{1DLimsupOrder2}. Let $v_\e$ be the one-dimensional function constructed in Theorem \ref{1DLimsupOrder2}, using $\eta$ chosen via \eqref{etaChoice}.   Define $u_\e(x) := v_\e(d_{E}(x))$. By the coarea formula for Lipschitz functions we have that
\[
F_\e^{(2)}(u_\e) =\frac1\e\left( \int_\mathbb{R} (\e^{-1}W(v_\e(t)) + \e (v_\e')^2)  \mathcal{H}^{n-1}(\{x: d_E(x) = t\}) \dt -2c_W\eta(0)\right).
\]
Applying Theorem \ref{1DLimsupOrder2} then gives the desired result.
\end{proof}

\begin{remark}
We note that our recovery sequence construction does not require any smoothness properties on $\I$, such as \eqref{isoFunctionSmooth}. The technique given in this section thus establishes the $\Gamma$-$\limsup$ inequality at any mass level. Although we do not currently have a liminf inequality to support it, we suspect that the $\Gamma$-limit we establish in this paper should hold at all mass levels. When $W$ is symmetric with respect to $\frac{a+b}{2}$ this would imply that surfaces with higher magnitude mean curvature are (slightly) energetically favored.
\end{remark}

Next we prove Corollary \ref{cor:symmetricMainResult}.

\begin{proof}[Proof of Corollary \ref{cor:symmetricMainResult}.]
If $W$ is symmetric about $\frac{a+b}{2}$ then $W \circ z$ will be an even function. Thus (see \eqref{c1Definition})
\[
c_{\operatorname*{sym}} = \int_\mathbb{R} W(z(t)) t  \dt = 0 .
\]
Furthermore if $W$ is symmetric then $z$ will satisfy $b-z(t) = z(-t) - a$ for all $t > 0$. In turn, this implies that
\[
\int_\mathbb{R} \operatorname*{sgn}\nolimits_{a,b}(t) - z(t) \dt = 0,
\]
and hence
\[
\int_\mathbb{R} z(t-\tau_u) - \operatorname*{sgn}\nolimits_{a,b}(t) \dt = \int_\mathbb{R} z(t-\tau_u) - z(t)  \dt.
\]
Utilizing the fundamental theorem of calculus and Fubini's theorem
\begin{align}
\int_\mathbb{R} z(t-\tau_u) - z(t)  \dt &= -\int_\mathbb{R} \int_{t-\tau_u}^t z'(s) \ds \dt \\
&= -\int_{\mathbb{R}} \left(\int_{s}^{s+\tau_u}\dt \right)z'(s)\ds = -\tau_u \int_{\mathbb{R}} z'(s) \ds =  -\tau_u(b-a).
\end{align}
Recalling \eqref{deltaUDefinition}, \eqref{def:lambdaNot} and \eqref{deltaUDefinitionCase2}we then find that
\[
\tau_u = \begin{cases}
-\dfrac{\Lambda_u}{W''(a)\operatorname*{P}(\{u=a\};\Omega) (b-a)} &\text{ if } q =1, \\
0 & \text{ if } q<1.
\end{cases}
\]
This then gives
\[
 \mathcal{F}^{(2)}(u) = \begin{cases}
\frac1{W''(a)} \left(\frac{\Lambda_u^2}{2} - 2\kappa_u(n-1)c_W\frac{\Lambda_u}{ (b-a)}  \right)&\text{ if } q = 1, \\
0 &\text{ if } q<1.
\end{cases}
\]
Again recalling \eqref{def:lambdaNot} we find that
\[
 \mathcal{F}^{(2)}(u) = \begin{cases}
-\frac{\Lambda_u^2}{2W''(a)} &\text{ if } q =1, \\
0 &\text{ if } q<1,
\end{cases}
\]
as desired.
\end{proof}

\section*{Acknowledgments}
This paper is part of the second author's Ph.D. thesis, Carnegie Mellon University. 
The authors warmly thank the Center for Nonlinear Analysis. The authors would like to thank Bob Pego for his helpful insights. R. Murray also wishes to acknowledge the hospitality and helpful discussions of Gianni Dal Maso at SISSA, Trieste, Italy, where part of this work was carried out. The authors would also like to thank the anonymous referee for carefully reading the manuscript and making many very valuable suggestions that improved the presentation.

\section*{Compliance with Ethical Standards}
Part of this research was carried out at Center for Nonlinear Analysis. The center is partially
supported by NSF Grant No. DMS-0635983 and NSF PIRE Grant No. OISE-0967140. 
The research of G. Leoni was partially funded by the NSF under Grant Nos.  DMS-1007989 and DMS-1412095 and the one of R. Murray by NSF PIRE Grant No. OISE-0967140.

\bibliography{ModicaMortolaRefs}
\bibliographystyle{acm}

\end{document}